\documentclass[dvipdfmx]{article}
\usepackage{amsmath,graphicx,amssymb,amsthm,bm,comment,color,soul}
\makeatletter

\makeatletter

\@addtoreset{equation}{section}
\makeatother

\usepackage{geometry}
\numberwithin{equation}{section}
\newtheorem{thm}{Theorem}[section]
\newtheorem{cor}[thm]{Corollary}
\newtheorem{lem}[thm]{Lemma}
\newtheorem{prop}[thm]{Proposition}
\newtheorem{rem}{Remark}
\newtheorem{ass}{Assumption}

\title{Limit theorem associated with Wishart matrices with application to hypothesis testing for common principal components}
\author{Koji Tsukuda\footnote{Faculty of Mathematics, Kyushu University, 744 Motooka, Nishi-ku, Fukuoka-shi, Fukuoka 819-0395, Japan.} \quad {\it and} \quad Shun Matsuura\footnote{Faculty of Science and Technology, Keio University, 3-14-1 Hiyoshi, Kohoku-ku, Yokohama, Kanagawa 223-8522, Japan.}}

\begin{document}
\maketitle

\allowdisplaybreaks[3]

\begin{abstract}
This study derives a new property of the Wishart distribution when the degree-of-freedom and the size of the matrix parameter of the distribution grow simultaneoulsy.
Particularly, the asymptotic normality of 
the product of four independent Wishart matrices
is shown under 
a high dimensional asymptotic regime.
As an application of the result, a statistical test procedure for the common principal components hypothesis is proposed.
For this problem, the proposed test statistic is asymptotically normal under the null hypothesis.
In addition, the proposed test statistic diverges to positive infinity in probability under the alternative hypothesis.
\end{abstract}

\vspace{5truemm}
\textbf{keywords}: asymptotic test; central limit theorem;  common principal components model; high-dimension;  Wishart distribution.\par
\textbf{MSC2020 subject classifications}: primary 60F05, 62F05; secondary 62H25.

\section{Introduction}\label{sec:1}

This paper shows the asymptotic normality of the trace of products of four independent Wishart matrices in a high-dimensional setting, and proposes a statistical test procedure for the common principal components (CPC) hypothesis which is a typical null hypothesis in the context of multivariate statistical analysis.

A classical setting in multivariate analysis is that population distributions are normal and that the number of observed variables are much less than the number of individuals in a sample.
In 1928, John Wishart derived the celebrated Wishart distribution as the distribution of a scatter matrix $\sum_{i=1}^n \bm{v}_i \bm{v}_i'$ calculated from iid centered $p$-dimensional normal vectors $\{ \bm{v}_i \}_{i=1}^n$, where $'$ denotes the transpose.
Starting from the derivation of the Wishart distribution, a lot of studies have investigated its asymptotic properties under the traditional multivariate analysis setting: $n\to\infty$ with fixed $p$.
On the other hand, as observed variables have increased with the development of information technology, different settings have become possible and so multivariate statistical methods have been developed to deal with this situation.
In particular, when few variables are observed, the likelihood ratio test is quite useful to test hypotheses about population covariance matrices.
However, when more variables are observed than the number of the individuals in samples (so-called \textit{high-dimensional setting}), the likelihood ratio test is unavailable in many cases because the scatter matrices are not full-rank.
In one-sample testing problems such as ``the population covariance matrix is an identity matrix'', ``the covariance population matrix is spherical'', and ``the covariance matrix is diagonal'', alternative test procedures with clever usage of the trace of some functions of a scatter matrix have been proposed.
Such procedures are considered to be effective in high-dimensional settings; see, e.g., Chen et al.~\cite{RefCZZ}, Srivastava~\cite{RefSr}, and Srivastava et al.~\cite{RefSYK}.

In two-sample testing problems for covariance matrices, the hypotheses such as ``two covariance population matrices are identical'', ``two population covariance matrices are proportional'', and ``two population covariance matrices have the same eigenvectors (CPC hypothesis)'' have been considered.
These three hypotheses testing are especially typical in two sample problems in the multivariate analysis.
Indeed, they correspond Flury's hierarchical model, 
and the likelihood ratio test can be used for model selection \cite{RefF88}.
As the likelihood ratio test for these hypotheses testing is unavailable in a high-dimensional setting, alternative test procedures have been proposed for the former two hypotheses (equality and proportionality); see, e.g., Li and Chen~\cite{RefLC}, Liu et al.~\cite{RefLXZT}, Schott~\cite{RefSc}, Srivastava and Yanagihara~\cite{RefSY}, Srivastava et al.~\cite{RefSYK}, Tsukuda and Matsuura~\cite{RefTM} and Xu et al.~\cite{RefXLZB}.
Particularly, some of them adopt test statistics based on the trace of some functions of scatter matrices.
Testing the CPC hypothesis was first considered by Flury~\cite{RefF84}, and several studies including Boente et al.~\cite{RefBPR09}, Boik~\cite{RefB} and Hallin et al.~\cite{RefHPV10,RefHPV13} have proposed test procedures.
However, none of them have considered high-dimensional settings.
Therefore, in this paper, we propose a test procedure for CPC hypothesis in a high-dimensional setting by applying our main result.

This paper is organized as follows.
In Section~\ref{sec:2}, we present the main result with the outline of its proof.
An application of the limit theorem to testing CPC hypothesis is given in Section~\ref{sec:3}.
Section~\ref{sec:4} is devoted to prepare preliminary results which are used in the proof of the main result.
Section~\ref{sec:proof} supplements the technically missing part in the former sections and concludes the proof of the main result.

\section{Limit theorem}\label{sec:2}

\subsection{Problem setting and assumption}

Let $n_a$, $n_b$, $n_c$, $n_d$ and $p$ be positive integers and $\bm{\Sigma}_a$, $\bm{\Sigma}_b$, $\bm{\Sigma}_c$ and $\bm{\Sigma}_d$ positive definite matrices.
Consider four independent Wishart matrices 
\[ 
\bm{T}_{a}(n_a) \sim W_p(n_a,\bm{\Sigma}_a), \ 
\bm{T}_{b}(n_b) \sim W_p(n_b,\bm{\Sigma}_b), \]
\[ 
\bm{T}_{c}(n_c) \sim W_p(n_c,\bm{\Sigma}_c), \ 
\bm{T}_{d}(n_d) \sim W_p(n_d,\bm{\Sigma}_d), \]
where $\bm{T} \sim W_p(n,\bm{\Sigma})$ denotes a random $p\times p$ matrix $\bm{T}$ follows the $p$-dimensional Wishart distribution with its degrees-of-freedom $n$ and its matrix parameter $\bm{\Sigma}$.
We will study the asymptotic behavior of 
\[ M = \frac{1}{r_p}{\rm tr}\left( 
\bm{T}_{a}(n_a) \bm{T}_{b}(n_b) \bm{T}_{c}(n_c) \bm{T}_{d}(n_d) \right) \]
under the following high-dimensional asymptotic regime
\begin{equation}\label{ar1}
n_a,n_b,n_c,n_d \asymp p^{\delta}, \quad 0<\delta < 1, 
\end{equation}
where
\[ r_p =  r_{p,n_a,n_b,n_c,n_d} =  p^2 \sqrt{n_a n_b n_c n_d}. \]
To provide our limit theorem, we pose the following assumption.

\begin{ass}\label{ass1}
As $p\to \infty$ with \eqref{ar1}, it holds that
\begin{eqnarray*}
&& \frac{{\rm tr}(\bm{\Sigma}_i \bm{\Sigma}_j)}{p} \to \sigma_{ij}\in (0,\infty),  \\
&& \frac{{\rm tr}(\bm{\Sigma}_i \bm{\Sigma}_j \bm{\Sigma}_k)}{p} \to \sigma_{ijk}\in (-\infty,\infty), \\
&& \qquad \vdots \\
&& \frac{{\rm tr}(\bm{\Sigma}_i \bm{\Sigma}_j \bm{\Sigma}_k \bm{\Sigma}_l 
\bm{\Sigma}_{i'} \bm{\Sigma}_{j'} \bm{\Sigma}_{k'} \bm{\Sigma}_{l'}
\bm{\Sigma}_{i''} \bm{\Sigma}_{j''} \bm{\Sigma}_{k''} \bm{\Sigma}_{l''}
\bm{\Sigma}_{i'''} \bm{\Sigma}_{j'''} \bm{\Sigma}_{k'''} \bm{\Sigma}_{l'''})}{p} \\
&& \to \sigma_{ijkli'j'k'l'i''j''k''l''i'''j'''k'''l'''}\in (-\infty,\infty)  
\end{eqnarray*}
for $i,j,k,l,i',j',k',l',i'',j'',k'',l'',i''',j''',k''',l'''=a,b,c,d$.
\end{ass}

Obviously, it holds that 
\[ {\rm E}[M]=\frac{{n_an_bn_cn_d}}{r_p}{\rm tr}\left( 
\bm{\Sigma}_{a} \bm{\Sigma}_{b} \bm{\Sigma}_{c} \bm{\Sigma}_{d} \right). \]
Moreover, the following proposition provides the limit of variance of $M$ under our asymptotic regime.

\begin{prop}\label{prop21}
Under Assumption~\ref{ass1}, it holds that
\[ {\rm V}[M] \to \sigma_{ab}\sigma_{ad}\sigma_{bc}\sigma_{cd} \]
as $p\to\infty$ with \eqref{ar1}.
\end{prop}

\textit{Proof.}
See Section~\ref{sec:proof}.
\qed

To close this subsection, let us define four independent iid $p$-dimensional random sequences $\{\bm{x}_i\}_{i=1}^{n_a}$, $\{ \bm{y}_i \}_{i=1}^{n_b}$, $\{ \bm{z}_i\}_{i=1}^{n_c}$ and $\{ \bm{w}_i \}_{i=1}^{n_d}$ which satisfy
\[ 
\bm{T}_{a}(n_a) = \sum_{i=1}^{n_a}\bm{x}_i\bm{x}_i' , \ 
\bm{T}_{b}(n_b) = \sum_{i=1}^{n_b}\bm{y}_i\bm{y}_i' , \   
\bm{T}_{c}(n_c) = \sum_{i=1}^{n_c}\bm{z}_i\bm{z}_i' , \ 
\bm{T}_{d}(n_d) = \sum_{i=1}^{n_d}\bm{w}_i\bm{w}_i' 
\]
where
\[ \bm{x}_i\sim N_p(\bm{0}_p,\bm{\Sigma}_a) \ (i=1,\dots,n_a), \ 
\bm{y}_i\sim N_p(\bm{0}_p,\bm{\Sigma}_b) \ (i=1,\dots,n_b), \]
\[ \bm{z}_i\sim N_p(\bm{0}_p,\bm{\Sigma}_c) \ (i=1,\dots,n_c), \ 
\bm{w}_i\sim N_p(\bm{0}_p,\bm{\Sigma}_d) \ (i=1,\dots,n_d). \]
Moreover, for later discussions, let us denote
\[ \bm{T}_{a}(h) = \sum_{i=1}^{h}\bm{x}_i\bm{x}_i' \ (h=1,\dots,n_a), \ 
\bm{T}_{b}(h) = \sum_{i=1}^{h}\bm{y}_i\bm{y}_i' \ (h=1,\dots,n_b), \]
\[ \bm{T}_{c}(h) = \sum_{i=1}^{h}\bm{z}_i\bm{z}_i' \ (h=1,\dots,n_c), \ 
\bm{T}_{d}(h) = \sum_{i=1}^{h}\bm{w}_i\bm{w}_i' \ (h=1,\dots,n_d). \]

\subsection{Main result}

The main result of this paper is the following theorem.

\begin{thm}\label{mthm}
Under Assumption~\ref{ass1}, it holds that
\[ M - {\rm E}[M] \Rightarrow N(0,  \sigma_{ab}\sigma_{ad}\sigma_{bc}\sigma_{cd})\]
as $p\to\infty$ with \eqref{ar1}, where $\Rightarrow$ denotes the convergence in distribution.
\end{thm}

\begin{proof}
Define a sequence $\{\bm{u}_i\}_{i=1}^{n_a+n_b+n_c+n_d}$ by
\[ \bm{u}_i=\bm{x}_i \ (i=1,\dots,n_a), \ 
\bm{u}_{n_a+i}=\bm{y}_i \ (i=1,\dots,n_b), \]
\[ \bm{u}_{n_a+n_b+i}=\bm{z}_i \ (i=1,\dots,n_c), \ 
\bm{u}_{n_a+n_b+n_c+i}=\bm{w}_i \ (i=1,\dots,n_d). \]
Moreover, introduce a filtration $\{\mathcal{F}_h\}_{h=1}^{n_a+n_b+n_c+n_d}$ defined by $\mathcal{F}_h=\sigma (\bm{u}_1,\dots,\bm{u}_h)$ $(h=1,\dots,n_a+n_b+n_c+n_d)$.
Consider a martingale difference array $\{ D_h \}_{h=1}^{n_a+n_b+n_c+n_d}$ defined by
\[ D_h = {\rm E}_h[M] - {\rm E}_{h-1}[M] \quad (h=1,\dots,n_a+n_b+n_c+n_d) , \]
where we use the notation ${\rm E}_0[\cdot] = {\rm E}[\cdot]$ and ${\rm E}_h[\cdot] = {\rm E}[\cdot| \mathcal{F}_h]$ $(h=1,\ldots,n_a{}+\cdots+n_{d})$ for simplicity.
From the definition, it holds that
\begin{eqnarray*}
 \sum_{h=1}^{n_a+n_b+n_c+n_d}D_h 
&=& M - {\rm E}[M]  \\
&=& 
\frac{1}{r_p}{\rm tr}\left( 
\bm{T}_{a}(n_a) \bm{T}_{b}(n_b) \bm{T}_{c}(n_c) \bm{T}_{d}(n_d) \right) 
-\frac{{n_an_bn_cn_d}}{r_p}{\rm tr}\left( 
\bm{\Sigma}_{a} \bm{\Sigma}_{b} \bm{\Sigma}_{c} \bm{\Sigma}_{d} \right). 
\end{eqnarray*}
Its quadratic variation is denoted by $\{ \sigma_h^2 \}_{h=1}^{n_a+n_b+n_c+n_d}$; i.e., 
\[ \sigma_{h}^{2} = {\rm E}_{h-1}[D_{h}^{2}] \quad (h=1,\dots,n_a+n_b+n_c+n_d) . \]
It holds that
\[ \sigma_{h}^{2} = {\rm E}_{h-1}[({\rm E}_h[M])^2]-({\rm E}_{h-1}[M])^2 \ (h=1,\dots,n_a+n_b+n_c+n_d) \]
and so
\begin{eqnarray}
{\rm E} \left[ \sum_{h=1}^{n_a+n_b+n_c+n_d}\sigma_{h}^{2} \right] 
&=& {\rm E} \left[ \sum_{h=1}^{n_a+n_b+n_c+n_d} \left\{ {\rm E}_{h-1}[({\rm E}_h[M])^2]-({\rm E}_{h-1}[M])^2 \right\} \right] \nonumber \\
&=& \sum_{h=1}^{n_a+n_b+n_c+n_d}\left\{ {\rm E}[({\rm E}_h[M])^2]-{\rm E}[({\rm E}_{h-1}[M])^2] \right\} \nonumber \\
&=& {\rm E}[({\rm E}_{n_a+n_b+n_c+n_d}[M])^2]-{\rm E}[({\rm E}_{0}[M])^2] \nonumber \\
&=& {\rm E}[M^2]-({\rm E}[M])^2={\rm V}[M] . \label{vme}
\end{eqnarray}

As we will see in Section~\ref{sec:proof}, the following two lemmas hold:

\begin{lem}\label{lem23}
Under Assumption~\ref{ass1}, it holds that
\[ {\rm V}\left[ \sum_{h=1}^{n_a+n_b+n_c+n_d}\sigma_{h}^{2} \right] \to 0 \]
as $p\to\infty$ with \eqref{ar1}.
\end{lem}
\begin{lem}\label{lem24}
Under Assumption~\ref{ass1}, it holds that
\[ \sum_{h=1}^{n_a+n_b+n_c+n_d} {\rm E}[D_{h}^{4}] \to 0 \]
as $p\to\infty$ with \eqref{ar1}.
\end{lem}

Proposition~\ref{prop21}, \eqref{vme} and these lemmas yield that
\[ \sum_{h=1}^{n_a+n_b+n_c+n_d}\sigma_{h}^{2} \to^p \sigma_{ab} \sigma_{ad} \sigma_{bc} \sigma_{cd} \]
and
\begin{equation}
 \frac{\sum_{h=1}^{n_a+n_b+n_c+n_d} {\rm E} [D_{h}^{4}]}{({\rm V}[M])^2} \to 0.  \label{lyap}
\end{equation}
Therefore, the conclusion follows from the martingale central limit theorem.
Indeed, the Lyapunov condition follows from \eqref{lyap}.
This completes the proof.
\end{proof}

\section{Testing for common principal components model}\label{sec:3}

\subsection{Problem setting}

The Common Principal Components (CPC) model is the model that the eigenvectors of the covariance matrices of (more than) two populations are identical.
The CPC model was first introduced in \cite{RefF84}, and fundamental asymptotic theory of statistical inference was established in \cite{RefF86}.
Flury~\cite{RefF84,RefF86} considered that the population distributions are normal.
Afterwards, the CPC model was discussed in Hallin et al.~\cite{RefHPV08} for other elliptical and possibly heterokurtic distributions than the normal distribution.
As introduced in Section~\ref{sec:1}, tests for CPC model have been studied in Boente et al.~\cite{RefBPR09}, Boik~\cite{RefB} and Hallin et al.~\cite{RefHPV10,RefHPV13}, but CPC test under the high-dimensional setting has not been studied in the literature.
In this section, we consider the problem by using Theorem~\ref{mthm}.

Denote spectral decompositions of two population covariance matrices $\bm{\Sigma}_x$ and $\bm{\Sigma}_y$ be 
$\bm{\Sigma}_x=\bm{U}_x\bm{\Lambda}_x\bm{U}_x'$ and  $\ \bm{\Sigma}_y=\bm{U}_y\bm{\Lambda}_y\bm{U}_y'$, respectively.
The CPC model means that there exist spectral decompositions satisfying $\bm{U}_x=\bm{U}_y$.
It is well-known that this model is equivalently expressed as
$\bm{\Sigma}_x \bm{\Sigma}_y  = \bm{\Sigma}_y \bm{\Sigma}_x$.

Henceforth, let $p, m, n$ be positive integers, and let $\bm{\Sigma}_x$ and $\bm{\Sigma}_y$ be $p\times p$ positive-definite matrices.
Suppose that we have a random sample of size $M=4m$ from $N_p(\bm{\mu}_x,\bm{\Sigma}_x)$, and the sample is randomly split  to four subsamples of size $m$.
In the same way, suppose that  we have a random sample of size $N=4n$ from $N_p(\bm{\mu}_y,\bm{\Sigma}_y)$, and the sample is split to four subsamples of size $n$.
Under this setting, we wish to test
\begin{eqnarray*}
 &\mathcal{H}_0 {\rm (Null)}&: \ \bm{\Sigma}_x \bm{\Sigma}_y  = \bm{\Sigma}_y \bm{\Sigma}_x, \\
 &\mathcal{H}_1 {\rm (Alternative)}&: \ \bm{\Sigma}_x \bm{\Sigma}_y  \neq \bm{\Sigma}_y \bm{\Sigma}_x .
\end{eqnarray*}
We consider the asymptotic regime $p\to\infty$ with
\begin{equation}\label{ar2}
m,n \asymp p^{\delta}, \quad 0<\delta < 1.
\end{equation}
When the power of the test is discussed, the regime is limited to $1/2 < \delta <1$ in order to guarantee the consistency.
The following assumption is posed on the covariane matrices.

\begin{ass}\label{ass2}
As $p\to\infty$ with \eqref{ar2}, it holds that
\begin{eqnarray*}
&& \frac{{\rm tr}(\bm{\Sigma}_i \bm{\Sigma}_j)}{p} \to \sigma_{ij}\in (0,\infty), \\
&& \qquad \vdots \\
&& \frac{{\rm tr}(\bm{\Sigma}_i \bm{\Sigma}_j \bm{\Sigma}_k \bm{\Sigma}_l \bm{\Sigma}_{i'})}{p} \to \sigma_{ijkli'}\in (0,\infty), \\
&& \frac{{\rm tr}(\bm{\Sigma}_i \bm{\Sigma}_j \bm{\Sigma}_k \bm{\Sigma}_l 
\bm{\Sigma}_{i'} \bm{\Sigma}_{j'} )}{p} \to \sigma_{ijkli'j'}\in (-\infty,\infty), \\
&& \qquad \vdots \\
&& \frac{{\rm tr}(\bm{\Sigma}_i \bm{\Sigma}_j \bm{\Sigma}_k \bm{\Sigma}_l 
\bm{\Sigma}_{i'} \bm{\Sigma}_{j'} \bm{\Sigma}_{k'} \bm{\Sigma}_{l'}
\bm{\Sigma}_{i''} \bm{\Sigma}_{j''} \bm{\Sigma}_{k''} \bm{\Sigma}_{l''}
\bm{\Sigma}_{i'''} \bm{\Sigma}_{j'''} \bm{\Sigma}_{k'''} \bm{\Sigma}_{l'''})}{p} \\
&&\to \sigma_{ijkli'j'k'l'i''j''k''l''i'''j'''k'''l'''}\in (-\infty,\infty)  
\end{eqnarray*}
for $i,j,k,l,i',j',k',l',i'',j'',k'',l'',i''',j''',k''',l'''=x,y$.
\end{ass}

\subsection{Test procedure}

We can equivalently transform $\mathcal{H}_0$ as follows:
\begin{eqnarray*}
&& \mathcal{H}_0: \ \bm{\Sigma}_x \bm{\Sigma}_y  = \bm{\Sigma}_y \bm{\Sigma}_x \\
& \Longleftrightarrow & \mathcal{H}_0: \ {\rm tr}\left\{ 
\left( \bm{\Sigma}_x \bm{\Sigma}_y - \bm{\Sigma}_y \bm{\Sigma}_x\right)
\left( \bm{\Sigma}_x \bm{\Sigma}_y - \bm{\Sigma}_y \bm{\Sigma}_x\right)' \right\} = 0 \\
& \Longleftrightarrow & \mathcal{H}_0: \ {\rm tr}\left\{ 
\left( \bm{\Sigma}_x \bm{\Sigma}_y - \bm{\Sigma}_y \bm{\Sigma}_x\right)
\left( \bm{\Sigma}_y \bm{\Sigma}_x - \bm{\Sigma}_x \bm{\Sigma}_y\right) \right\} = 0 \\
& \Longleftrightarrow & \mathcal{H}_0: \ {\rm tr}\left( 
\bm{\Sigma}_x \bm{\Sigma}_x \bm{\Sigma}_y \bm{\Sigma}_y \right) 
- {\rm tr}\left( 
\bm{\Sigma}_x \bm{\Sigma}_y \bm{\Sigma}_x \bm{\Sigma}_y \right) = 0.
\end{eqnarray*}
Hence, $\mathcal{H}_0$ can be equivalently transformed into
$ \mathcal{H}_0: \theta =  0$,
where
\[
\theta = \theta_p = \sigma_{xxyy}(p) - \sigma_{xyxy}(p), \]
\[
 \sigma_{xxyy}(p)=\frac{{\rm tr}\left( 
\bm{\Sigma}_x \bm{\Sigma}_x \bm{\Sigma}_y \bm{\Sigma}_y \right)}{p} , \
\sigma_{xyxy}(p)=\frac{{\rm tr}\left( 
\bm{\Sigma}_x \bm{\Sigma}_y \bm{\Sigma}_x \bm{\Sigma}_y \right)}{p}. \]
Moreover, $\mathcal{H}_1$ can be equivalently transformed into
$ \mathcal{H}_1: \theta > 0$. 

Let us denote the scatter matrices calculated from split subsamples by
\[ \bm{T}_{x1},\bm{T}_{x2},\bm{T}_{x3},\bm{T}_{x4},\bm{T}_{y1},\bm{T}_{y2},\bm{T}_{y3},\bm{T}_{y4}. \]
In this case, it holds that
\[\bm{T}_{xk} \sim W_p(m-1,\bm{\Sigma}_x) \ (k=1,2,3,4) \]
and
\[\bm{T}_{yk} \sim  W_p(n-1,\bm{\Sigma}_y) \ (k=1,2,3,4) . \]
Clearly,
\[ \hat{\theta}
= \frac{1}{(m-1)^2(n-1)^2p} \left\{ 
{\rm tr}\left( 
\bm{T}_{x1} \bm{T}_{x2} \bm{T}_{y1} \bm{T}_{y2} \right) 
- {\rm tr}\left( 
\bm{T}_{x3} \bm{T}_{y3} \bm{T}_{x4} \bm{T}_{y4} \right) \right\} \]
is an unbiased estimator of $\theta=\sigma_{xxyy}(p) - \sigma_{xyxy}(p)$.
As for the variance of $\frac{(m-1)(n-1)}{p}\hat{\theta}$, 
it follows from Proposition 2.1 that
\begin{eqnarray*}
\lefteqn{ {\rm V} \left[ \frac{(m-1)(n-1)}{p}\hat{\theta} \right] } \\
&=& {\rm V} \left[ \frac{1}{(m-1)(n-1)p^2} \left\{ 
{\rm tr}\left( 
\bm{T}_{x1} \bm{T}_{x2} \bm{T}_{y1} \bm{T}_{y2} \right) 
- {\rm tr}\left( 
\bm{T}_{x3} \bm{T}_{y3} \bm{T}_{x4} \bm{T}_{y4} \right) \right\} \right] \\
&=& {\rm V} \left[ \frac{1}{(m-1)(n-1)p^2} 
{\rm tr}\left( 
\bm{T}_{x1} \bm{T}_{x2} \bm{T}_{y1} \bm{T}_{y2} \right) \right] 
\\&&
+ {\rm V} \left[ \frac{1}{(m-1)(n-1)p^2} 
{\rm tr}\left( 
\bm{T}_{x3} \bm{T}_{y3} \bm{T}_{x4} \bm{T}_{y4} \right) \right] \\
& \to & \sigma_{xx}\sigma_{yy}\sigma_{xy}^{2} + \sigma_{xy}^{4}.
\end{eqnarray*}

The following proposition establishes the asymptotic behavior of $\hat\theta$ under our asymptotic regime.

\begin{prop}\label{prop31}
Under Assumption~\ref{ass2}, it holds that
\[ 
\frac{(m-1)(n-1)}{p} \left( \hat{\theta} - \theta \right) 
\Rightarrow 
N\left(0 , \sigma_{xx}\sigma_{yy}\sigma_{xy}^{2} + \sigma_{xy}^{4} \right) 
\]
as $p\to\infty$ with \eqref{ar2}.
\end{prop}

\begin{proof}
The left-hand side equals
\begin{eqnarray}
&& \frac{1}{(m-1) (n-1) p^2}
{\rm tr}\left( 
\bm{T}_{x1} \bm{T}_{x2} \bm{T}_{y1} \bm{T}_{y2} \right) 
- \frac{(m-1)(n-1)}{p}\sigma_{xxyy}(p) \nonumber \\
&& - \left\{ 
\frac{1}{(m-1) (n-1) p^2} {\rm tr}\left( 
\bm{T}_{x3} \bm{T}_{y3} \bm{T}_{x4} \bm{T}_{y4} \right) 
- \frac{(m-1)(n-1)}{p}\sigma_{xyxy}(p)  \right\} .  \label{pt31}
\end{eqnarray}
It follows from Theorem~\ref{mthm} that
\[
\frac{1}{(m-1) (n-1) p^2} {\rm tr}\left( \bm{T}_{x1} \bm{T}_{x2} \bm{T}_{y1} \bm{T}_{y2} \right) - \frac{(m-1)(n-1)}{p}\sigma_{xxyy} (p)
\Rightarrow N(0, \sigma_{xx}\sigma_{yy}\sigma_{xy}^{2} )
\]
and that
\[ 
\frac{1}{(m-1) (n-1) p^2} {\rm tr}\left( 
\bm{T}_{x3} \bm{T}_{y3} \bm{T}_{x4} \bm{T}_{y4} \right) 
- \frac{(m-1)(n-1)}{p}\sigma_{xyxy} (p)
\Rightarrow N(0, \sigma_{xy}^{4} ).
 \]
As the first and second terms of \eqref{pt31} are independent, the conclusion follows from the Slutsky theorem.
\end{proof}

When the null hypothesis $\mathcal{H}_0: \theta = 0$ is true, Proposition~\ref{prop31} implies that 
\[ \frac{(m-1)(n-1)}{p}\hat{\theta} 
\Rightarrow 
N(0,\sigma_{xx}\sigma_{yy}\sigma_{xy}^{2} + \sigma_{xy}^{4}). \]
Hence, constructing a consistent estimator of $\sigma_{xx}\sigma_{yy}\sigma_{xy}^{2} + \sigma_{xy}^{4}$ enables us to propose a test procedure.
Let us denote by $\bm{T}_{x \cdot}$ and $\bm{T}_{y \cdot}$ the Scatter matrices calculated from two samples before splitting.
Define
\begin{eqnarray*} 
\hat\sigma_{xx} &=& \frac{(M-1)^2}{p(M-2)(M+1)}\left[ {\rm tr}\left( \bm{T}_{x \cdot} \bm{T}_{x \cdot} \right) -\frac{ \{ {\rm tr} (\bm{T}_{x \cdot}) \}^2 }{M-1} \right] , \\
\hat\sigma_{yy} &=& \frac{(N-1)^2}{p(N-2)(N+1)}\left[ {\rm tr}\left( \bm{T}_{y \cdot} \bm{T}_{y \cdot} \right) -\frac{ \{ {\rm tr} (\bm{T}_{y \cdot}) \}^2 }{N-1} \right] , \\
\hat\sigma_{xy} &=& \frac{1}{p} {\rm tr}(\bm{T}_{x \cdot} \bm{T}_{y \cdot}) . 
\end{eqnarray*}

\begin{rem}
The estimators $\hat\sigma_{xx}$ and $\hat\sigma_{yy}$ are originally introduced by Bai and Saranadasa~\cite{RefBS}.
It is known that under Assumption~\ref{ass2}, it holds that
\begin{equation} \label{pr32}
 \hat\sigma_{xx} \to^p \sigma_{xx}, \
 \hat\sigma_{yy} \to^p \sigma_{yy}, \
 \hat\sigma_{xy} \to^p \sigma_{xy} 
\end{equation}
as $p\to\infty$ with \eqref{ar2}.
\end{rem}

Let us propose the test statistic $T$ defined by 
\[ T = \frac{(m-1)(n-1)}{p} \frac{\hat{\theta}}{ \sqrt{\hat\sigma_{xx} \hat\sigma_{yy} \hat\sigma_{xy}^{2} + \hat\sigma_{xy}^{4}} }, \]
and we propose the following test procedure (approximate significance level is $\alpha$):
\begin{itemize}
\item If $T> \Phi^{-1}(1-\alpha)$ then reject $\mathcal{H}_0$;
\end{itemize}
where $\Phi^{-1}(\cdot)$ is the quantile function of the standard normal distribution.
This test procedure is justified by the following corollaries.
Particularly, Corollary~\ref{cor33} guarantees the consistency of our test procedure when $1/2<\delta<1$.

\begin{cor}
Under Assumption~\ref{ass2}, when $\mathcal{H}_0$ is true, ${\rm P}(T > \Phi^{-1}(1-\alpha)) \to \alpha$ as $p\to\infty$ with \eqref{ar2}.
\end{cor}

\begin{proof}
When $\mathcal{H}_0$ is true, Proposition \ref{prop31} and \eqref{pr32} conjunction with the Slutsky theorem yield that $T \Rightarrow N(0,1)$ as $p\to\infty$ with \eqref{ar2}.
This completes the proof.
\end{proof}

\begin{cor}\label{cor33}
Under Assumption~\ref{ass2}, when $\mathcal{H}_1$ is true, if $ \sigma_{xxyy} - \sigma_{xyxy} > 0 $ then ${\rm P}(T>C) \to 1$ for any positive constant $C$ as $p\to\infty$ with 
$ m,n \asymp p^{\delta}$, $1/2<\delta < 1$.
\end{cor}

\begin{proof}
It holds that
\begin{eqnarray*}
T &=& 
\frac{(m-1)(n-1)}{p} \frac{\hat{\theta} - \theta }{ \sqrt{\hat\sigma_{xx} \hat\sigma_{yy} \hat\sigma_{xy}^{2} + \hat\sigma_{xy}^{4}} } 
+\frac{(m-1)(n-1)}{p} \frac{\sigma_{xxyy}(p) - \sigma_{xyxy}(p) }{ \sqrt{\hat\sigma_{xx} \hat\sigma_{yy} \hat\sigma_{xy}^{2} + \hat\sigma_{xy}^{4}} }\\
&\geq& \frac{(m-1)(n-1)}{p} \frac{ \hat{\theta} - \theta }{ \sqrt{\hat\sigma_{xx} \hat\sigma_{yy} \hat\sigma_{xy}^{2} + \hat\sigma_{xy}^{4}} } 
\\&& 
+\frac{(m-1)(n-1)}{p} \frac{\inf_{q\geq p} \{\sigma_{xxyy}(q) - \sigma_{xyxy}(q)\} }{ \sqrt{\hat\sigma_{xx} \hat\sigma_{yy} \hat\sigma_{xy}^{2} + \hat\sigma_{xy}^{4}} }
\end{eqnarray*}
The first term of the right-hand side is $O_P(1)$ by using Proposition~\ref{prop31}, \eqref{pr32} and the Slutsky theorem.
The second term tends to positive infinity because $(m-1)(n-1)/p \to \infty$ and $\inf_{q\geq p} \{\sigma_{xxyy}(q) - \sigma_{xyxy}(q)\} \to \sigma_{xxyy} - \sigma_{xyxy} > 0$ as $p\to\infty$.
This completes the proof.
\end{proof}

\begin{rem}
There is another natural unbiased estimator of $\theta$ other than $\hat\theta$.
For instance,
\begin{eqnarray*}
&& 
\frac{1}{(M-2)(M+1)(N-2)(N+1)}
\biggl[ 
{\rm tr}(\bm{T}_{x\cdot} \bm{T}_{x\cdot} \bm{T}_{y\cdot} \bm{T}_{y\cdot}) 
- \frac{1}{N-1} {\rm tr}(\bm{T}_{x\cdot} \bm{T}_{x\cdot} \bm{T}_{y\cdot}) {\rm tr}(\bm{T}_{y\cdot})  \\
&& - \frac{1}{M-1} {\rm tr}(\bm{T}_{x\cdot} \bm{T}_{y\cdot} \bm{T}_{y\cdot}) {\rm tr}(\bm{T}_{x\cdot})
+ \frac{1}{(M-1)(N-1)}  {\rm tr}(\bm{T}_{x\cdot} \bm{T}_{y\cdot}) {\rm tr}(\bm{T}_{x\cdot}) {\rm tr}(\bm{T}_{y\cdot}) \\
&&  -\frac{MN-M-N+3}{(M-1)(N-1)} {\rm tr}(\bm{T}_{x\cdot} \bm{T}_{y\cdot} \bm{T}_{x\cdot} \bm{T}_{y\cdot}) 
+ \frac{M+N-1}{(M-1)(N-1)} {\rm tr}(\bm{T}_{x\cdot} \bm{T}_{y\cdot}) {\rm tr}(\bm{T}_{x\cdot} \bm{T}_{y\cdot})
\biggr]
\end{eqnarray*}
is an unbiased estimator of $\theta$.
Deriving the asymptotic behavior of this quantity is a possible future direction.
\end{rem}

\section{Preliminary results}\label{sec:4}

\subsection{Results for quadratic form of standard normal vectors}

In this subsection, we provide some properties concerning quadratic form of standard normal vectors.

\begin{lem}\label{lem41}
If $\bm{x} = (x_1,\ldots,x_p) \sim N_p(\bm{0}_p,\bm{I}_p)$ then
\begin{eqnarray*}
&& {\rm E} \left[ (\bm{x}'\bm{A}\bm{x})^2 \right] 
= 2{\rm tr}(\bm{A}^2) + \left\{ {\rm tr}(\bm{A}) \right\}^2, \\
&& {\rm E} \left[ (\bm{x}'\bm{A}\bm{x})^3 \right] 
= 8{\rm tr}(\bm{A}^3) + 6{\rm tr}(\bm{A}^2){\rm tr}(\bm{A}) + \left( {\rm tr}(\bm{A}) \right)^3, 
\\
&& {\rm E}\left[ (\bm{x}'\bm{A}\bm{x})^4 \right] \\
&&= 48{\rm tr}(\bm{A}^4) 
+ 32{\rm tr}(\bm{A}^3){\rm tr}(\bm{A})
+ 12\left( {\rm tr}(\bm{A}^2) \right)^2
+ 12{\rm tr}(\bm{A}^2) \left( {\rm tr}(\bm{A}) \right)^2
+ \left( {\rm tr}(\bm{A}) \right)^4 
\end{eqnarray*}
for any $p\times p$ symmetric matrix $\bm{A}$.
\end{lem}

\textit{Proof.}
By using the spectral decomposition, $\bm{A}$ is denoted by $\bm{U}\bm{\Lambda}\bm{U}'$, 
where $\bm{U}$ is an orthogonal matrix and $\bm{\Lambda}={\rm diag}(\lambda_1,\ldots,\lambda_p)$.
It holds that
\begin{eqnarray*}
\lefteqn{{\rm E} \left[ (\bm{x}'\bm{A}\bm{x})^2 \right] 
= {\rm E} \left[ (\bm{x}'\bm{\Lambda}\bm{x})^2 \right] 
= {\rm E} \left[ \left( \sum_{i=1}^{p}\lambda_{i}x_{i}^{2} \right)^2 \right] 
= {\rm E} \left[ \sum_{i=1}^{p}\lambda_{i}^{2}x_{i}^{4} 
+ \sum_{i=1}^{p}\sum_{j\neq i}\lambda_{i}x_{i}^{2}\lambda_{j}x_{j}^{2}
 \right]} \\
&=& 3\sum_{i=1}^{p}\lambda_{i}^{2} 
+ \sum_{i=1}^{p}\sum_{j\neq i}\lambda_{i}\lambda_{j}
= 2\sum_{i=1}^{p}\lambda_{i}^{2} 
+ \left( \sum_{i=1}^{p}\lambda_i \right)^2
= 2{\rm tr}(\bm{A}^2)
+ \left( {\rm tr}(\bm{A}) \right)^2.
\end{eqnarray*}
Moreover, it holds that
\begin{eqnarray*}
\lefteqn{{\rm E} \left[ (\bm{x}'\bm{A}\bm{x})^3 \right] 
= {\rm E} \left[ (\bm{x}'\bm{\Lambda}\bm{x})^3 \right] 
= {\rm E} \left[ \left( \sum_{i=1}^{p}\lambda_{i}x_{i}^{2} \right)^3 \right] }
\\
&=& {\rm E} \left[ \sum_{i=1}^{p}\lambda_{i}^{3}x_{i}^{6} 
+ 3\sum_{i=1}^{p}\sum_{j\neq i}\lambda_{i}^{2}x_{i}^{4}\lambda_{j}x_{j}^{2}
+ \sum_{i=1}^{p}\sum_{j\neq i}\sum_{k\neq i,j}\lambda_{i}x_{i}^{2}\lambda_{j}x_{j}^{2}
\lambda_{k}x_{k}^{2} \right] \\
&=& 15\sum_{i=1}^{p}\lambda_{i}^{3} 
+ 9\sum_{i=1}^{p}\sum_{j\neq i}\lambda_{i}^{2}\lambda_{j}
+ \sum_{i=1}^{p}\sum_{j\neq i}\sum_{k\neq i,j}\lambda_{i}\lambda_{j}\lambda_{k} 
= 14\sum_{i=1}^{p}\lambda_{i}^{3} 
+ 6\sum_{i=1}^{p}\sum_{j\neq i}\lambda_{i}^{2}\lambda_{j}
+ \left( \sum_{i=1}^{p}\lambda_i \right)^3
\\
&=& 8\sum_{i=1}^{p}\lambda_{i}^{3} 
+ 6\left( \sum_{i=1}^{p}\lambda_{i}^{2} \right) \left( \sum_{i=1}^{p}\lambda_i \right)
+ \left( \sum_{i=1}^{p}\lambda_i \right)^3
= 8{\rm tr}(\bm{A}^3)
+ 6{\rm tr}(\bm{A}^2){\rm tr}(\bm{A})
+ \left( {\rm tr}(\bm{A}) \right)^3.
\end{eqnarray*}
Furthermore, it holds that
\begin{eqnarray*}
\lefteqn{{\rm E} \left[ (\bm{x}'\bm{A}\bm{x})^4 \right] 
= {\rm E} \left[ (\bm{x}'\bm{\Lambda}\bm{x})^4 \right] 
= {\rm E} \left[ \left( \sum_{i=1}^{p}\lambda_{i}x_{i}^{2} \right)^4 \right] }
\\
&=& {\rm E} \left[ \sum_{i=1}^{p}\lambda_{i}^{4}x_{i}^{8} 
+ 4\sum_{i=1}^{p}\sum_{j\neq i}\lambda_{i}^{3}x_{i}^{6}\lambda_{j}x_{j}^{2}
+ 3\sum_{i=1}^{p}\sum_{j\neq i}\lambda_{i}^{2}x_{i}^{4}\lambda_{j}^{2}x_{j}^{4}
+ 6\sum_{i=1}^{p}\sum_{j\neq i}\sum_{k\neq i,j}\lambda_{i}^{2}x_{i}^{4}\lambda_{j}x_{j}^{2}
\lambda_{k}x_{k}^{2} \right. \\
&& \left. + \sum_{i=1}^{p}\sum_{j\neq i}\sum_{k\neq i,j}\sum_{l\neq i,j,k}
\lambda_{i}x_{i}^{2}\lambda_{j}x_{j}^{2}\lambda_{k}x_{k}^{2}\lambda_{l}x_{l}^{2} 
\right] 
\\
&=& 105\sum_{i=1}^{p}\lambda_{i}^{4} 
+ 60\sum_{i=1}^{p}\sum_{j\neq i}\lambda_{i}^{3}\lambda_{j}
+ 27\sum_{i=1}^{p}\sum_{j\neq i}\lambda_{i}^{2}\lambda_{j}^{2}
+ 18\sum_{i=1}^{p}\sum_{j\neq i}\sum_{k\neq i,j}\lambda_{i}^{2}\lambda_{j}\lambda_{k} 
+ \sum_{i=1}^{p}\sum_{j\neq i}\sum_{k\neq i,j}\sum_{l\neq i,j,k}
\lambda_{i}\lambda_{j}\lambda_{k}\lambda_{l} 
\\
&=& 104\sum_{i=1}^{p}\lambda_{i}^{4} 
+ 56\sum_{i=1}^{p}\sum_{j\neq i}\lambda_{i}^{3}\lambda_{j}
+ 24\sum_{i=1}^{p}\sum_{j\neq i}\lambda_{i}^{2}\lambda_{j}^{2}
+ 12\sum_{i=1}^{p}\sum_{j\neq i}\sum_{k\neq i,j}\lambda_{i}^{2}\lambda_{j}\lambda_{k} 
+ \left( \sum_{i=1}^{p}\lambda_i \right)^4 
\\
&=& 92\sum_{i=1}^{p}\lambda_{i}^{4} 
+ 32\sum_{i=1}^{p}\sum_{j\neq i}\lambda_{i}^{3}\lambda_{j}
+ 12\sum_{i=1}^{p}\sum_{j\neq i}\lambda_{i}^{2}\lambda_{j}^{2}
+ 12\left( \sum_{i=1}^{p}\lambda_{i}^{2} \right) \left( \sum_{i=1}^{p}\lambda_i \right)^2
+ \left( \sum_{i=1}^{p}\lambda_i \right)^4 
\\
&=& 80\sum_{i=1}^{p}\lambda_{i}^{4} 
+ 32\sum_{i=1}^{p}\sum_{j\neq i}\lambda_{i}^{3}\lambda_{j}
+ 12\left( \sum_{i=1}^{p}\lambda_{i}^{2} \right)^2
+ 12\left( \sum_{i=1}^{p}\lambda_{i}^{2} \right) \left( \sum_{i=1}^{p}\lambda_i \right)^2
+ \left( \sum_{i=1}^{p}\lambda_i \right)^4 
\\
&=& 48\sum_{i=1}^{p}\lambda_{i}^{4} 
+ 32\left( \sum_{i=1}^{p}\lambda_{i}^{3} \right) \left( \sum_{i=1}^{p}\lambda_i \right)
+ 12\left( \sum_{i=1}^{p}\lambda_{i}^{2} \right)^2
+ 12\left( \sum_{i=1}^{p}\lambda_{i}^{2} \right) \left( \sum_{i=1}^{p}\lambda_i \right)^2
+ \left( \sum_{i=1}^{p}\lambda_i \right)^4 
\\
&=& 48{\rm tr}(\bm{\Lambda}^4) 
+ 32{\rm tr}(\bm{\Lambda}^3){\rm tr}(\bm{\Lambda})
+ 12\left( {\rm tr}(\bm{\Lambda}^2) \right)^2
+ 12{\rm tr}(\bm{\Lambda}^2) \left( {\rm tr}(\bm{\Lambda}) \right)^2
+ \left( {\rm tr}(\bm{\Lambda}) \right)^4 
\\
&=& 48{\rm tr}(\bm{A}^4) 
+ 32{\rm tr}(\bm{A}^3){\rm tr}(\bm{A})
+ 12\left( {\rm tr}(\bm{A}^2) \right)^2
+ 12{\rm tr}(\bm{A}^2) \left( {\rm tr}(\bm{A}) \right)^2
+ \left( {\rm tr}(\bm{A}) \right)^4. 
\end{eqnarray*}
This completes the proof.
\qed

By using Lemma~\ref{lem41}, we can show the following assertions which will be used in Section~\ref{sec:proof}.

\begin{lem}\label{lem42}
If $\bm{x}\sim N_p(\bm{0}_p,\bm{I}_p)$ then
\[
 {\rm E} \left[ \left\{ \bm{x}'\bm{A}\bm{x}-{\rm tr}(\bm{A}) \right\}^2 \right] 
= {\rm tr}(\bm{A}\bm{A}) + {\rm tr}(\bm{A}\bm{A}') 
\]
for any $p\times p$ matrix $\bm{A}$.
\end{lem}

\textit{Proof}.
First, for symmetric $\bm{A}$, it follows from Lemma 4.1 that
\[ {\rm E} \left[ \left\{ \bm{x}'\bm{A}\bm{x}-{\rm tr}(\bm{A}) \right\}^2 \right] 
= {\rm E} \left[ \left( \bm{x}'\bm{A}\bm{x} \right)^2 
-2 \bm{x}'\bm{A}\bm{x} {\rm tr}(\bm{A}) 
+\left\{ {\rm tr}(\bm{A}) \right\}^2 
\right] 
= 2{\rm tr}(\bm{A}^2) + \left\{ {\rm tr}(\bm{A}) \right\}^2 -\left\{ {\rm tr}(\bm{A}) \right\}^2 
= 
2{\rm tr}(\bm{A}^2). \]
Hence, for possibly asymmetric $\bm{A}$, it holds that
\begin{eqnarray*}
\lefteqn{ {\rm E} \left[ \left\{ \bm{x}'\bm{A}\bm{x}-{\rm tr}(\bm{A}) \right\}^2 \right] }\\
&=& {\rm E} \left[ \left\{ \bm{x}'\left(\frac{\bm{A}+\bm{A}'}{2}\right)\bm{x}-{\rm tr}\left(\frac{\bm{A}+\bm{A}'}{2}\right) 
\right\}^2 \right]
= 2{\rm tr}\left( \left(\frac{\bm{A}+\bm{A}'}{2}\right)^2 \right) 
= \frac{1}{2}{\rm tr}\left( \left( \bm{A}+\bm{A}' \right)^2 \right) 
 \\
&=& \frac{1}{2}\left\{ {\rm tr}(\bm{A}\bm{A}) + {\rm tr}(\bm{A}\bm{A}') 
+ {\rm tr}(\bm{A}'\bm{A}) + {\rm tr}(\bm{A}'\bm{A}') \right\} 
= {\rm tr}(\bm{A}\bm{A}) + {\rm tr}(\bm{A}\bm{A}') .
\end{eqnarray*}
This completes the proof.
\qed

\begin{lem}\label{lem43}
If $\bm{x}\sim N_p(\bm{0}_p,\bm{I}_p)$ then
\[
{\rm E} \left[ \left( \bm{x}'\bm{A}\bm{x}-{\rm tr}(\bm{A}) \right)^4 \right] 
= 3{\rm tr}\left( \left( \bm{A}+\bm{A}' \right)^4 \right) 
+ \frac{3}{4}\left\{ {\rm tr}\left( \left( \bm{A}+\bm{A}' \right)^2 \right) \right\}^2
\]
for any $p\times p$ matrix $\bm{A}$.
\end{lem}

\textit{Proof}.
First, for symmetric $\bm{A}$, it follows from Lemma 4.1 that
\begin{eqnarray*}
\lefteqn{ {\rm E} \left[ \left( \bm{x}'\bm{A}\bm{x}-{\rm tr}(\bm{A}) \right)^4 \right]} \\
&=& {\rm E} \left[ \left( \bm{x}'\bm{A}\bm{x} \right)^4 
-4 \left( \bm{x}'\bm{A}\bm{x} \right)^3 {\rm tr}(\bm{A}) 
+6\left( \bm{x}'\bm{A}\bm{x} \right)^2 \left( {\rm tr}(\bm{A}) \right)^2
-4\bm{x}'\bm{A}\bm{x} \left( {\rm tr}(\bm{A}) \right)^3
+\left( {\rm tr}(\bm{A}) \right)^4 
\right] 
\\
&=& 
\left\{ 48{\rm tr}(\bm{A}^4) 
+ 32{\rm tr}(\bm{A}^3){\rm tr}(\bm{A})
+ 12\left( {\rm tr}(\bm{A}^2) \right)^2
+ 12{\rm tr}(\bm{A}^2) \left( {\rm tr}(\bm{A}) \right)^2
+ \left( {\rm tr}(\bm{A}) \right)^4 \right\} \\
&& -4 \left\{ 
8{\rm tr}(\bm{A}^3)
+ 6{\rm tr}(\bm{A}^2){\rm tr}(\bm{A})
+ \left( {\rm tr}(\bm{A}) \right)^3 \right\} {\rm tr}(\bm{A}) 
+6\left\{ 2{\rm tr}(\bm{A}^2)
+ \left( {\rm tr}(\bm{A}) \right)^2 \right\}  \left( {\rm tr}(\bm{A}) \right)^2 \\
&& -4{\rm tr}(\bm{A}) \left( {\rm tr}(\bm{A}) \right)^3
+\left( {\rm tr}(\bm{A}) \right)^4 
\\
&=& 
\left\{ 48{\rm tr}(\bm{A}^4) 
+ 32{\rm tr}(\bm{A}^3){\rm tr}(\bm{A})
+ 12\left( {\rm tr}(\bm{A}^2) \right)^2
+ 12{\rm tr}(\bm{A}^2) \left( {\rm tr}(\bm{A}) \right)^2
+ \left( {\rm tr}(\bm{A}) \right)^4 \right\} \\
&& - \left\{ 
32{\rm tr}(\bm{A}^3)
+ 24{\rm tr}(\bm{A}^2){\rm tr}(\bm{A})
+ 4\left( {\rm tr}(\bm{A}) \right)^3 \right\} {\rm tr}(\bm{A}) 
+\left\{ 12{\rm tr}(\bm{A}^2)
+ 6\left( {\rm tr}(\bm{A}) \right)^2 \right\}  \left( {\rm tr}(\bm{A}) \right)^2 \\
&& -4\left( {\rm tr}(\bm{A}) \right)^4
+\left( {\rm tr}(\bm{A}) \right)^4 
\\
&=& 48{\rm tr}(\bm{A}^4) 
+ 12\left( {\rm tr}(\bm{A}^2) \right)^2.
\end{eqnarray*}
Hence, for possibly asymmetric $\bm{A}$, it holds that
\begin{eqnarray*}
 {\rm E} \left[ \left( \bm{x}'\bm{A}\bm{x}-{\rm tr}(\bm{A}) \right)^4 \right]
&=& {\rm E} \left[ \left( \bm{x}'\left( \frac{\bm{A}+\bm{A}'}{2} \right) \bm{x}
-{\rm tr}\left( \frac{\bm{A}+\bm{A}'}{2} \right) \right)^4 \right] \\
&=& \frac{1}{16} {\rm E} \left[ \left( \bm{x}'\left( \bm{A}+\bm{A}' \right) \bm{x}
-{\rm tr}\left( \bm{A}+\bm{A}' \right) \right)^4 \right] \\
&=& \frac{1}{16}\left\{ 48{\rm tr}\left( \left( \bm{A}+\bm{A}' \right)^4 \right) 
+ 12\left( {\rm tr}\left( \left( \bm{A}+\bm{A}' \right)^2 \right) \right)^2 \right\} \\
&=& 3{\rm tr}\left( \left( \bm{A}+\bm{A}' \right)^4 \right) 
+ \frac{3}{4}\left\{ {\rm tr}\left( \left( \bm{A}+\bm{A}' \right)^2 \right) \right\}^2.
\end{eqnarray*}
This completes the proof.
\qed

\subsection{Results for Wishart matrices}

In this subsection, we provide some asymptotic properties concerning Wishart matrices under  a high-dimensional asymptotic regime.
Throughout this section, let $a$, $b$, and $p$ be positive integers, and $\bm{\Sigma}$, $\bm{\Sigma}_a$, and $\bm{\Sigma}_b$ be $p\times p$ positive definite matrices.
Before presenting results, recall that if $\bm{z} \sim N_p(\bm{0},\bm{\Sigma})$ then
\[ 
{\rm E} [(\bm{z}'\bm{A}\bm{z})\bm{z}\bm{z}'] 
= {\rm E} [\bm{z}\bm{z}'\bm{A}\bm{z}\bm{z}'] 
= \bm{\Sigma}\bm{A}\bm{\Sigma} + \bm{\Sigma}\bm{A}'\bm{\Sigma} 
+ {\rm tr}(\bm{\Sigma}\bm{A})\bm{\Sigma} 
\]
for any $p\times p$ matrix $\bm{A}$.

First we provide asymptotic evaluations for moments of a Wishart matrix.

\begin{prop}\label{prop44}
If $\bm{T} \sim W_p(a,\bm{\Sigma})$ then
\[ {\rm E} \left[ {\rm tr}\left( \bm{T} \bm{A} \bm{T} \bm{B} \right) \right] = O(ap^2) , \quad
{\rm E} \left[ {\rm tr}\left( \bm{T} \bm{A} \right) {\rm tr}\left( \bm{T}\bm{B} \right) \right] 
= O(a^2p^2)   \]
as $p \to \infty$ with 
\begin{equation}\label{ar3}
a \asymp p^\delta, \quad 0<\delta<1.
\end{equation}
for $p\times p $ matrices $\bm{A}$ and $\bm{B}$ which satisfy
\[
 {\rm tr}(\bm{\Sigma} \bm{A} \bm{\Sigma} \bm{B}) =O(p), \ {\rm tr}(\bm{\Sigma} \bm{A}' \bm{\Sigma} \bm{B}) =O(p), \ {\rm tr}(\bm{\Sigma} \bm{A}) =O(p), \ {\rm tr}(\bm{\Sigma} \bm{B}) =O(p) .
\]
\end{prop}

\textit{Proof}
Let $\{ \bm{z}_i \}_{i=1}^{a}$ be an iid sequence satisfying $\bm{z}_1,\dots,\bm{z}_a \sim N_p(\bm{0},\bm{\Sigma})$.
Note that
\[ \bm{T} \stackrel{d}{=} \sum_{i=1}^{a}\bm{z}_i\bm{z}_i', \]
where $\stackrel{d}{=}$ means that the distributions of left-hand and right-hand sides are the same.
It follows from
\begin{eqnarray*}
\lefteqn{ {\rm E} \left[ \bm{T} \bm{A} \bm{T} \right] 
= {\rm E} \left[ \left( \sum_{i=1}^{a}\bm{z}_i\bm{z}_i' \right) \bm{A} 
\left( \sum_{i=1}^{a}\bm{z}_i\bm{z}_i' \right) \right] } \\
&=& \sum_{i=1}^{a} {\rm E} \left[ \bm{z}_i\bm{z}_i'\bm{A}\bm{z}_i\bm{z}_i' \right] 
+ \sum_{i=1}^{a}\sum_{j\neq i} {\rm E} \left[ \bm{z}_i\bm{z}_i'\bm{A}\bm{z}_j\bm{z}_j' \right] \\
&=& \sum_{i=1}^{a} \left\{ \bm{\Sigma} \bm{A} \bm{\Sigma} + \bm{\Sigma} \bm{A}' \bm{\Sigma} 
+ {\rm tr}(\bm{\Sigma} \bm{A})\bm{\Sigma} \right\} 
+ \sum_{i=1}^{a}\sum_{j\neq i} \bm{\Sigma} \bm{A} \bm{\Sigma} \\
&=& a^2 \bm{\Sigma} \bm{A} \bm{\Sigma} 
+ a \bm{\Sigma} \bm{A}' \bm{\Sigma} 
+ a {\rm tr}(\bm{\Sigma} \bm{A}) \bm{\Sigma}.
\end{eqnarray*}
that
\[
{\rm E} \left[ {\rm tr}\left( \bm{T} \bm{A} \bm{T} \bm{B} \right) \right] = O(a^2p) + O(ap^2) = O(ap^2) .
\]
Moreover, it follows from
\begin{eqnarray*}
\lefteqn{ {\rm E} \left[ {\rm tr}\left( \bm{T} \bm{A} \right) \bm{T} \right]
= {\rm E} \left[ {\rm tr}\left( \sum_{i=1}^{a}\bm{z}_i\bm{z}_i' \bm{A} \right) 
\left( \sum_{i=1}^{a}\bm{z}_i\bm{z}_i' \right) \right] } \\
&=& {\rm E} \left[ \left( \sum_{i=1}^{a}\bm{z}_i' \bm{A}\bm{z}_i \right) 
\left( \sum_{i=1}^{a}\bm{z}_i\bm{z}_i' \right) \right] \\
&=& \sum_{i=1}^{a} {\rm E} \left[ \left( \bm{z}_i' \bm{A}\bm{z}_i \right) 
\left( \bm{z}_i\bm{z}_i' \right) \right] 
+ \sum_{i=1}^{a}\sum_{j\neq i} {\rm E} \left[ \left( \bm{z}_i' \bm{A}\bm{z}_i \right) 
\left( \bm{z}_j\bm{z}_j' \right) \right] \\
&=& \sum_{i=1}^{a} 
\left\{ \bm{\Sigma}\bm{A}\bm{\Sigma} + \bm{\Sigma}\bm{A}'\bm{\Sigma} 
+ {\rm tr}(\bm{\Sigma}\bm{A})\bm{\Sigma} \right\} 
+ \sum_{i=1}^{a}\sum_{j\neq i}
{\rm tr}(\bm{\Sigma}\bm{A}) \bm{\Sigma} \\
&=& a\bm{\Sigma} \bm{A}\bm{\Sigma} 
+ a\bm{\Sigma} \bm{A}'\bm{\Sigma} 
+ a^2{\rm tr}(\bm{\Sigma}\bm{A}) \bm{\Sigma}
\end{eqnarray*}
that
\[
{\rm E} \left[ {\rm tr}\left( \bm{T} \bm{A} \right) {\rm tr}\left( \bm{T}\bm{B} \right) \right] 
= O(a^2p^2).
\]
This completes the proof.
\qed

\begin{prop}\label{prop45}
If $\bm{T} \sim W_p(a,\bm{\Sigma})$ then
\begin{eqnarray*} 
&& {\rm E} \left[ {\rm tr}\left( \bm{T} \bm{A} \bm{T} \bm{B} \bm{T} \bm{C} \right) \right] = O(ap^3),  \\
&& {\rm E} \left[ {\rm tr}\left( \bm{T} \bm{A} \bm{T} \bm{B} \right) {\rm tr}\left( \bm{T} \bm{C} \right) \right] 
= O(a^2 p^3) , \\
&& {\rm E} \left[ {\rm tr}\left( \bm{T} \bm{A} \right) {\rm tr}\left( \bm{T} \bm{B} \right) 
{\rm tr}\left( \bm{T} \bm{C} \right) \right] = O(a^3p^3) 
\end{eqnarray*}
as $p \to \infty$ with \eqref{ar3} 
for $p\times p $ matrices $\bm{A}$, $\bm{B}$, and $\bm{C}$ which satisfy
\[ {\rm tr}((\bm{\Sigma}\bm{A})^4) =O(p), \ {\rm tr}((\bm{\Sigma}\bm{B})^4) =O(p), \ {\rm tr}((\bm{\Sigma}\bm{C})^4) =O(p), \]
\[ {\rm tr}(\bm{\Sigma} \bm{A} \bm{\Sigma} \bm{B} \bm{\Sigma} \bm{C}) =O(p), \ 
{\rm tr}(\bm{\Sigma} \bm{A}' \bm{\Sigma} \bm{B} \bm{\Sigma} \bm{C}) =O(p), \] 
\[{\rm tr}(\bm{\Sigma} \bm{B}' \bm{\Sigma} \bm{A}' \bm{\Sigma} \bm{C}) =O(p),  \ 
{\rm tr}(\bm{\Sigma} \bm{A} \bm{\Sigma} \bm{B}' \bm{\Sigma} \bm{C}) =O(p),  \]
\[  {\rm tr}(\bm{\Sigma} \bm{A} \bm{\Sigma} \bm{B} \bm{\Sigma} \bm{C}') =O(p), \ 
{\rm tr}(\bm{\Sigma} \bm{A} \bm{\Sigma} \bm{C} \bm{\Sigma} \bm{B}) =O(p), \
{\rm tr}(\bm{\Sigma} \bm{A} \bm{\Sigma} \bm{C}' \bm{\Sigma} \bm{B}) =O(p), \]
\[ {\rm tr}(\bm{\Sigma} \bm{A} \bm{\Sigma} \bm{B} ) =O(p), \ 
{\rm tr}(\bm{\Sigma} \bm{A} \bm{\Sigma} \bm{C}) =O(p), \ 
{\rm tr}(\bm{\Sigma} \bm{B} \bm{\Sigma} \bm{C}) =O(p), \]
\[ {\rm tr}(\bm{\Sigma} \bm{A}' \bm{\Sigma} \bm{B} ) =O(p), \ 
{\rm tr}(\bm{\Sigma} \bm{A}' \bm{\Sigma} \bm{C}) =O(p), \ {\rm tr}(\bm{\Sigma} \bm{B}' \bm{\Sigma} \bm{C}) =O(p), \]
\[ {\rm tr}(\bm{\Sigma} \bm{A}) =O(p), \ {\rm tr}(\bm{\Sigma} \bm{B}) =O(p) , \ {\rm tr}(\bm{\Sigma} \bm{C}) =O(p). \]
\end{prop}

\textit{Proof}
Let $\{ \bm{z}_i \}_{i=1}^{a}$ be an iid sequence satisfying $\bm{z}_1,\dots,\bm{z}_a \sim N_p(\bm{0},\bm{\Sigma})$.
We see that
\begin{eqnarray*}
\lefteqn{ 
\sum_{i=1}^{a} {\rm E} \left[ \left( \bm{z}_i' \bm{A} \bm{z}_i \right) 
\left( \bm{z}_i' \bm{B} \bm{z}_i \right) \left( \bm{z}_i' \bm{C}\bm{z}_i \right) \right] 
= a {\rm E} \left[ \left( \bm{z}_1' \bm{A} \bm{z}_1 \right) 
\left( \bm{z}_1' \bm{B} \bm{z}_1 \right) \left( \bm{z}_1' \bm{C}\bm{z}_1 \right) \right] } \\
& \le & a \left( 
{\rm E} \left[ \left( \bm{z}_1' \bm{A} \bm{z}_1 \right)^2 
\left( \bm{z}_1' \bm{B} \bm{z}_1 \right)^2 \right] 
{\rm E} \left[ \left( \bm{z}_1' \bm{C} \bm{z}_1 \right)^2 \right]
\right)^{\frac{1}{2}} \\
& \le & a \left( 
{\rm E} \left[ \left( \bm{z}_1' \bm{A} \bm{z}_1 \right)^4 \right] 
{\rm E} \left[ \left( \bm{z}_1' \bm{B} \bm{z}_1 \right)^4 \right] 
{\rm E} \left[ \left( \bm{z}_1' \bm{C} \bm{z}_1 \right)^4 \right]
\right)^{\frac{1}{4}} \\
&=& O(ap^3),
\end{eqnarray*}
where ${\rm E} \left[ \left( \bm{z}_1' \bm{A} \bm{z}_1 \right)^4 \right] = O(p^4)$, ${\rm E} \left[ \left( \bm{z}_1' \bm{B} \bm{z}_1 \right)^4 \right] = O(p^4)$, and ${\rm E} \left[ \left( \bm{z}_1' \bm{C} \bm{z}_1 \right)^4 \right] = O(p^4)$ follow from Lemma 4.1 and $2 \bm{z}_1' \bm{A} \bm{z}_1 = \bm{z}_1' ( \bm{A}+\bm{A}' ) \bm{z}_1$.
Then, we have
\begin{eqnarray*}
&& {\rm E} \left[ {\rm tr}\left( \bm{T} \bm{A} \bm{T} \bm{B} \bm{T} \bm{C} \right) \right] \\
&=& {\rm tr}\left( {\rm E} \left[ \left( \sum_{i=1}^{a}\bm{z}_i\bm{z}_i' \right) \bm{A} 
\left( \sum_{j=1}^{a}\bm{z}_j\bm{z}_j' \right) \bm{B} \left( \sum_{k=1}^{a}\bm{z}_k\bm{z}_k' \right) \bm{C} \right] 
\right) \\
&=& {\rm tr}\left( {\rm E} \left[ \sum_{i=1}^{a}\sum_{j=1}^{a}\sum_{k=1}^{a}
\bm{z}_i\bm{z}_i' \bm{A} \bm{z}_j\bm{z}_j' \bm{B} \bm{z}_k\bm{z}_k' \bm{C} \right] \right) \\
&=& {\rm tr}\left( 
{\rm E} \left[ \sum_{i=1}^{a}\bm{z}_i\bm{z}_i' \bm{A} \bm{z}_i\bm{z}_i' \bm{B} \bm{z}_i\bm{z}_i' \bm{C} \right] 
+ {\rm E} \left[ \sum_{i=1}^{a}\sum_{k\neq i}
\bm{z}_i\bm{z}_i' \bm{A} \bm{z}_i\bm{z}_i' \bm{B} \bm{z}_k\bm{z}_k' \bm{C} \right] 
+ {\rm E} \left[ \sum_{i=1}^{a}\sum_{j\neq i}
\bm{z}_i\bm{z}_i' \bm{A} \bm{z}_j\bm{z}_j' \bm{B} \bm{z}_i\bm{z}_i' \bm{C} \right] \right. \\
&& \left. + {\rm E} \left[ \sum_{i=1}^{a}\sum_{j\neq i}
\bm{z}_i\bm{z}_i' \bm{A} \bm{z}_j\bm{z}_j' \bm{B} \bm{z}_j\bm{z}_j' \bm{C} \right] 
+ {\rm E} \left[ \sum_{i=1}^{a}\sum_{j\neq i}\sum_{k\neq i,j}
\bm{z}_i\bm{z}_i' \bm{A} \bm{z}_j\bm{z}_j' \bm{B} \bm{z}_k\bm{z}_k' \bm{C} \right] \right) \\
&=& \sum_{i=1}^{a} {\rm E} \left[ \left( \bm{z}_i' \bm{A} \bm{z}_i \right) 
\left( \bm{z}_i' \bm{B} \bm{z}_i \right) \left( \bm{z}_i' \bm{C}\bm{z}_i \right) \right] 
+ {\rm tr}\left( \sum_{i=1}^{a}\sum_{k\neq i}
\left\{ \left( \bm{\Sigma}\bm{A}\bm{\Sigma} + \bm{\Sigma}\bm{A}'\bm{\Sigma} 
+ {\rm tr}(\bm{\Sigma}\bm{A})\bm{\Sigma} \right) \bm{B} \bm{\Sigma} \bm{C} \right\} \right. \\
&& + \sum_{i=1}^{a}\sum_{j\neq i} \left\{ 
\left( \bm{\Sigma}\bm{A}\bm{\Sigma} \bm{B}\bm{\Sigma} + \bm{\Sigma}\bm{B}'\bm{\Sigma}\bm{A}'\bm{\Sigma} 
+ {\rm tr}(\bm{\Sigma}\bm{A}\bm{\Sigma} \bm{B})\bm{\Sigma} \right) \bm{C} \right\} \\
&& \left. + \sum_{i=1}^{a}\sum_{j\neq i}
\left\{ \bm{\Sigma} \bm{A} \left( \bm{\Sigma}\bm{B}\bm{\Sigma} + \bm{\Sigma}\bm{B}'\bm{\Sigma} 
+ {\rm tr}(\bm{\Sigma}\bm{B})\bm{\Sigma} \right) \bm{C} \right\} 
+ \sum_{i=1}^{a}\sum_{j\neq i}\sum_{k\neq i,j}
\left\{ \bm{\Sigma} \bm{A} \bm{\Sigma} \bm{B} \bm{\Sigma} \bm{C} \right\} \right) \\
&=& O(ap^3) + O(a^2p^2) + O(a^2p^2) + O(a^2p^2) + O(a^3p) \\
&=& O(ap^3) .
\end{eqnarray*}

The second assertion follows from
\begin{eqnarray*}
&& {\rm E} \left[ {\rm tr}\left( \bm{T} \bm{A} \bm{T} \bm{B} \right) {\rm tr}\left( \bm{T} \bm{C} \right) \right] \\
&=& {\rm E} \left[ {\rm tr}\left( \left( \sum_{i=1}^{a}\bm{z}_i\bm{z}_i' \right) \bm{A} 
\left( \sum_{j=1}^{a}\bm{z}_j\bm{z}_j' \right) \bm{B} \right) 
{\rm tr}\left( \left( \sum_{k=1}^{a}\bm{z}_k\bm{z}_k' \right) \bm{C} \right) \right] \\
&=& {\rm E} \left[ \sum_{i=1}^{a} \sum_{j=1}^{a} \sum_{k=1}^{a} 
{\rm tr}\left( \bm{z}_i\bm{z}_i' \bm{A} \bm{z}_j\bm{z}_j' \bm{B} \right) 
{\rm tr}\left( \bm{z}_k\bm{z}_k' \bm{C} \right) \right] \\
&=& {\rm E} \left[ \sum_{i=1}^{a} 
{\rm tr}\left( \bm{z}_i\bm{z}_i' \bm{A} \bm{z}_i\bm{z}_i' \bm{B} \right) 
{\rm tr}\left( \bm{z}_i\bm{z}_i' \bm{C} \right) \right] 
+ {\rm E} \left[ \sum_{i=1}^{a} \sum_{k\neq i} 
{\rm tr}\left( \bm{z}_i\bm{z}_i' \bm{A} \bm{z}_i\bm{z}_i' \bm{B} \right) 
{\rm tr}\left( \bm{z}_k\bm{z}_k' \bm{C} \right) \right] \\
&& + {\rm E} \left[ \sum_{i=1}^{a} \sum_{j\neq i}
{\rm tr}\left( \bm{z}_i\bm{z}_i' \bm{A} \bm{z}_j\bm{z}_j' \bm{B} \right) 
{\rm tr}\left( \bm{z}_i\bm{z}_i' \bm{C} \right) \right] 
+ {\rm E} \left[ \sum_{i=1}^{a} \sum_{j\neq i} 
{\rm tr}\left( \bm{z}_i\bm{z}_i' \bm{A} \bm{z}_j\bm{z}_j' \bm{B} \right) 
{\rm tr}\left( \bm{z}_j\bm{z}_j' \bm{C} \right) \right] \\
&& + {\rm E} \left[ \sum_{i=1}^{a} \sum_{j\neq i} \sum_{k\neq i,j} 
{\rm tr}\left( \bm{z}_i\bm{z}_i' \bm{A} \bm{z}_j\bm{z}_j' \bm{B} \right) 
{\rm tr}\left( \bm{z}_k\bm{z}_k' \bm{C} \right) \right] 
 \\
&=& {\rm E} \left[ \sum_{i=1}^{a} (\bm{z}_i' \bm{A} \bm{z}_i) (\bm{z}_i' \bm{B} \bm{z}_i)
( \bm{z}_i' \bm{C} \bm{z}_i ) \right] 
+ \sum_{i=1}^{a} \sum_{k\neq i} 
{\rm tr}\left( \left( 
\bm{\Sigma}\bm{A}\bm{\Sigma} + \bm{\Sigma}\bm{A}'\bm{\Sigma} 
+ {\rm tr}(\bm{\Sigma}\bm{A})\bm{\Sigma} \right) 
\bm{B} \right) 
{\rm tr}\left( \bm{\Sigma} \bm{C} \right) \\
&& + \sum_{i=1}^{a} \sum_{j\neq i}
{\rm tr}\left( \left( 
\bm{\Sigma}\bm{C}\bm{\Sigma} + \bm{\Sigma}\bm{C}'\bm{\Sigma} 
+ {\rm tr}(\bm{\Sigma}\bm{C})\bm{\Sigma} \right) 
\bm{A} \bm{\Sigma} \bm{B} \right) 
+ \sum_{i=1}^{a} \sum_{j\neq i} 
{\rm tr}\left( \bm{\Sigma} \bm{A} 
\left( 
\bm{\Sigma}\bm{C}\bm{\Sigma} + \bm{\Sigma}\bm{C}'\bm{\Sigma} 
+ {\rm tr}(\bm{\Sigma}\bm{C})\bm{\Sigma} \right) 
\bm{B} \right)  \\
&& + \sum_{i=1}^{a} \sum_{j\neq i} \sum_{k\neq i,j} 
{\rm tr}\left( \bm{\Sigma} \bm{A} \bm{\Sigma} \bm{B} \right) 
{\rm tr}\left( \bm{\Sigma} \bm{C} \right) 
\\
&=& O(ap^3) + O(a^2p^3) + O(a^2p^2) + O(a^2p^2) + O(a^3p^2) \\
&=& O(a^2p^3).
\end{eqnarray*}

The last assertion follows from
\begin{eqnarray*}
&& {\rm E} \left[ {\rm tr}\left( \bm{T} \bm{A} \right) {\rm tr} \left( \bm{T} \bm{B} \right) {\rm tr}\left( \bm{T} \bm{C} \right) \right] \\
&=& {\rm E} \left[ {\rm tr}\left( \left( \sum_{i=1}^{a}\bm{z}_i\bm{z}_i' \right) \bm{A} \right) 
{\rm tr}\left( \left( \sum_{j=1}^{a}\bm{z}_j\bm{z}_j' \right) \bm{B} \right) 
{\rm tr}\left( \left( \sum_{k=1}^{a}\bm{z}_k\bm{z}_k' \right) \bm{C} \right) \right] \\
&=& {\rm E} \left[ \sum_{i=1}^{a} \sum_{j=1}^{a} \sum_{k=1}^{a} 
{\rm tr}\left( \bm{z}_i\bm{z}_i' \bm{A} \right) 
{\rm tr}\left( \bm{z}_j\bm{z}_j' \bm{B} \right) 
{\rm tr}\left( \bm{z}_k\bm{z}_k' \bm{C} \right) \right] \\
&=& {\rm E} \left[ \sum_{i=1}^{a} 
{\rm tr}\left( \bm{z}_i\bm{z}_i' \bm{A} \right) {\rm tr}\left( \bm{z}_i\bm{z}_i' \bm{B} \right) 
{\rm tr}\left( \bm{z}_i\bm{z}_i' \bm{C} \right) \right] 
+ {\rm E} \left[ \sum_{i=1}^{a} \sum_{k\neq i} 
{\rm tr}\left( \bm{z}_i\bm{z}_i' \bm{A} \right) {\rm tr}\left( \bm{z}_i\bm{z}_i' \bm{B} \right) 
{\rm tr}\left( \bm{z}_k\bm{z}_k' \bm{C} \right) \right] \\
&& + {\rm E} \left[ \sum_{i=1}^{a} \sum_{j\neq i}
{\rm tr}\left( \bm{z}_i\bm{z}_i' \bm{A} \right) {\rm tr}\left( \bm{z}_j\bm{z}_j' \bm{B} \right) 
{\rm tr}\left( \bm{z}_i\bm{z}_i' \bm{C} \right) \right] 
+ {\rm E} \left[ \sum_{i=1}^{a} \sum_{j\neq i} 
{\rm tr}\left( \bm{z}_i\bm{z}_i' \bm{A} \right) {\rm tr}\left( \bm{z}_j\bm{z}_j' \bm{B} \right) 
{\rm tr}\left( \bm{z}_j\bm{z}_j' \bm{C} \right) \right] \\
&& + {\rm E} \left[ \sum_{i=1}^{a} \sum_{j\neq i} \sum_{k\neq i,j} 
{\rm tr}\left( \bm{z}_i\bm{z}_i' \bm{A} \right) {\rm tr}\left( \bm{z}_j\bm{z}_j' \bm{B} \right) 
{\rm tr}\left( \bm{z}_k\bm{z}_k' \bm{C} \right) \right] 
 \\
&=& {\rm E} \left[ \sum_{i=1}^{a} (\bm{z}_i' \bm{A} \bm{z}_i) (\bm{z}_i' \bm{B} \bm{z}_i)
( \bm{z}_i' \bm{C} \bm{z}_i ) \right] 
+ \sum_{i=1}^{a} \sum_{k\neq i} 
{\rm tr}\left( \left( 
\bm{\Sigma}\bm{A}\bm{\Sigma} + \bm{\Sigma}\bm{A}'\bm{\Sigma} 
+ {\rm tr}(\bm{\Sigma}\bm{A})\bm{\Sigma} \right) 
\bm{B} \right) 
{\rm tr}\left( \bm{\Sigma} \bm{C} \right) \\
&& + \sum_{i=1}^{a} \sum_{j\neq i}
{\rm tr}\left( \left( 
\bm{\Sigma}\bm{A}\bm{\Sigma} + \bm{\Sigma}\bm{A}'\bm{\Sigma} 
+ {\rm tr}(\bm{\Sigma}\bm{A})\bm{\Sigma} \right) 
\bm{C} \right) 
{\rm tr}\left( \bm{\Sigma} \bm{B} \right) 
+ \sum_{i=1}^{a} \sum_{j\neq i} 
{\rm tr}\left( \left( 
\bm{\Sigma}\bm{B}\bm{\Sigma} + \bm{\Sigma}\bm{B}'\bm{\Sigma} 
+ {\rm tr}(\bm{\Sigma}\bm{B})\bm{\Sigma} \right) 
\bm{C} \right) 
{\rm tr}\left( \bm{\Sigma} \bm{A} \right) \\
&& + \sum_{i=1}^{a} \sum_{j\neq i} \sum_{k\neq i,j} 
{\rm tr}\left( \bm{\Sigma} \bm{A} \right) 
{\rm tr}\left( \bm{\Sigma} \bm{B} \right) 
{\rm tr}\left( \bm{\Sigma} \bm{C} \right) 
\\
&=& O(ap^3) + O(a^2p^3) + O(a^2p^3) + O(a^2p^3)+ O(a^3p^3) \\
&=& O(a^3p^3).
\end{eqnarray*}
This completes the proof.
\qed

\begin{prop}\label{prop46}
If $\bm{T} \sim W_p(a,\bm{\Sigma})$ then
\begin{eqnarray*}
&& {\rm E} \left[ {\rm tr}\left( \bm{T} \bm{A} \bm{T} \bm{B} \bm{T} \bm{C} \bm{T} \bm{D} \right) \right] = O(ap^4), \\
&& {\rm E} \left[ {\rm tr}\left( \bm{T} \bm{A} \bm{T} \bm{B} \bm{T} \bm{C} \right) 
{\rm tr}\left( \bm{T} \bm{D} \right) \right] = O(a^2p^4), \\
&& {\rm E} \left[ {\rm tr}\left( \bm{T} \bm{A} \bm{T} \bm{B} \right) 
{\rm tr}\left( \bm{T} \bm{C} \bm{T} \bm{D} \right) \right] = O(a^2p^4), \\
&& {\rm E} \left[ {\rm tr}\left( \bm{T} \bm{A} \bm{T} \bm{B} \right) 
{\rm tr}\left( \bm{T} \bm{C} \right) 
{\rm tr}\left( \bm{T} \bm{D} \right) \right] = O(a^3p^4), \\
&& {\rm E} \left[ {\rm tr}\left( \bm{T} \bm{A} \right) 
{\rm tr}\left( \bm{T} \bm{B} \right) 
{\rm tr}\left( \bm{T} \bm{C} \right) 
{\rm tr}\left( \bm{T} \bm{D} \right) \right] = O(a^4p^4)
\end{eqnarray*}
as $p \to \infty$ with \eqref{ar3} for $p\times p $ matrices $\bm{A}$, $\bm{B}$, $\bm{C}$, and $\bm{D}$ which satisfy 
\[ {\rm tr}(\bm{\Sigma} \bm{\Psi}_1 \bm{\Sigma} \bm{\Psi}_2 \bm{\Sigma} \bm{\Psi}_3 \bm{\Sigma} \bm{\Psi}_4) =O(p), \ 
\bm{\Psi}_1,\bm{\Psi}_2,\bm{\Psi}_3,\bm{\Psi}_4 = \bm{A},\bm{A}',\bm{B},\bm{B}',\bm{C},\bm{C}',\bm{D},\bm{D}' \]
\[ {\rm tr}(\bm{\Sigma} \bm{\Psi}_1 \bm{\Sigma} \bm{\Psi}_2 \bm{\Sigma} \bm{\Psi}_3) =O(p), \ 
\bm{\Psi}_1,\bm{\Psi}_2,\bm{\Psi}_3 = \bm{A},\bm{A}',\bm{B},\bm{B}',\bm{C},\bm{C}',\bm{D},\bm{D}' \]
\[ {\rm tr}(\bm{\Sigma} \bm{\Psi}_1 \bm{\Sigma} \bm{\Psi}_2) =O(p), \ 
\bm{\Psi}_1,\bm{\Psi}_2 = \bm{A},\bm{A}',\bm{B},\bm{B}',\bm{C},\bm{C}',\bm{D},\bm{D}' \]
\[ {\rm tr}(\bm{\Sigma} \bm{\Psi}_1) =O(p), \ 
\bm{\Psi}_1 = \bm{A},\bm{B},\bm{C},\bm{D} \]
\end{prop}

\textit{Proof.}
First, we see that
\begin{eqnarray*}
\lefteqn{ 
\sum_{i=1}^{a} {\rm E} \left[ \left( \bm{z}_i' \bm{A} \bm{z}_i \right) 
\left( \bm{z}_i' \bm{B} \bm{z}_i \right) \left( \bm{z}_i' \bm{C}\bm{z}_i \right) \left( \bm{z}_i' \bm{D}\bm{z}_i \right) \right] 
= a {\rm E} \left[ \left( \bm{z}_1' \bm{A} \bm{z}_1 \right) 
\left( \bm{z}_1' \bm{B} \bm{z}_1 \right) \left( \bm{z}_1' \bm{C}\bm{z}_1 \right) \left( \bm{z}_1' \bm{D}\bm{z}_1 \right) \right] } \\
& \le & a \left( 
{\rm E} \left[ \left( \bm{z}_1' \bm{A} \bm{z}_1 \right)^2 
\left( \bm{z}_1' \bm{B} \bm{z}_1 \right)^2 \right] 
{\rm E} \left[ \left( \bm{z}_1' \bm{C} \bm{z}_1 \right)^2 
\left( \bm{z}_1' \bm{D} \bm{z}_1 \right)^2 \right] 
\right)^{\frac{1}{2}} \\
& \le & a \left( 
{\rm E} \left[ \left( \bm{z}_1' \bm{A} \bm{z}_1 \right)^4 \right] 
{\rm E} \left[ \left( \bm{z}_1' \bm{B} \bm{z}_1 \right)^4 \right] 
{\rm E} \left[ \left( \bm{z}_1' \bm{C} \bm{z}_1 \right)^4 \right]
{\rm E} \left[ \left( \bm{z}_1' \bm{D} \bm{z}_1 \right)^4 \right]
\right)^{\frac{1}{4}} \\
&=& O(ap^4).
\end{eqnarray*}
Then, we have 
\begin{eqnarray*}
\lefteqn{{\rm E} \left[ {\rm tr}\left( \bm{T} \bm{A} \bm{T} \bm{B} \bm{T} \bm{C} \bm{T} \bm{D} \right) \right] } \\
&=& {\rm tr}\left( {\rm E} \left[ \left( \sum_{i=1}^{a}\bm{z}_i\bm{z}_i' \right) \bm{A} 
\left( \sum_{j=1}^{a}\bm{z}_j\bm{z}_j' \right) \bm{B} \left( \sum_{k=1}^{a}\bm{z}_k\bm{z}_k' \right) \bm{C} 
\left( \sum_{l=1}^{a}\bm{z}_l\bm{z}_l' \right) \bm{D} \right] 
\right) \\
&=& {\rm tr}\left( {\rm E} \left[ \sum_{i=1}^{a}\sum_{j=1}^{a}\sum_{k=1}^{a}\sum_{l=1}^{a}
\bm{z}_i\bm{z}_i' \bm{A} \bm{z}_j\bm{z}_j' \bm{B} \bm{z}_k\bm{z}_k' \bm{C} \bm{z}_l\bm{z}_l' \bm{D} \right] \right) \\
&=& {\rm tr}\left( 
{\rm E} \left[ \sum_{i=1}^{a}\bm{z}_i\bm{z}_i' \bm{A} \bm{z}_i\bm{z}_i' \bm{B} \bm{z}_i\bm{z}_i' \bm{C} \bm{z}_i\bm{z}_i' \bm{D} \right] 
+ {\rm E} \left[ \sum_{i=1}^{a}\sum_{l\neq i}
\bm{z}_i\bm{z}_i' \bm{A} \bm{z}_i\bm{z}_i' \bm{B} \bm{z}_i\bm{z}_i' \bm{C} \bm{z}_l\bm{z}_l' \bm{D} \right] 
+ {\rm E} \left[ \sum_{i=1}^{a}\sum_{k\neq i}
\bm{z}_i\bm{z}_i' \bm{A} \bm{z}_i\bm{z}_i' \bm{B} \bm{z}_k\bm{z}_k' \bm{C} \bm{z}_i\bm{z}_i' \bm{D} \right] \right. \\
&& + {\rm E} \left[ \sum_{i=1}^{a}\sum_{j\neq i}
\bm{z}_i\bm{z}_i' \bm{A} \bm{z}_j\bm{z}_j' \bm{B} \bm{z}_i\bm{z}_i' \bm{C} \bm{z}_i\bm{z}_i' \bm{D} \right] 
+ {\rm E} \left[ \sum_{i=1}^{a}\sum_{j\neq i}
\bm{z}_i\bm{z}_i' \bm{A} \bm{z}_j\bm{z}_j' \bm{B} \bm{z}_j\bm{z}_j' \bm{C} \bm{z}_j\bm{z}_j' \bm{D} \right] 
+ {\rm E} \left[ \sum_{i=1}^{a}\sum_{k\neq i}
\bm{z}_i\bm{z}_i' \bm{A} \bm{z}_i\bm{z}_i' \bm{B} \bm{z}_k\bm{z}_k' \bm{C} \bm{z}_k\bm{z}_k' \bm{D} \right] \\
&& + {\rm E} \left[ \sum_{i=1}^{a}\sum_{j\neq i}
\bm{z}_i\bm{z}_i' \bm{A} \bm{z}_j\bm{z}_j' \bm{B} \bm{z}_i\bm{z}_i' \bm{C} \bm{z}_j\bm{z}_j' \bm{D} \right] 
+ {\rm E} \left[ \sum_{i=1}^{a}\sum_{j\neq i}
\bm{z}_i\bm{z}_i' \bm{A} \bm{z}_j\bm{z}_j' \bm{B} \bm{z}_j\bm{z}_j' \bm{C} \bm{z}_i\bm{z}_i' \bm{D} \right] \\
&& + {\rm E} \left[ \sum_{i=1}^{a}\sum_{k\neq i}\sum_{l\neq i,k}
\bm{z}_i\bm{z}_i' \bm{A} \bm{z}_i\bm{z}_i' \bm{B} \bm{z}_k\bm{z}_k' \bm{C} \bm{z}_l\bm{z}_l' \bm{D} \right] 
+ {\rm E} \left[ \sum_{i=1}^{a}\sum_{j\neq i}\sum_{l\neq i,j}
\bm{z}_i\bm{z}_i' \bm{A} \bm{z}_j\bm{z}_j' \bm{B} \bm{z}_i\bm{z}_i' \bm{C} \bm{z}_l\bm{z}_l' \bm{D} \right] \\
&& + {\rm E} \left[ \sum_{i=1}^{a}\sum_{j\neq i}\sum_{k\neq i,j}
\bm{z}_i\bm{z}_i' \bm{A} \bm{z}_j\bm{z}_j' \bm{B} \bm{z}_k\bm{z}_k' \bm{C} \bm{z}_i\bm{z}_i' \bm{D} \right] 
+ {\rm E} \left[ \sum_{i=1}^{a}\sum_{j\neq i}\sum_{l\neq i,j}
\bm{z}_i\bm{z}_i' \bm{A} \bm{z}_j\bm{z}_j' \bm{B} \bm{z}_j\bm{z}_j' \bm{C} \bm{z}_l\bm{z}_l' \bm{D} \right] \\
&& + {\rm E} \left[ \sum_{i=1}^{a}\sum_{j\neq i}\sum_{k\neq i,j}
\bm{z}_i\bm{z}_i' \bm{A} \bm{z}_j\bm{z}_j' \bm{B} \bm{z}_k\bm{z}_k' \bm{C} \bm{z}_j\bm{z}_j' \bm{D} \right] 
+ {\rm E} \left[ \sum_{i=1}^{a}\sum_{j\neq i}\sum_{k\neq i,j}
\bm{z}_i\bm{z}_i' \bm{A} \bm{z}_j\bm{z}_j' \bm{B} \bm{z}_k\bm{z}_k' \bm{C} \bm{z}_k\bm{z}_k' \bm{D} \right] \\
&& \left. + {\rm E} \left[ \sum_{i=1}^{a}\sum_{j\neq i}\sum_{k\neq i,j}\sum_{l\neq i,j,k} 
\bm{z}_i\bm{z}_i' \bm{A} \bm{z}_j\bm{z}_j' \bm{B} \bm{z}_k\bm{z}_k' \bm{C} \bm{z}_l\bm{z}_l' \bm{D} \right] \right) \\
& = & 
{\rm E} \left[ \sum_{i=1}^{a}(\bm{z}_i' \bm{A} \bm{z}_i)(\bm{z}_i' \bm{B} \bm{z}_i)(\bm{z}_i' \bm{C} \bm{z}_i)(\bm{z}_i' \bm{D}\bm{z}_i) \right] 
+ {\rm E} \left[ \sum_{i=1}^{a}\sum_{l\neq i}
(\bm{z}_i' \bm{A} \bm{z}_i)(\bm{z}_i' \bm{B} \bm{z}_i)(\bm{z}_i' \bm{C} \bm{\Sigma} \bm{D} \bm{z}_i) \right] \\
&& + {\rm E} \left[ \sum_{i=1}^{a}\sum_{k\neq i}
(\bm{z}_i' \bm{A} \bm{z}_i)(\bm{z}_i' \bm{B} \bm{\Sigma} \bm{C} \bm{z}_i)(\bm{z}_i' \bm{D}\bm{z}_i) \right] 
+ {\rm E} \left[ \sum_{i=1}^{a}\sum_{j\neq i}
(\bm{z}_i' \bm{A} \bm{\Sigma} \bm{B} \bm{z}_i)(\bm{z}_i' \bm{C} \bm{z}_i)(\bm{z}_i' \bm{D} \bm{z}_i) \right] \\
&& + {\rm E} \left[ \sum_{i=1}^{a}\sum_{j\neq i}
(\bm{z}_j' \bm{D} \bm{\Sigma} \bm{A} \bm{z}_j)(\bm{z}_j' \bm{B} \bm{z}_j)(\bm{z}_j' \bm{C} \bm{z}_j) \right] \\
&& + {\rm E} \left[ \sum_{i=1}^{a}\sum_{k\neq i}
{\rm tr} \left\{ (\bm{\Sigma}\bm{A}\bm{\Sigma}\bm{B}+\bm{\Sigma}\bm{A}'\bm{\Sigma}\bm{B}+{\rm tr}(\bm{\Sigma}\bm{A})\bm{\Sigma}\bm{B})
(\bm{\Sigma}\bm{C}\bm{\Sigma}\bm{D}+\bm{\Sigma}\bm{C}'\bm{\Sigma}\bm{D}+{\rm tr}(\bm{\Sigma}\bm{C})\bm{\Sigma}\bm{D})
\right\} 
\right] \\
&& + {\rm E} \left[ \sum_{i=1}^{a}\sum_{j\neq i} 
{\rm tr} \left\{ 
(\bm{\Sigma}\bm{A} \bm{z}_j\bm{z}_j' \bm{B}\bm{\Sigma}\bm{C}+\bm{\Sigma}\bm{B}' \bm{z}_j\bm{z}_j' \bm{A}'\bm{\Sigma}\bm{C}+{\rm tr}(\bm{\Sigma}\bm{A} \bm{z}_j\bm{z}_j' \bm{B})\bm{\Sigma}\bm{C})
\bm{z}_j\bm{z}_j' \bm{D} \right\} \right] \\
&& + {\rm E} \left[ \sum_{i=1}^{a}\sum_{j\neq i} 
{\rm tr} \left\{ (\bm{\Sigma}\bm{D}\bm{\Sigma}\bm{A}+\bm{\Sigma}\bm{D}'\bm{\Sigma}\bm{A}+{\rm tr}(\bm{\Sigma}\bm{D})\bm{\Sigma}\bm{A})
(\bm{\Sigma}\bm{B}\bm{\Sigma}\bm{C}+\bm{\Sigma}\bm{B}'\bm{\Sigma}\bm{C}+{\rm tr}(\bm{\Sigma}\bm{B})\bm{\Sigma}\bm{C})
\right\} \right] \\
&& + {\rm E} \left[ \sum_{i=1}^{a}\sum_{k\neq i}\sum_{l\neq i,k}
\bm{z}_i' \bm{A} \bm{z}_i\bm{z}_i' \bm{B} \bm{\Sigma} \bm{C} \bm{\Sigma} \bm{D} \bm{z}_i \right] 
+ {\rm E} \left[ \sum_{i=1}^{a}\sum_{j\neq i}\sum_{l\neq i,j}
\bm{z}_i' \bm{A} \bm{\Sigma} \bm{B} \bm{z}_i\bm{z}_i' \bm{C} \bm{\Sigma} \bm{D} \bm{z}_i \right] \\
&& + {\rm E} \left[ \sum_{i=1}^{a}\sum_{j\neq i}\sum_{k\neq i,j}
\bm{z}_i' \bm{A} \bm{\Sigma} \bm{B} \bm{\Sigma} \bm{C} \bm{z}_i\bm{z}_i' \bm{D} \bm{z}_i \right] 
+ {\rm E} \left[ \sum_{i=1}^{a}\sum_{j\neq i}\sum_{l\neq i,j}
\bm{z}_j' \bm{B} \bm{z}_j\bm{z}_j' \bm{C} \bm{\Sigma} \bm{D} \bm{\Sigma} \bm{A} \bm{z}_j \right] \\
&& + {\rm E} \left[ \sum_{i=1}^{a}\sum_{j\neq i}\sum_{k\neq i,j}
\bm{z}_j' \bm{B} \bm{\Sigma} \bm{C} \bm{z}_j\bm{z}_j' \bm{D} \bm{\Sigma} \bm{A} \bm{z}_j \right] 
+ {\rm E} \left[ \sum_{i=1}^{a}\sum_{j\neq i}\sum_{k\neq i,j}
\bm{z}_k' \bm{C} \bm{z}_k\bm{z}_k' \bm{D} \bm{\Sigma} \bm{A} \bm{\Sigma} \bm{B} \bm{z}_k \right] \\
&& + {\rm E} \left[ \sum_{i=1}^{a}\sum_{j\neq i}\sum_{k\neq i,j}\sum_{l\neq i,j,k} 
{\rm tr} \left\{ \bm{\Sigma} \bm{A} \bm{\Sigma} \bm{B} \bm{\Sigma} \bm{C} \bm{\Sigma} \bm{D} \right\} \right] \\
& = & O(ap^4) + \left\{ O(a^2p^3) + O(a^2p^3) + O(a^2p^3) + O(a^2p^3) \right\} \\
&& + \left\{ O(a^2p^3) + {\rm E} \left[ \sum_{i=1}^{a}\sum_{j\neq i} 
{\rm tr} \left\{ 
\bm{\Sigma}\bm{A} \bm{z}_j' \bm{B}\bm{\Sigma}\bm{C}\bm{z}_j\bm{z}_j\bm{z}_j' \bm{D}
+\bm{\Sigma}\bm{B}' \bm{z}_j' \bm{A}'\bm{\Sigma}\bm{C}\bm{z}_j \bm{z}_j\bm{z}_j' \bm{D}
+ \bm{\Sigma}\bm{C}\bm{z}_j' \bm{B}\bm{\Sigma}\bm{A} \bm{z}_j\bm{z}_j\bm{z}_j' \bm{D} \right\} \right] 
+ O(a^2p^3) \right\} \\
&& + \left\{ O(a^3p^2) + O(a^3p^2) + O(a^3p^2) + O(a^3p^2) + O(a^3p^2) + O(a^3p^2) \right\} 
+ O(a^4p) \\
&=& O(ap^4) + \left\{ O(a^2p^3) + O(a^2p^3) + O(a^2p^3) + O(a^2p^3) \right\} 
+ \left\{ O(a^2p^3) + O(a^2p^2) + O(a^2p^3) \right\} \\
&& + \left\{ O(a^3p^2) + O(a^3p^2) + O(a^3p^2) + O(a^3p^2) + O(a^3p^2) + O(a^3p^2) \right\} 
+ O(a^4p) \\
&=& O(ap^4) .
\end{eqnarray*}
We can prove the followings in similar ways:
\begin{eqnarray*}
&& {\rm E} \left[ {\rm tr}\left( \bm{T} \bm{A} \bm{T} \bm{B} \bm{T} \bm{C} \right) 
{\rm tr}\left( \bm{T} \bm{D} \right) \right] 
= O(a^4p^2) + O(a^3p^3) + O(a^2p^4) = O(a^2p^4), \\
&& {\rm E} \left[ {\rm tr}\left( \bm{T} \bm{A} \bm{T} \bm{B} \right) 
{\rm tr}\left( \bm{T} \bm{C} \bm{T} \bm{D} \right) \right] 
= O(a^4p^2) + O(a^3p^3) + O(a^2p^4) = O(a^2p^4), \\
&& {\rm E} \left[ {\rm tr}\left( \bm{T} \bm{A} \bm{T} \bm{B} \right) 
{\rm tr}\left( \bm{T} \bm{C} \right) 
{\rm tr}\left( \bm{T} \bm{D} \right) \right] 
= O(a^4p^3) + O(a^3p^4) = O(a^3p^4), \\
&& {\rm E} \left[ {\rm tr}\left( \bm{T} \bm{A} \right) 
{\rm tr}\left( \bm{T} \bm{B} \right) 
{\rm tr}\left( \bm{T} \bm{C} \right) 
{\rm tr}\left( \bm{T} \bm{D} \right) \right] = O(a^4p^4),
\end{eqnarray*}
so we omit their details.
This completes the proof.
\qed

When applying Propositions~\ref{prop44}, \ref{prop45}, and \ref{prop46}, their conditions are not useful.
Hence, let us show a simple sufficient condition.

\begin{ass}\label{ass3}
As $p\to\infty$ with \eqref{ar3}, it holds that
\begin{eqnarray*}
&& \frac{{\rm tr}(\bm{\Sigma}^2 )}{p} \to \sigma_{2}\in (0,\infty), \\
&& \frac{{\rm tr}(\bm{\Sigma}^3)}{p} \to \sigma_{3}\in (0,\infty), \\
&& \qquad \vdots \\
&& \frac{{\rm tr}(\bm{\Sigma}^{16})}{p} \to \sigma_{16}\in (0,\infty). 
\end{eqnarray*}
\end{ass}

The following statement is a direct corollary to Propositions~\ref{prop44}, \ref{prop45}, and \ref{prop46}.

\begin{cor}
Suppose Assumption~\ref{ass3}.
If $\bm{T} \sim W_p(a,\bm{\Sigma})$ then the conclusions of Propositions~\ref{prop44}, \ref{prop45}, and \ref{prop46} hold for $\bm{A},\bm{B},\bm{C},\bm{D} = \bm{I}_p, \bm{\Sigma} $.
\end{cor}

Next, we provide asymptotic evaluations for moments of two Wishart matrices.
Before presenting results, we introduce the following condition.

\begin{ass}\label{ass4}
As $p\to\infty$ with 
\begin{equation}\label{ar4}
a , b \asymp p^\delta, \quad 0<\delta<1,
\end{equation}
it holds that
\begin{eqnarray*}
&& \frac{{\rm tr}(\bm{\Sigma}_i \bm{\Sigma}_j)}{p} \to \sigma_{ij}\in (0,\infty), \\
&& \qquad \vdots \\
&& \frac{{\rm tr}(\bm{\Sigma}_i \bm{\Sigma}_j \bm{\Sigma}_k \bm{\Sigma}_l \bm{\Sigma}_{i'})}{p} 
\to \sigma_{ijkli'}\in (0,\infty), \\
&& \frac{{\rm tr}(\bm{\Sigma}_i \bm{\Sigma}_j \bm{\Sigma}_k \bm{\Sigma}_l 
\bm{\Sigma}_{i'} \bm{\Sigma}_{j'} )}{p} 
\to \sigma_{ijkli'j'}\in (-\infty,\infty), \\
&& \qquad \vdots \\
&& \frac{{\rm tr}(\bm{\Sigma}_i \bm{\Sigma}_j \bm{\Sigma}_k \bm{\Sigma}_l 
\bm{\Sigma}_{i'} \bm{\Sigma}_{j'} \bm{\Sigma}_{k'} \bm{\Sigma}_{l'}
\bm{\Sigma}_{i''} \bm{\Sigma}_{j''} \bm{\Sigma}_{k''} \bm{\Sigma}_{l''}
\bm{\Sigma}_{i'''} \bm{\Sigma}_{j'''} \bm{\Sigma}_{k'''} \bm{\Sigma}_{l'''})}{p} \\
&& \to \sigma_{ijkli'j'k'l'i''j''k''l''i'''j'''k'''l'''}\in (-\infty,\infty)
\end{eqnarray*}
for $i,j,k,l,i',j',k',l',i'',j'',k'',l'',i''',j''',k''',l'''=a,b$.
\end{ass}

\begin{prop}\label{prop48}
(i) Under Assumption~\ref{ass4}, if $\bm{T}_a\sim W_p(a,\bm{\Sigma}_a)$ and $\bm{T}_b\sim W_p(b,\bm{\Sigma}_b)$ are independent, then
\begin{eqnarray*}
&& {\rm E} \left[ {\rm tr}\left( \bm{T}_a  \bm{T}_a  \bm{T}_a \bm{T}_a 
\bm{T}_b  \bm{T}_b  \bm{T}_b  \bm{T}_b  \right) \right]  = O(abp^7), \\
&& {\rm E} \left[ {\rm tr}\left( \bm{T}_a  \bm{T}_a  \bm{T}_a  \bm{T}_b 
\bm{T}_a  \bm{T}_b  \bm{T}_b  \bm{T}_b  \right) \right]  = O(ab(a+b)p^6), \\
&& {\rm E} \left[ {\rm tr}\left( \bm{T}_a  \bm{T}_a  \bm{T}_b  \bm{T}_b 
\bm{T}_a  \bm{T}_a  \bm{T}_b  \bm{T}_b  \right) \right] = O(ab (a+b) p^6), \\
&& {\rm E} \left[ {\rm tr}\left( \bm{T}_a  \bm{T}_b  \bm{T}_a  \bm{T}_b 
\bm{T}_a  \bm{T}_b  \bm{T}_a  \bm{T}_b  \right) \right] = O(ab (a^3+b^3) p^4)
\end{eqnarray*}
as $p \to \infty$ with \eqref{ar4}.\\
(ii)  If $\bm{T}_a\sim W_p(a,\bm{\Sigma}_a)$ and $\bm{T}_b\sim W_p(b,\bm{\Sigma}_b)$ are independent, then
\begin{eqnarray*}
&& {\rm E} \left[ {\rm tr}\left( \bm{T}_a \bm{A} \bm{T}_a \bm{B} \bm{T}_a \bm{C} \bm{T}_a \bm{D} \right) 
{\rm tr}\left( \bm{T}_b \bm{E} \bm{T}_b \bm{F} \bm{T}_b \bm{G} \bm{T}_b \bm{H} \right) \right] 
=  O(abp^8) \\
&& {\rm E} \left[ {\rm tr}\left( \bm{T}_a \bm{A} \bm{T}_a \bm{B} \bm{T}_b \bm{C} \bm{T}_b \bm{D} \right) 
{\rm tr}\left( \bm{T}_a \bm{E} \bm{T}_a \bm{F} \bm{T}_b \bm{G} \bm{T}_b \bm{H} \right) \right] 
 =  O(a^2b^2p^6) \\
&& {\rm E} \left[ {\rm tr}\left( \bm{T}_a \bm{A} \bm{T}_b \bm{B} \bm{T}_a \bm{C} \bm{T}_b \bm{D} \right) 
{\rm tr}\left( \bm{T}_a \bm{E} \bm{T}_b \bm{F} \bm{T}_a \bm{G} \bm{T}_b \bm{H} \right) \right] 
=  O(a^2b^2 (a^2 + b^2)p^4)
\end{eqnarray*}
as $p \to \infty$ with \eqref{ar4} 
for $p\times p $ matrices $\bm{A}$, $\bm{B}$, $\bm{C}$, $\bm{D}$, $\bm{E}$, $\bm{F}$, $\bm{G}$, and $\bm{H}$ which satisfy
\begin{eqnarray*}
&& {\rm tr}\left( \bm{\Sigma}_1\bm{\Psi}_1\bm{\Sigma}_2\bm{\Psi}_2\bm{\Sigma}_3\bm{\Psi}_3\bm{\Sigma}_4\bm{\Psi}_4
\bm{\Sigma}_5\bm{\Psi}_5\bm{\Sigma}_6\bm{\Psi}_6\bm{\Sigma}_7\bm{\Psi}_7\bm{\Sigma}_8\bm{\Psi}_8 \right) =O(p),\\
&& \bm{\Sigma}_i=\bm{I}_p,\bm{\Sigma}_a,\bm{\Sigma}_b, \ (i=1,2,3,4,5,6,7,8), \\ 
&& \bm{\Psi}_i=\bm{I}_p,\bm{A},\bm{B},\bm{C},\bm{D},\bm{E},\bm{F},\bm{G},\bm{H}, \ (i=1,2,3,4,5,6,7,8) .
\end{eqnarray*}
In partilular, when Assumption~\ref{ass4} is additionaly supposed, the conclusions  hold for 
\[ \bm{A},\bm{B},\bm{C},\bm{D},\bm{E},\bm{F},\bm{G},\bm{H} 
= \bm{I}_p, \bm{\Sigma}_a, \bm{\Sigma}_b . \]
\end{prop}

\textit{Proof.}
(i) To avoid showing quite long routine calculations, we simply give a sketch of the proof only for
\[ {\rm E} \left[ {\rm tr}\left( \bm{T}_a  \bm{T}_a  \bm{T}_a \bm{T}_a 
\bm{T}_b  \bm{T}_b  \bm{T}_b  \bm{T}_b  \right) \right]  = O(abp^7). \]
We have
\begin{eqnarray*}
&& {\rm E} \left[ {\rm tr}\left( \bm{T}_a  \bm{T}_a  \bm{T}_a \bm{T}_a 
\bm{T}_b  \bm{T}_b  \bm{T}_b  \bm{T}_b  \right) \right] \\
&=&  {\rm E} \left[ \sum_{i=1}^{a} \sum_{j=1}^{a} \sum_{k=1}^{a} \sum_{l=1}^{a} 
{\rm tr}\left( \bm{\Sigma}_{a}^{\frac{1}{2}} \bm{z}_i\bm{z}_i' 
\bm{\Sigma}_{a} \bm{z}_j\bm{z}_j' \bm{\Sigma}_{a} \bm{z}_k\bm{z}_k' 
\bm{\Sigma}_{a} \bm{z}_l\bm{z}_l' \bm{\Sigma}_{a}^{\frac{1}{2}}
\bm{T}_b  \bm{T}_b  \bm{T}_b  \bm{T}_b \right) \right] \\
&=&  
{\rm E} \left[ \sum_{i=1}^{a} 
\left( \bm{z}_i' \bm{\Sigma}_{a} \bm{z}_i \bm{z}_i' \bm{\Sigma}_{a} \bm{z}_i\bm{z}_i' 
\bm{\Sigma}_{a} \bm{z}_i\bm{z}_i' \bm{\Sigma}_{a}^{\frac{1}{2}}
\bm{T}_b  \bm{T}_b  \bm{T}_b  \bm{T}_b \bm{\Sigma}_{a}^{\frac{1}{2}} \bm{z}_i \right) \right.  \\
&& 
+ \sum_{i=1}^{a} \sum_{j\neq i} 
\left( \bm{z}_i' \bm{\Sigma}_{a} \bm{z}_i\bm{z}_i' \bm{\Sigma}_{a} \bm{z}_i\bm{z}_i' 
\bm{\Sigma}_{a} \bm{z}_j\bm{z}_j' \bm{\Sigma}_{a}^{\frac{1}{2}}
\bm{T}_b  \bm{T}_b  \bm{T}_b  \bm{T}_b \bm{\Sigma}_{a}^{\frac{1}{2}} \bm{z}_i \right) 
+ \sum_{i=1}^{a} \sum_{j\neq i}
\left( \bm{z}_i' \bm{\Sigma}_{a} \bm{z}_i\bm{z}_i' \bm{\Sigma}_{a} \bm{z}_j\bm{z}_j' 
\bm{\Sigma}_{a} \bm{z}_i\bm{z}_i' \bm{\Sigma}_{a}^{\frac{1}{2}}
\bm{T}_b  \bm{T}_b  \bm{T}_b  \bm{T}_b \bm{\Sigma}_{a}^{\frac{1}{2}} \bm{z}_i \right) \\
&& 
+ \sum_{i=1}^{a} \sum_{j\neq i}
\left( \bm{z}_i' 
\bm{\Sigma}_{a} \bm{z}_j\bm{z}_j' \bm{\Sigma}_{a} \bm{z}_i\bm{z}_i' 
\bm{\Sigma}_{a} \bm{z}_i\bm{z}_i' \bm{\Sigma}_{a}^{\frac{1}{2}}
\bm{T}_b  \bm{T}_b  \bm{T}_b  \bm{T}_b \bm{\Sigma}_{a}^{\frac{1}{2}} \bm{z}_i \right)
+ \sum_{i=1}^{a} \sum_{j\neq i}
\left( \bm{z}_i' \bm{\Sigma}_{a}^{\frac{1}{2}} \bm{T}_b  \bm{T}_b  \bm{T}_b  \bm{T}_b \bm{\Sigma}_{a}^{\frac{1}{2}} \bm{z}_j
\bm{z}_j' \bm{\Sigma}_{a} \bm{z}_i\bm{z}_i' \bm{\Sigma}_{a} \bm{z}_i\bm{z}_i' 
\bm{\Sigma}_{a} \bm{z}_i  \right) \\
&& + \sum_{i=1}^{a} \sum_{j\neq i}
\left( \bm{z}_i' 
\bm{\Sigma}_{a} \bm{z}_i\bm{z}_i' \bm{\Sigma}_{a} \bm{z}_j\bm{z}_j' 
\bm{\Sigma}_{a} \bm{z}_j\bm{z}_j' \bm{\Sigma}_{a}^{\frac{1}{2}}
\bm{T}_b  \bm{T}_b  \bm{T}_b  \bm{T}_b \bm{\Sigma}_{a}^{\frac{1}{2}} \bm{z}_i \right) 
+ \sum_{i=1}^{a} \sum_{j\neq i}
\left( \bm{z}_i' 
\bm{\Sigma}_{a} \bm{z}_j\bm{z}_j' \bm{\Sigma}_{a} \bm{z}_i\bm{z}_i' 
\bm{\Sigma}_{a} \bm{z}_j\bm{z}_j' \bm{\Sigma}_{a}^{\frac{1}{2}}
\bm{T}_b  \bm{T}_b  \bm{T}_b  \bm{T}_b \bm{\Sigma}_{a}^{\frac{1}{2}} \bm{z}_i \right) \\
&& + \sum_{i=1}^{a} \sum_{j\neq i}
\left( \bm{z}_i' 
\bm{\Sigma}_{a} \bm{z}_j\bm{z}_j' \bm{\Sigma}_{a} \bm{z}_j\bm{z}_j' 
\bm{\Sigma}_{a} \bm{z}_i\bm{z}_i' \bm{\Sigma}_{a}^{\frac{1}{2}}
\bm{T}_b  \bm{T}_b  \bm{T}_b  \bm{T}_b \bm{\Sigma}_{a}^{\frac{1}{2}} \bm{z}_i \right) \\
&& + \sum_{i=1}^{a} \sum_{j\neq i} \sum_{k\neq i,j}
\left( \bm{z}_i' 
\bm{\Sigma}_{a} \bm{z}_i\bm{z}_i' \bm{\Sigma}_{a} \bm{z}_j\bm{z}_j' 
\bm{\Sigma}_{a} \bm{z}_k\bm{z}_k' \bm{\Sigma}_{a}^{\frac{1}{2}}
\bm{T}_b  \bm{T}_b  \bm{T}_b  \bm{T}_b \bm{\Sigma}_{a}^{\frac{1}{2}} \bm{z}_i \right) 
+ \sum_{i=1}^{a} \sum_{j\neq i} \sum_{k\neq i,j}
\left( \bm{z}_i' 
\bm{\Sigma}_{a} \bm{z}_j\bm{z}_j' \bm{\Sigma}_{a} \bm{z}_i\bm{z}_i' 
\bm{\Sigma}_{a} \bm{z}_k\bm{z}_k' \bm{\Sigma}_{a}^{\frac{1}{2}}
\bm{T}_b  \bm{T}_b  \bm{T}_b  \bm{T}_b \bm{\Sigma}_{a}^{\frac{1}{2}} \bm{z}_i \right) \\
&& + \sum_{i=1}^{a} \sum_{j\neq i} \sum_{k\neq i,j}
\left( \bm{z}_i' 
\bm{\Sigma}_{a} \bm{z}_j\bm{z}_j' \bm{\Sigma}_{a} \bm{z}_k\bm{z}_k' 
\bm{\Sigma}_{a} \bm{z}_i\bm{z}_i' \bm{\Sigma}_{a}^{\frac{1}{2}}
\bm{T}_b  \bm{T}_b  \bm{T}_b  \bm{T}_b \bm{\Sigma}_{a}^{\frac{1}{2}} \bm{z}_i \right) 
+ \sum_{i=1}^{a} \sum_{j\neq i} \sum_{k\neq i,j}
\left( \bm{z}_j' \bm{\Sigma}_{a} \bm{z}_j\bm{z}_j' 
\bm{\Sigma}_{a} \bm{z}_k\bm{z}_k' \bm{\Sigma}_{a}^{\frac{1}{2}}
\bm{T}_b  \bm{T}_b  \bm{T}_b  \bm{T}_b \bm{\Sigma}_{a}^{\frac{1}{2}} \bm{z}_i\bm{z}_i' 
\bm{\Sigma}_{a} \bm{z}_j \right) \\
&& + \sum_{i=1}^{a} \sum_{j\neq i} \sum_{k\neq i,j}
\left( \bm{z}_j' \bm{\Sigma}_{a} \bm{z}_k\bm{z}_k' 
\bm{\Sigma}_{a} \bm{z}_j\bm{z}_j' \bm{\Sigma}_{a}^{\frac{1}{2}}
\bm{T}_b  \bm{T}_b  \bm{T}_b  \bm{T}_b \bm{\Sigma}_{a}^{\frac{1}{2}} \bm{z}_i\bm{z}_i' 
\bm{\Sigma}_{a} \bm{z}_j \right) 
+ \sum_{i=1}^{a} \sum_{j\neq i} \sum_{k\neq i,j}
\left( \bm{z}_k' \bm{\Sigma}_{a}^{\frac{1}{2}}
\bm{T}_b  \bm{T}_b  \bm{T}_b  \bm{T}_b \bm{\Sigma}_{a}^{\frac{1}{2}} \bm{z}_i\bm{z}_i' 
\bm{\Sigma}_{a} \bm{z}_j\bm{z}_j' \bm{\Sigma}_{a} \bm{z}_k\bm{z}_k' 
\bm{\Sigma}_{a} \bm{z}_k \right) \\
&& \left. 
+ \sum_{i=1}^{a} \sum_{j\neq i} \sum_{k\neq i,j} \sum_{l\neq i,j,k} 
\left( \bm{z}_i' 
\bm{\Sigma}_{a} \bm{z}_j\bm{z}_j' \bm{\Sigma}_{a} \bm{z}_k\bm{z}_k' 
\bm{\Sigma}_{a} \bm{z}_l\bm{z}_l' \bm{\Sigma}_{a}^{\frac{1}{2}}
\bm{T}_b  \bm{T}_b  \bm{T}_b  \bm{T}_b \bm{\Sigma}_{a}^{\frac{1}{2}} \bm{z}_i \right) 
\right]  \\
&\sim& 
{\rm E} \left[ \sum_{i=1}^{a} 
\left( {\rm tr}(\bm{\Sigma}_{a}) {\rm tr}(\bm{\Sigma}_{a}) {\rm tr}(\bm{\Sigma}_{a}) 
{\rm tr}(\bm{\Sigma}_{a}\bm{T}_b  \bm{T}_b  \bm{T}_b  \bm{T}_b) \right) \right. \\
&& 
+ \sum_{i=1}^{a} \sum_{j\neq i} 
\left( {\rm tr}(\bm{\Sigma}_{a}) {\rm tr}(\bm{\Sigma}_{a}) 
{\rm tr}( \bm{\Sigma}_{a}^{2}\bm{T}_b  \bm{T}_b  \bm{T}_b  \bm{T}_b) \right) 
+ \sum_{i=1}^{a} \sum_{j\neq i}
\left( {\rm tr}(\bm{\Sigma}_{a}) {\rm tr}(\bm{\Sigma}_{a}^{2}) 
{\rm tr}(\bm{\Sigma}_{a}\bm{T}_b  \bm{T}_b  \bm{T}_b  \bm{T}_b) \right) \\
&& 
+ \sum_{i=1}^{a} \sum_{j\neq i}
\left( {\rm tr}(\bm{\Sigma}_{a}^{2}) {\rm tr}(\bm{\Sigma}_{a}) 
{\rm tr}(\bm{\Sigma}_{a}\bm{T}_b  \bm{T}_b  \bm{T}_b  \bm{T}_b) \right)
+ \sum_{i=1}^{a} \sum_{j\neq i}
\left( {\rm tr}( \bm{\Sigma}_{a}^{2} \bm{T}_b  \bm{T}_b  \bm{T}_b  \bm{T}_b) 
{\rm tr}(\bm{\Sigma}_{a}) {\rm tr}(\bm{\Sigma}_{a}) \right) \\
&& + \sum_{i=1}^{a} \sum_{j\neq i}
\left( {\rm tr}(\bm{\Sigma}_{a})
{\rm tr}(\bm{\Sigma}_{a}^{2}\bm{T}_b  \bm{T}_b  \bm{T}_b  \bm{T}_b)  {\rm tr}(\bm{\Sigma}_{a}) \right) 
+ \sum_{i=1}^{a} \sum_{j\neq i}
\left( \bm{z}_i' 
\bm{\Sigma}_{a} 
\left( 2\bm{\Sigma}_{a} \bm{z}_i\bm{z}_i' \bm{\Sigma}_{a} + {\rm tr}(\bm{\Sigma}_{a} \bm{z}_i\bm{z}_i' \bm{\Sigma}_{a})\bm{I}_p \right) 
\bm{\Sigma}_{a}^{\frac{1}{2}}
\bm{T}_b  \bm{T}_b  \bm{T}_b  \bm{T}_b \bm{\Sigma}_{a}^{\frac{1}{2}} \bm{z}_i \right) \\
&& + \sum_{i=1}^{a} \sum_{j\neq i}
\left( {\rm tr}(\bm{\Sigma}_{a}^{2}) 
{\rm tr}(\bm{\Sigma}_{a}\bm{T}_b  \bm{T}_b  \bm{T}_b  \bm{T}_b) {\rm tr}(\bm{\Sigma}_{a}) \right) \\
&& + \sum_{i=1}^{a} \sum_{j\neq i} \sum_{k\neq i,j}
\left({\rm tr}(\bm{\Sigma}_{a}) {\rm tr}(\bm{\Sigma}_{a}^{3} \bm{T}_b  \bm{T}_b  \bm{T}_b  \bm{T}_b) \right) 
+ \sum_{i=1}^{a} \sum_{j\neq i} \sum_{k\neq i,j}
\left( {\rm tr}(\bm{\Sigma}_{a}^{2})
{\rm tr}(\bm{\Sigma}_{a}^{2} \bm{T}_b  \bm{T}_b  \bm{T}_b  \bm{T}_b) \right) \\
&& + \sum_{i=1}^{a} \sum_{j\neq i} \sum_{k\neq i,j}
\left( {\rm tr}(\bm{\Sigma}_{a}^{3}) 
{\rm tr}(\bm{\Sigma}_{a}\bm{T}_b  \bm{T}_b  \bm{T}_b  \bm{T}_b) \right) 
+ \sum_{i=1}^{a} \sum_{j\neq i} \sum_{k\neq i,j}
\left( {\rm tr}(\bm{\Sigma}_{a}) 
{\rm tr}(\bm{\Sigma}_{a}^{3} \bm{T}_b  \bm{T}_b  \bm{T}_b  \bm{T}_b) \right) \\
&& + \sum_{i=1}^{a} \sum_{j\neq i} \sum_{k\neq i,j}
\left( {\rm tr}(\bm{\Sigma}_{a}^{2})
{\rm tr}(\bm{\Sigma}_{a}^{2}\bm{T}_b  \bm{T}_b  \bm{T}_b  \bm{T}_b) \right) 
+ \sum_{i=1}^{a} \sum_{j\neq i} \sum_{k\neq i,j}
\left( {\rm tr}(\bm{\Sigma}_{a}^{3}\bm{T}_b  \bm{T}_b  \bm{T}_b  \bm{T}_b) 
{\rm tr}(\bm{\Sigma}_{a}) \right) \\
&& \left. 
+ \sum_{i=1}^{a} \sum_{j\neq i} \sum_{k\neq i,j} \sum_{l\neq i,j,k} 
\left( {\rm tr}(\bm{\Sigma}_{a}^{4}\bm{T}_b  \bm{T}_b  \bm{T}_b  \bm{T}_b) \right) 
\right]  \\
&\sim&  a {\rm tr}\left( \bm{\Sigma}_{a} \right) {\rm tr}\left( \bm{\Sigma}_{a} \right) {\rm tr}\left( \bm{\Sigma}_{a} \right) 
{\rm E}\left[ {\rm tr}\left( \bm{\Sigma}_{a} \bm{T}_b  \bm{T}_b  \bm{T}_b  \bm{T}_b \right) \right]  \\
&=& O(abp^7),
\end{eqnarray*}
where we use the following facts in the above asymptotic evaluation: 
Let $\bm{x}\sim N_p(\bm{0}_p,\bm{I}_p)$ and 
$\bm{Q},\bm{R},\bm{S},\bm{T}$ be $p\times p$ matrices, 
then
\[
E\left[ \bm{x}(\bm{x}'\bm{Q}\bm{x})\bm{x}' \right] 
= \bm{Q} + \bm{Q}' + {\rm tr}(\bm{Q})\bm{I}_p, \]
\[ E\left[ (\bm{x}'\bm{Q}\bm{x}) (\bm{x}'\bm{R}\bm{x}) \right] 
= {\rm tr}(\bm{Q}){\rm tr}(\bm{R}) + {\rm tr}(\bm{Q}\bm{R}) + {\rm tr}(\bm{Q}\bm{R}'), \]
\begin{eqnarray*}
&& E\left[ (\bm{x}'\bm{Q}\bm{x}) (\bm{x}'\bm{R}\bm{x}) (\bm{x}'\bm{S}\bm{x}) \right] \\
&=& 
{\rm tr}(\bm{Q}) {\rm tr}(\bm{R}) {\rm tr}(\bm{S}) 
+{\rm tr}(\bm{Q}) {\rm tr}(\bm{R}\bm{S}) 
+{\rm tr}(\bm{Q}) {\rm tr}(\bm{R}\bm{S}') 
+{\rm tr}(\bm{R}) {\rm tr}(\bm{Q}\bm{S}) 
+{\rm tr}(\bm{R}) {\rm tr}(\bm{Q}\bm{S}') 
+{\rm tr}(\bm{S}) {\rm tr}(\bm{Q}\bm{R}) 
+{\rm tr}(\bm{S}) {\rm tr}(\bm{Q}\bm{R}') \\
&& +{\rm tr}(\bm{Q}\bm{R}\bm{S}) 
+{\rm tr}(\bm{Q}\bm{R}\bm{S}') 
+{\rm tr}(\bm{Q}\bm{R}'\bm{S}) 
+{\rm tr}(\bm{Q}\bm{R}'\bm{S}') 
+{\rm tr}(\bm{Q}\bm{S}\bm{R}) 
+{\rm tr}(\bm{Q}\bm{S}\bm{R}') 
+{\rm tr}(\bm{Q}\bm{S}'\bm{R}) 
+{\rm tr}(\bm{Q}\bm{S}'\bm{R}'),
\end{eqnarray*}
and
\begin{eqnarray*}
&&E\left[ (\bm{x}'\bm{Q}\bm{x}) (\bm{x}'\bm{R}\bm{x}) (\bm{x}'\bm{S}\bm{x}) (\bm{x}'\bm{T}\bm{x}) \right] \\
&=& {\rm tr}(\bm{Q}) {\rm tr}(\bm{R}) {\rm tr}(\bm{S}) {\rm tr}(\bm{T}) \\
&& + {\rm tr}(\bm{Q}) {\rm tr}(\bm{R}) {\rm tr}(\bm{S}\bm{T}) 
+ {\rm tr}(\bm{Q}) {\rm tr}(\bm{R}) {\rm tr}(\bm{S}\bm{T}') 
+ {\rm tr}(\bm{Q}) {\rm tr}(\bm{S}) {\rm tr}(\bm{R}\bm{T}) 
+ {\rm tr}(\bm{Q}) {\rm tr}(\bm{S}) {\rm tr}(\bm{R}\bm{T}') \\
&& + {\rm tr}(\bm{Q}) {\rm tr}(\bm{T}) {\rm tr}(\bm{R}\bm{S}) 
+ {\rm tr}(\bm{Q}) {\rm tr}(\bm{T}) {\rm tr}(\bm{R}\bm{S}') 
+ {\rm tr}(\bm{R}) {\rm tr}(\bm{S}) {\rm tr}(\bm{Q}\bm{T}) 
+ {\rm tr}(\bm{R}) {\rm tr}(\bm{S}) {\rm tr}(\bm{Q}\bm{T}') \\
&& + {\rm tr}(\bm{R}) {\rm tr}(\bm{T}) {\rm tr}(\bm{Q}\bm{S}) 
+ {\rm tr}(\bm{R}) {\rm tr}(\bm{T}) {\rm tr}(\bm{Q}\bm{S}') 
+ {\rm tr}(\bm{S}) {\rm tr}(\bm{T}) {\rm tr}(\bm{Q}\bm{R}) 
+ {\rm tr}(\bm{S}) {\rm tr}(\bm{T}) {\rm tr}(\bm{Q}\bm{R}') \\
&& + {\rm tr}(\bm{Q}\bm{R}) {\rm tr}(\bm{S}\bm{T}) 
+ {\rm tr}(\bm{Q}\bm{R}') {\rm tr}(\bm{S}\bm{T}) 
+ {\rm tr}(\bm{Q}\bm{R}) {\rm tr}(\bm{S}\bm{T}') 
+ {\rm tr}(\bm{Q}\bm{R}') {\rm tr}(\bm{S}\bm{T}') 
+ {\rm tr}(\bm{Q}\bm{S}) {\rm tr}(\bm{R}\bm{T}) 
+ {\rm tr}(\bm{Q}\bm{S}') {\rm tr}(\bm{R}\bm{T}) \\
&& + {\rm tr}(\bm{Q}\bm{S}) {\rm tr}(\bm{R}\bm{T}') 
+ {\rm tr}(\bm{Q}\bm{S}') {\rm tr}(\bm{R}\bm{T}') 
+ {\rm tr}(\bm{Q}\bm{T}) {\rm tr}(\bm{R}\bm{S}) 
+ {\rm tr}(\bm{Q}\bm{T}') {\rm tr}(\bm{R}\bm{S}) 
+ {\rm tr}(\bm{Q}\bm{T}) {\rm tr}(\bm{R}\bm{S}') 
+ {\rm tr}(\bm{Q}\bm{T}') {\rm tr}(\bm{R}\bm{S}') \\
&& + {\rm tr}(\bm{Q}) {\rm tr}(\bm{R}\bm{S}\bm{T}) 
+ {\rm tr}(\bm{Q}) {\rm tr}(\bm{R}\bm{S}'\bm{T}) 
+ {\rm tr}(\bm{Q}) {\rm tr}(\bm{R}\bm{S}\bm{T}') 
+ {\rm tr}(\bm{Q}) {\rm tr}(\bm{R}\bm{S}'\bm{T}') 
+ {\rm tr}(\bm{Q}) {\rm tr}(\bm{R}\bm{T}\bm{S}) 
+ {\rm tr}(\bm{Q}) {\rm tr}(\bm{R}\bm{T}'\bm{S}) \\
&& + {\rm tr}(\bm{Q}) {\rm tr}(\bm{R}\bm{T}\bm{S}') 
+ {\rm tr}(\bm{Q}) {\rm tr}(\bm{R}\bm{T}'\bm{S}') 
+ {\rm tr}(\bm{R}) {\rm tr}(\bm{Q}\bm{S}\bm{T}) 
+ {\rm tr}(\bm{R}) {\rm tr}(\bm{Q}\bm{S}'\bm{T}) 
+ {\rm tr}(\bm{R}) {\rm tr}(\bm{Q}\bm{S}\bm{T}') 
+ {\rm tr}(\bm{R}) {\rm tr}(\bm{Q}\bm{S}'\bm{T}') \\
&& + {\rm tr}(\bm{R}) {\rm tr}(\bm{Q}\bm{T}\bm{S}) 
+ {\rm tr}(\bm{R}) {\rm tr}(\bm{Q}\bm{T}'\bm{S}) 
+ {\rm tr}(\bm{R}) {\rm tr}(\bm{Q}\bm{T}\bm{S}') 
+ {\rm tr}(\bm{R}) {\rm tr}(\bm{Q}\bm{T}'\bm{S}') 
+ {\rm tr}(\bm{S}) {\rm tr}(\bm{Q}\bm{R}\bm{T}) 
+ {\rm tr}(\bm{S}) {\rm tr}(\bm{Q}\bm{R}'\bm{T}) \\
&& + {\rm tr}(\bm{S}) {\rm tr}(\bm{Q}\bm{R}\bm{T}') 
+ {\rm tr}(\bm{S}) {\rm tr}(\bm{Q}\bm{R}'\bm{T}') 
+ {\rm tr}(\bm{S}) {\rm tr}(\bm{Q}\bm{T}\bm{R}) 
+ {\rm tr}(\bm{S}) {\rm tr}(\bm{Q}\bm{T}'\bm{R}) 
+ {\rm tr}(\bm{S}) {\rm tr}(\bm{Q}\bm{T}\bm{R}') 
+ {\rm tr}(\bm{S}) {\rm tr}(\bm{Q}\bm{T}'\bm{R}') \\
&& + {\rm tr}(\bm{T}) {\rm tr}(\bm{Q}\bm{R}\bm{S}) 
+ {\rm tr}(\bm{T}) {\rm tr}(\bm{Q}\bm{R}'\bm{S}) 
+ {\rm tr}(\bm{T}) {\rm tr}(\bm{Q}\bm{R}\bm{S}') 
+ {\rm tr}(\bm{T}) {\rm tr}(\bm{Q}\bm{R}'\bm{S}') 
+ {\rm tr}(\bm{T}) {\rm tr}(\bm{Q}\bm{S}\bm{R}) 
+ {\rm tr}(\bm{T}) {\rm tr}(\bm{Q}\bm{S}'\bm{R}) \\
&& + {\rm tr}(\bm{T}) {\rm tr}(\bm{Q}\bm{S}\bm{R}') 
+ {\rm tr}(\bm{T}) {\rm tr}(\bm{Q}\bm{S}'\bm{R}') \\
&& + {\rm tr}(\bm{Q}\bm{R}\bm{S}\bm{T}) 
+ {\rm tr}(\bm{Q}\bm{R}'\bm{S}\bm{T}) 
+ {\rm tr}(\bm{Q}\bm{R}\bm{S}'\bm{T}) 
+ {\rm tr}(\bm{Q}\bm{R}\bm{S}\bm{T}') 
+ {\rm tr}(\bm{Q}\bm{R}'\bm{S}'\bm{T}) 
+ {\rm tr}(\bm{Q}\bm{R}'\bm{S}\bm{T}') 
+ {\rm tr}(\bm{Q}\bm{R}\bm{S}'\bm{T}') 
+ {\rm tr}(\bm{Q}\bm{R}'\bm{S}'\bm{T}') \\
&& + {\rm tr}(\bm{Q}\bm{R}\bm{T}\bm{S}) 
+ {\rm tr}(\bm{Q}\bm{R}'\bm{T}\bm{S}) 
+ {\rm tr}(\bm{Q}\bm{R}\bm{T}'\bm{S}) 
+ {\rm tr}(\bm{Q}\bm{R}\bm{T}\bm{S}') 
+ {\rm tr}(\bm{Q}\bm{R}'\bm{T}'\bm{S}) 
+ {\rm tr}(\bm{Q}\bm{R}'\bm{T}\bm{S}') 
+ {\rm tr}(\bm{Q}\bm{R}\bm{T}'\bm{S}') 
+ {\rm tr}(\bm{Q}\bm{R}'\bm{T}'\bm{S}') \\
&& + {\rm tr}(\bm{Q}\bm{S}\bm{R}\bm{T}) 
+ {\rm tr}(\bm{Q}\bm{S}'\bm{R}\bm{T}) 
+ {\rm tr}(\bm{Q}\bm{S}\bm{R}'\bm{T}) 
+ {\rm tr}(\bm{Q}\bm{S}\bm{R}\bm{T}') 
+ {\rm tr}(\bm{Q}\bm{S}'\bm{R}'\bm{T}) 
+ {\rm tr}(\bm{Q}\bm{S}'\bm{R}\bm{T}') 
+ {\rm tr}(\bm{Q}\bm{S}\bm{R}'\bm{T}') 
+ {\rm tr}(\bm{Q}\bm{S}'\bm{R}'\bm{T}') \\
&& + {\rm tr}(\bm{Q}\bm{S}\bm{T}\bm{R}) 
+ {\rm tr}(\bm{Q}\bm{S}'\bm{T}\bm{R}) 
+ {\rm tr}(\bm{Q}\bm{S}\bm{T}'\bm{R}) 
+ {\rm tr}(\bm{Q}\bm{S}\bm{T}\bm{R}') 
+ {\rm tr}(\bm{Q}\bm{S}'\bm{T}'\bm{R}) 
+ {\rm tr}(\bm{Q}\bm{S}'\bm{T}\bm{R}') 
+ {\rm tr}(\bm{Q}\bm{S}\bm{T}'\bm{R}') 
+ {\rm tr}(\bm{Q}\bm{S}'\bm{T}'\bm{R}') \\
&& + {\rm tr}(\bm{Q}\bm{T}\bm{R}\bm{S}) 
+ {\rm tr}(\bm{Q}\bm{T}'\bm{R}\bm{S}) 
+ {\rm tr}(\bm{Q}\bm{T}\bm{R}'\bm{S}) 
+ {\rm tr}(\bm{Q}\bm{T}\bm{R}\bm{S}') 
+ {\rm tr}(\bm{Q}\bm{T}'\bm{R}'\bm{S}) 
+ {\rm tr}(\bm{Q}\bm{T}'\bm{R}\bm{S}') 
+ {\rm tr}(\bm{Q}\bm{T}\bm{R}'\bm{S}') 
+ {\rm tr}(\bm{Q}\bm{T}'\bm{R}'\bm{S}') \\
&& + {\rm tr}(\bm{Q}\bm{T}\bm{S}\bm{R}) 
+ {\rm tr}(\bm{Q}\bm{T}'\bm{S}\bm{R}) 
+ {\rm tr}(\bm{Q}\bm{T}\bm{S}'\bm{R}) 
+ {\rm tr}(\bm{Q}\bm{T}\bm{S}\bm{R}') 
+ {\rm tr}(\bm{Q}\bm{T}'\bm{S}'\bm{R}) 
+ {\rm tr}(\bm{Q}\bm{T}'\bm{S}\bm{R}') 
+ {\rm tr}(\bm{Q}\bm{T}\bm{S}'\bm{R}') 
+ {\rm tr}(\bm{Q}\bm{T}'\bm{S}'\bm{R}').
\end{eqnarray*}
We can prove the rest of the assertions in similar ways, so we omit their details.
(ii) Under the assumptions, the assetions of (ii) can be proven by similar arguments to the proof of (i).
This completes the proof.
\qed

\section{Proofs for Section 2}\label{sec:proof}
\subsection{Proof of Proposition~\ref{prop21}}

The expectation of 
$\{{\rm tr}\left( \bm{T}_{a} \bm{T}_{b} \bm{T}_{c} \bm{T}_{d} \right)^2\}$ can be calculated as follows:
\begin{eqnarray*}
\lefteqn{ {\rm E} \left[ \left\{ {\rm tr}\left( 
\bm{T}_{a} \bm{T}_{b} \bm{T}_{c} \bm{T}_{d} \right) \right\}^2 \right] 
= {\rm E} \left[ {\rm tr}\left( 
\bm{T}_{a} \bm{T}_{b} \bm{T}_{c} \bm{T}_{d} \right) 
{\rm tr}\left( 
\bm{T}_{a} \bm{T}_{b} \bm{T}_{c} \bm{T}_{d} \right) \right] 
= {\rm E} \left[ {\rm tr}\left( 
{\rm tr}\left( \bm{T}_{a} \bm{T}_{b} \bm{T}_{c} \bm{T}_{d} \right)
\bm{T}_{a} \bm{T}_{b} \bm{T}_{c} \bm{T}_{d} \right) 
 \right] 
}\\
&=& n_a {\rm E} \left[ {\rm tr}\left( 
\left\{ 
\bm{\Sigma}_a\bm{T}_{b} \bm{T}_{c} \bm{T}_{d}\bm{\Sigma}_a 
+ \bm{\Sigma}_a\bm{T}_{d} \bm{T}_{c} \bm{T}_{b}\bm{\Sigma}_a 
+ n_a {\rm tr}(\bm{\Sigma}_a\bm{T}_{b} \bm{T}_{c} \bm{T}_{d}) \bm{\Sigma}_a 
\right\}
\bm{T}_{b} \bm{T}_{c} \bm{T}_{d} \right) 
 \right] \\
&=& n_a {\rm E} \left[ {\rm tr}\left( 
\bm{T}_{b} \bm{T}_{c} \bm{T}_{d}\bm{\Sigma}_a 
\bm{T}_{b} \bm{T}_{c} \bm{T}_{d}\bm{\Sigma}_a \right) 
 \right]
+n_a {\rm E} \left[ {\rm tr}\left( 
\bm{T}_{b}\bm{\Sigma}_a 
\bm{T}_{b} \bm{T}_{c} \bm{T}_{d}\bm{\Sigma}_a\bm{T}_{d} \bm{T}_{c}  \right) 
 \right]
+n_a^2 {\rm E} \left[ {\rm tr}\left( 
{\rm tr}(\bm{T}_{b} \bm{T}_{c} \bm{T}_{d}\bm{\Sigma}_a) 
\bm{T}_{b} \bm{T}_{c} \bm{T}_{d}\bm{\Sigma}_a \right) 
 \right] \\
&=& n_a n_b {\rm E} \left[ {\rm tr}\left( 
\left\{ 
n_b \bm{\Sigma}_b \bm{T}_{c} \bm{T}_{d}\bm{\Sigma}_a \bm{\Sigma}_b 
+ \bm{\Sigma}_b \bm{\Sigma}_a\bm{T}_{d}\bm{T}_{c} \bm{\Sigma}_b 
+ {\rm tr}(\bm{\Sigma}_b \bm{T}_{c} \bm{T}_{d}\bm{\Sigma}_a) \bm{\Sigma}_b
\right\}
 \bm{T}_{c} \bm{T}_{d}\bm{\Sigma}_a \right) 
 \right] \\
&& +n_a n_b {\rm E} \left[ {\rm tr}\left( 
\left\{ 
(n_b+1) \bm{\Sigma}_b \bm{\Sigma}_a \bm{\Sigma}_b 
+ {\rm tr}(\bm{\Sigma}_b \bm{\Sigma}_a) \bm{\Sigma}_b
\right\} 
\bm{T}_{c} \bm{T}_{d}\bm{\Sigma}_a\bm{T}_{d} \bm{T}_{c}  \right) 
 \right] \\
&& +n_a^2 n_b {\rm E} \left[ {\rm tr}\left( 
\left\{ 
\bm{\Sigma}_b\bm{T}_{c} \bm{T}_{d}\bm{\Sigma}_a\bm{\Sigma}_b 
+ \bm{\Sigma}_b\bm{\Sigma}_a\bm{T}_{d}\bm{T}_{c} \bm{\Sigma}_b 
+ n_b {\rm tr}(\bm{\Sigma}_b\bm{T}_{c} \bm{T}_{d}\bm{\Sigma}_a) \bm{\Sigma}_b
\right\}
\bm{T}_{c} \bm{T}_{d}\bm{\Sigma}_a \right) 
 \right] \\
&=& n_a n_b^2 {\rm E} \left[ {\rm tr}\left( 
\bm{T}_{c} \bm{T}_{d}\bm{\Sigma}_a \bm{\Sigma}_b 
\bm{T}_{c} \bm{T}_{d}\bm{\Sigma}_a \bm{\Sigma}_b \right) 
 \right]
+ n_a n_b {\rm E} \left[ {\rm tr}\left( 
\bm{T}_{c} \bm{\Sigma}_b \bm{T}_{c} 
\bm{T}_{d}\bm{\Sigma}_a\bm{\Sigma}_b \bm{\Sigma}_a\bm{T}_{d} \right) 
 \right]
+ n_a n_b {\rm E} \left[ {\rm tr}\left( 
{\rm tr}(\bm{T}_{c} \bm{T}_{d}\bm{\Sigma}_a\bm{\Sigma}_b) 
\bm{T}_{c} \bm{T}_{d}\bm{\Sigma}_a\bm{\Sigma}_b \right) 
 \right] \\
&& + n_a n_b (n_b +1) {\rm E} \left[ {\rm tr}\left( 
\bm{T}_{c} \bm{\Sigma}_b \bm{\Sigma}_a \bm{\Sigma}_b 
\bm{T}_{c} \bm{T}_{d}\bm{\Sigma}_a\bm{T}_{d} \right) 
 \right]
+ n_a n_b {\rm tr}(\bm{\Sigma}_a\bm{\Sigma}_b) {\rm E} \left[ {\rm tr}\left( 
\bm{T}_{c} \bm{\Sigma}_b \bm{T}_{c} 
\bm{T}_{d}\bm{\Sigma}_a\bm{T}_{d}   \right) 
 \right] \\
&& + n_a^2 n_b {\rm E} \left[ {\rm tr}\left( 
\bm{T}_{c} \bm{T}_{d}\bm{\Sigma}_a\bm{\Sigma}_b 
\bm{T}_{c} \bm{T}_{d}\bm{\Sigma}_a\bm{\Sigma}_b \right) 
 \right]
+ n_a^2 n_b {\rm E} \left[ {\rm tr}\left( 
\bm{T}_{c} \bm{\Sigma}_b \bm{T}_{c} 
\bm{T}_{d}\bm{\Sigma}_a\bm{\Sigma}_b\bm{\Sigma}_a\bm{T}_{d} \right) 
 \right]
+ n_a^2 n_b^2 {\rm E} \left[ {\rm tr}\left( 
{\rm tr}(\bm{T}_{c} \bm{T}_{d}\bm{\Sigma}_a\bm{\Sigma}_b) 
\bm{T}_{c} \bm{T}_{d}\bm{\Sigma}_a\bm{\Sigma}_b \right) 
 \right] \\
&=& n_a n_b (n_a + n_b ) {\rm E} \left[ {\rm tr}\left( 
\bm{T}_{c} \bm{T}_{d}\bm{\Sigma}_a \bm{\Sigma}_b 
\bm{T}_{c} \bm{T}_{d}\bm{\Sigma}_a \bm{\Sigma}_b \right) 
 \right]
+ n_a n_b (n_a +1) {\rm E} \left[ {\rm tr}\left( 
\bm{T}_{c} \bm{\Sigma}_b \bm{T}_{c} 
\bm{T}_{d}\bm{\Sigma}_a\bm{\Sigma}_b \bm{\Sigma}_a\bm{T}_{d} \right) 
 \right] \\
&& 
+ n_a n_b (n_b +1) {\rm E} \left[ {\rm tr}\left( 
\bm{T}_{c} \bm{\Sigma}_b \bm{\Sigma}_a \bm{\Sigma}_b 
\bm{T}_{c} \bm{T}_{d}\bm{\Sigma}_a\bm{T}_{d} \right) 
 \right]
+ n_a n_b (n_a n_b+1) {\rm E} \left[ {\rm tr}\left( 
{\rm tr}(\bm{T}_{c} \bm{T}_{d}\bm{\Sigma}_a\bm{\Sigma}_b) 
\bm{T}_{c} \bm{T}_{d}\bm{\Sigma}_a\bm{\Sigma}_b \right) 
 \right] 
\\
&& + n_a n_b {\rm tr}(\bm{\Sigma}_a\bm{\Sigma}_b) {\rm E} \left[ {\rm tr}\left( 
\bm{T}_{c} \bm{\Sigma}_b \bm{T}_{c} 
\bm{T}_{d}\bm{\Sigma}_a\bm{T}_{d}   \right) 
 \right] \\
&=& n_a n_b n_c (n_a + n_b) {\rm E} \left[ {\rm tr}\left( 
\left\{
n_c \bm{\Sigma}_c \bm{T}_{d}\bm{\Sigma}_a \bm{\Sigma}_b \bm{\Sigma}_c 
+ \bm{\Sigma}_c \bm{\Sigma}_b\bm{\Sigma}_a \bm{T}_{d} \bm{\Sigma}_c 
+ {\rm tr}(\bm{\Sigma}_c \bm{T}_{d}\bm{\Sigma}_a \bm{\Sigma}_b) \bm{\Sigma}_c
\right\}
 \bm{T}_{d}\bm{\Sigma}_a \bm{\Sigma}_b \right) 
 \right] \\
&& +n_a n_b n_c(n_a +1) {\rm E} \left[ {\rm tr}\left( 
\left\{ 
(n_c +1) \bm{\Sigma}_c \bm{\Sigma}_b \bm{\Sigma}_c 
+ {\rm tr}(\bm{\Sigma}_c \bm{\Sigma}_b) \bm{\Sigma}_c
\right\}
\bm{T}_{d}\bm{\Sigma}_a\bm{\Sigma}_b \bm{\Sigma}_a\bm{T}_{d} \right) 
 \right] \\
&& + n_a n_b n_c(n_b+1) {\rm E} \left[ {\rm tr}\left( 
\left\{ 
(n_c+1) \bm{\Sigma}_c \bm{\Sigma}_b \bm{\Sigma}_a \bm{\Sigma}_b \bm{\Sigma}_c 
+ {\rm tr}(\bm{\Sigma}_c \bm{\Sigma}_b \bm{\Sigma}_a \bm{\Sigma}_b) \bm{\Sigma}_c
\right\} 
\bm{T}_{d}\bm{\Sigma}_a\bm{T}_{d} \right) 
 \right]
 \\
&& +n_a n_b n_c(n_a n_b+1) {\rm E} \left[ {\rm tr}\left( 
\left\{ 
\bm{\Sigma}_c\bm{T}_{d}\bm{\Sigma}_a\bm{\Sigma}_b\bm{\Sigma}_c 
+ \bm{\Sigma}_c\bm{\Sigma}_b\bm{\Sigma}_a\bm{T}_{d}\bm{\Sigma}_c 
+ n_c{\rm tr}(\bm{\Sigma}_c\bm{T}_{d}\bm{\Sigma}_a\bm{\Sigma}_b) \bm{\Sigma}_c
\right\}
\bm{T}_{d}\bm{\Sigma}_a\bm{\Sigma}_b \right) 
 \right] \\
&& + n_a n_b n_c {\rm tr}(\bm{\Sigma}_a\bm{\Sigma}_b) {\rm E} \left[ {\rm tr}\left( 
\left\{
(n_c +1) \bm{\Sigma}_c \bm{\Sigma}_b \bm{\Sigma}_c 
+ {\rm tr}(\bm{\Sigma}_c \bm{\Sigma}_b) \bm{\Sigma}_c
\right\}
\bm{T}_{d}\bm{\Sigma}_a\bm{T}_{d}   \right) 
 \right] \\
&=& n_a n_b n_c(n_a n_b+ n_a n_c+ n_b n_c+1) {\rm E} \left[ {\rm tr}\left( 
\bm{T}_{d}\bm{\Sigma}_a \bm{\Sigma}_b \bm{\Sigma}_c 
\bm{T}_{d}\bm{\Sigma}_a \bm{\Sigma}_b \bm{\Sigma}_c \right) 
 \right]
+n_a n_b n_c(n_a+1)(n_b+1) {\rm E} \left[ {\rm tr}\left( 
\bm{T}_{d} \bm{\Sigma}_c \bm{T}_{d}
\bm{\Sigma}_a \bm{\Sigma}_b\bm{\Sigma}_c \bm{\Sigma}_b\bm{\Sigma}_a  \right) 
 \right] \\
&& +n_a n_b n_c(n_a+1)(n_c+1) {\rm E} \left[ {\rm tr}\left( 
\bm{T}_{d}\bm{\Sigma}_a\bm{\Sigma}_b \bm{\Sigma}_a
\bm{T}_{d}\bm{\Sigma}_c \bm{\Sigma}_b \bm{\Sigma}_c \right) 
 \right]
+n_a n_b n_c(n_b+1)(n_c+1) {\rm E} \left[ {\rm tr}\left( 
\bm{T}_{d}\bm{\Sigma}_a\bm{T}_{d}
 \bm{\Sigma}_c \bm{\Sigma}_b \bm{\Sigma}_a \bm{\Sigma}_b \bm{\Sigma}_c  \right) 
 \right]
\\
&& 
+n_a n_b n_c(n_a n_b n_c+n_a+n_b+n_c) {\rm E} \left[ {\rm tr}\left( 
{\rm tr}(\bm{T}_{d}\bm{\Sigma}_a \bm{\Sigma}_b\bm{\Sigma}_c) 
 \bm{T}_{d}\bm{\Sigma}_a \bm{\Sigma}_b\bm{\Sigma}_c \right) 
 \right] 
+n_a n_b n_c(n_a+1){\rm tr}(\bm{\Sigma}_b\bm{\Sigma}_c) {\rm E} \left[ {\rm tr}\left( 
\bm{T}_{d} \bm{\Sigma}_c\bm{T}_{d}
\bm{\Sigma}_a\bm{\Sigma}_b \bm{\Sigma}_a \right) 
 \right] \\
&& 
+n_a n_b n_c(n_b+1){\rm tr}(\bm{\Sigma}_a \bm{\Sigma}_b\bm{\Sigma}_c\bm{\Sigma}_b) {\rm E} \left[ {\rm tr}\left( 
\bm{T}_{d}\bm{\Sigma}_a\bm{T}_{d} \bm{\Sigma}_c \right) 
 \right]
+n_a n_b n_c(n_c+1){\rm tr}(\bm{\Sigma}_a\bm{\Sigma}_b) {\rm E} \left[ {\rm tr}\left( 
\bm{T}_{d}\bm{\Sigma}_a\bm{T}_{d}\bm{\Sigma}_c \bm{\Sigma}_b \bm{\Sigma}_c \right) 
 \right]
 \\
&& + n_a n_b n_c{\rm tr}(\bm{\Sigma}_a\bm{\Sigma}_b){\rm tr}(\bm{\Sigma}_b\bm{\Sigma}_c) {\rm E} \left[ {\rm tr}\left( 
\bm{T}_{d}\bm{\Sigma}_a\bm{T}_{d}\bm{\Sigma}_c \right) 
 \right] \\
&=& n_a n_b n_c n_d(n_a n_b+ n_a n_c+ n_b n_c+1){\rm tr}\left( 
\left\{ 
n_d \bm{\Sigma}_d \bm{\Sigma}_a \bm{\Sigma}_b \bm{\Sigma}_c \bm{\Sigma}_d 
+ \bm{\Sigma}_d \bm{\Sigma}_c \bm{\Sigma}_b \bm{\Sigma}_a \bm{\Sigma}_d 
+ {\rm tr}(\bm{\Sigma}_d \bm{\Sigma}_a \bm{\Sigma}_b \bm{\Sigma}_c) \bm{\Sigma}_d
\right\}
\bm{\Sigma}_a \bm{\Sigma}_b \bm{\Sigma}_c \right) \\
&&
+n_a n_b n_c n_d (n_a+1)(n_b+1){\rm tr}\left( 
\left\{ 
(n_d+1) \bm{\Sigma}_d \bm{\Sigma}_c \bm{\Sigma}_d 
+ {\rm tr}(\bm{\Sigma}_d \bm{\Sigma}_c) \bm{\Sigma}_d
\right\}
\bm{\Sigma}_a \bm{\Sigma}_b\bm{\Sigma}_c \bm{\Sigma}_b\bm{\Sigma}_a  \right) 
 \\
&& + n_a n_b n_c n_d (n_a+1)(n_c+1) {\rm tr}\left( 
\left\{ 
(n_d+1) \bm{\Sigma}_d \bm{\Sigma}_a\bm{\Sigma}_b \bm{\Sigma}_a \bm{\Sigma}_d 
+ {\rm tr}(\bm{\Sigma}_d \bm{\Sigma}_a\bm{\Sigma}_b \bm{\Sigma}_a) \bm{\Sigma}_d
\right\}
\bm{\Sigma}_c \bm{\Sigma}_b \bm{\Sigma}_c \right) 
 \\
&&
+ n_a n_b n_c n_d (n_b+1)(n_c+1) {\rm tr}\left( 
\left\{
(n_d+1) \bm{\Sigma}_d \bm{\Sigma}_a \bm{\Sigma}_d 
+ {\rm tr}(\bm{\Sigma}_d \bm{\Sigma}_a) \bm{\Sigma}_d
\right\}
 \bm{\Sigma}_c \bm{\Sigma}_b \bm{\Sigma}_a \bm{\Sigma}_b \bm{\Sigma}_c  \right) 
\\
&& 
+ n_a n_b n_c n_d (n_a n_b n_c+ n_a+ n_b+ n_c) {\rm tr}\left( 
\left\{ 
\bm{\Sigma}_d\bm{\Sigma}_a \bm{\Sigma}_b\bm{\Sigma}_c\bm{\Sigma}_d 
+ \bm{\Sigma}_d\bm{\Sigma}_c \bm{\Sigma}_b\bm{\Sigma}_a\bm{\Sigma}_d 
+ n_d {\rm tr}(\bm{\Sigma}_d\bm{\Sigma}_a \bm{\Sigma}_b\bm{\Sigma}_c) \bm{\Sigma}_d
\right\}
\bm{\Sigma}_a \bm{\Sigma}_b\bm{\Sigma}_c \right) 
 \\
&&
+ n_a n_b n_c n_d (n_a+1){\rm tr}(\bm{\Sigma}_b\bm{\Sigma}_c) {\rm tr}\left( 
\left\{ 
(n_d+1) \bm{\Sigma}_d \bm{\Sigma}_c \bm{\Sigma}_d 
+ {\rm tr}(\bm{\Sigma}_d \bm{\Sigma}_c) \bm{\Sigma}_d
\right\}
\bm{\Sigma}_a\bm{\Sigma}_b \bm{\Sigma}_a \right) 
\\
&& 
+n_a n_b n_c n_d(n_b+1){\rm tr}(\bm{\Sigma}_a \bm{\Sigma}_b\bm{\Sigma}_c\bm{\Sigma}_b)
{\rm tr}\left( 
\left\{
(n_d+1) \bm{\Sigma}_d \bm{\Sigma}_a \bm{\Sigma}_d 
+ {\rm tr}(\bm{\Sigma}_d \bm{\Sigma}_a) \bm{\Sigma}_d
\right\}
\bm{\Sigma}_c \right) \\
&&
+n_a n_b n_c n_d(n_c+1){\rm tr}(\bm{\Sigma}_a\bm{\Sigma}_b)
{\rm tr}\left( 
\left\{
(n_d+1) \bm{\Sigma}_d \bm{\Sigma}_a \bm{\Sigma}_d 
+ {\rm tr}(\bm{\Sigma}_d \bm{\Sigma}_a) \bm{\Sigma}_d
\right\}
\bm{\Sigma}_c \bm{\Sigma}_b \bm{\Sigma}_c \right) 
 \\
&& + n_a n_b n_c n_d{\rm tr}(\bm{\Sigma}_a\bm{\Sigma}_b){\rm tr}(\bm{\Sigma}_b\bm{\Sigma}_c)
{\rm tr}\left( 
\left\{
(n_d+1) \bm{\Sigma}_d \bm{\Sigma}_a \bm{\Sigma}_d 
+ {\rm tr}(\bm{\Sigma}_d \bm{\Sigma}_a) \bm{\Sigma}_d
\right\}
\bm{\Sigma}_c \right) 
\\
&=& n_a n_b n_c n_d \biggl[ 
(n_a n_b+n_a n_c+n_b n_c+1) n_d {\rm tr}\left( 
\bm{\Sigma}_a \bm{\Sigma}_b \bm{\Sigma}_c \bm{\Sigma}_d 
\bm{\Sigma}_a \bm{\Sigma}_b \bm{\Sigma}_c \bm{\Sigma}_d \right)
+(n_a n_b+n_a n_c+ n_b n_c+1){\rm tr}\left( 
\bm{\Sigma}_a \bm{\Sigma}_b \bm{\Sigma}_c \bm{\Sigma}_d 
\bm{\Sigma}_c \bm{\Sigma}_b \bm{\Sigma}_a \bm{\Sigma}_d \right) \\
&&
+(n_a n_b+n_a n_c+ n_b n_c+1)
\left\{ {\rm tr}(\bm{\Sigma}_a \bm{\Sigma}_b \bm{\Sigma}_c \bm{\Sigma}_d) \right\}^2 
\\
&&
+(n_a+1)(n_b+1)(n_d+1){\rm tr}\left( 
\bm{\Sigma}_a \bm{\Sigma}_b\bm{\Sigma}_c \bm{\Sigma}_b
\bm{\Sigma}_a\bm{\Sigma}_d \bm{\Sigma}_c \bm{\Sigma}_d \right) 
+(n_a+1)(n_b+1){\rm tr}(\bm{\Sigma}_c \bm{\Sigma}_d)
{\rm tr}\left( \bm{\Sigma}_a \bm{\Sigma}_b\bm{\Sigma}_c \bm{\Sigma}_b\bm{\Sigma}_a \bm{\Sigma}_d \right) 
 \\
&& +(n_a+1)(n_c+1)(n_d+1) {\rm tr}\left( 
\bm{\Sigma}_a\bm{\Sigma}_b \bm{\Sigma}_a \bm{\Sigma}_d 
\bm{\Sigma}_c \bm{\Sigma}_b \bm{\Sigma}_c \bm{\Sigma}_d \right) 
+(n_a+1)(n_c+1) {\rm tr}(\bm{\Sigma}_a\bm{\Sigma}_b \bm{\Sigma}_a\bm{\Sigma}_d )
{\rm tr}( \bm{\Sigma}_b \bm{\Sigma}_c \bm{\Sigma}_d \bm{\Sigma}_c )  \\
&&
+(n_b+1)(n_c+1)(n_d+1) {\rm tr}\left( 
\bm{\Sigma}_a \bm{\Sigma}_b \bm{\Sigma}_c \bm{\Sigma}_d 
\bm{\Sigma}_a \bm{\Sigma}_d \bm{\Sigma}_c \bm{\Sigma}_b   \right) 
+(n_b+1)(n_c+1){\rm tr}(\bm{\Sigma}_a \bm{\Sigma}_d) {\rm tr}\left( 
 \bm{\Sigma}_a \bm{\Sigma}_b \bm{\Sigma}_c \bm{\Sigma}_d \bm{\Sigma}_c \bm{\Sigma}_b \right) 
\\
&& 
+(n_a n_b n_c+n_a+n_b+n_c) {\rm tr}\left( 
\bm{\Sigma}_a \bm{\Sigma}_b\bm{\Sigma}_c\bm{\Sigma}_d 
\bm{\Sigma}_a \bm{\Sigma}_b\bm{\Sigma}_c\bm{\Sigma}_d \right) 
+(n_a n_b n_c+n_a+n_b+n_c) {\rm tr}\left( 
\bm{\Sigma}_a \bm{\Sigma}_b\bm{\Sigma}_c\bm{\Sigma}_d
\bm{\Sigma}_c \bm{\Sigma}_b\bm{\Sigma}_a\bm{\Sigma}_d 
 \right) \\
&& 
+(n_a n_b n_c+n_a+n_b+n_c)n_d 
\left\{ {\rm tr}(\bm{\Sigma}_a \bm{\Sigma}_b\bm{\Sigma}_c\bm{\Sigma}_d) \right\}^2 
 \\
&&
+(n_a +1)(n_d +1){\rm tr}(\bm{\Sigma}_b\bm{\Sigma}_c)
{\rm tr}\left( \bm{\Sigma}_a\bm{\Sigma}_b \bm{\Sigma}_a \bm{\Sigma}_d \bm{\Sigma}_c \bm{\Sigma}_d \right) 
+(n_a +1){\rm tr}(\bm{\Sigma}_b\bm{\Sigma}_c) {\rm tr}(\bm{\Sigma}_c \bm{\Sigma}_d) 
{\rm tr}( \bm{\Sigma}_a\bm{\Sigma}_b \bm{\Sigma}_a  \bm{\Sigma}_d ) 
\\
&& 
+(n_b +1)(n_d +1){\rm tr}(\bm{\Sigma}_a \bm{\Sigma}_b\bm{\Sigma}_c\bm{\Sigma}_b)
{\rm tr}(\bm{\Sigma}_a \bm{\Sigma}_d \bm{\Sigma}_c \bm{\Sigma}_d)
+(n_b +1){\rm tr}(\bm{\Sigma}_a \bm{\Sigma}_d) {\rm tr}(\bm{\Sigma}_c\bm{\Sigma}_d) 
{\rm tr}(\bm{\Sigma}_a \bm{\Sigma}_b\bm{\Sigma}_c\bm{\Sigma}_b)
 \\
&&
+(n_c +1)(n_d +1){\rm tr}(\bm{\Sigma}_a\bm{\Sigma}_b)
{\rm tr}\left( \bm{\Sigma}_a \bm{\Sigma}_d \bm{\Sigma}_c \bm{\Sigma}_b \bm{\Sigma}_c \bm{\Sigma}_d \right) 
+(n_c +1){\rm tr}(\bm{\Sigma}_a\bm{\Sigma}_b) {\rm tr}(\bm{\Sigma}_a \bm{\Sigma}_d) 
{\rm tr}( \bm{\Sigma}_b \bm{\Sigma}_c \bm{\Sigma}_d \bm{\Sigma}_c ) 
 \\
&& +(n_d+1){\rm tr}(\bm{\Sigma}_a\bm{\Sigma}_b){\rm tr}(\bm{\Sigma}_b\bm{\Sigma}_c)
{\rm tr}( \bm{\Sigma}_a \bm{\Sigma}_d \bm{\Sigma}_c \bm{\Sigma}_d ) 
+{\rm tr}(\bm{\Sigma}_a\bm{\Sigma}_b){\rm tr}(\bm{\Sigma}_a \bm{\Sigma}_d)
{\rm tr}(\bm{\Sigma}_b\bm{\Sigma}_c){\rm tr}( \bm{\Sigma}_c \bm{\Sigma}_d)
 \biggr]
\\
&=& n_a n_b n_c n_d \biggl[ 
(n_a n_b n_c n_d+n_a n_b+n_a n_c+n_a n_d+n_b n_c+n_b n_d+n_c n_d+1)
\left\{ {\rm tr}(\bm{\Sigma}_a \bm{\Sigma}_b \bm{\Sigma}_c \bm{\Sigma}_d) \right\}^2
\\
&&
+(n_a n_b n_c+n_a n_b n_d+n_a n_c n_d+n_b n_c n_d+n_a+n_b+n_c+n_d){\rm tr}\left( 
\bm{\Sigma}_a \bm{\Sigma}_b \bm{\Sigma}_c \bm{\Sigma}_d 
\bm{\Sigma}_a \bm{\Sigma}_b \bm{\Sigma}_c \bm{\Sigma}_d \right) \\
&&
+(n_a+1)(n_b+1)(n_c+1){\rm tr}\left( 
\bm{\Sigma}_a \bm{\Sigma}_b \bm{\Sigma}_c \bm{\Sigma}_d 
\bm{\Sigma}_c \bm{\Sigma}_b \bm{\Sigma}_a \bm{\Sigma}_d \right) 
+(n_a+1)(n_b+1)(n_d+1){\rm tr}\left( 
\bm{\Sigma}_a \bm{\Sigma}_b\bm{\Sigma}_c \bm{\Sigma}_b
\bm{\Sigma}_a\bm{\Sigma}_d \bm{\Sigma}_c \bm{\Sigma}_d \right) \\
&&
+(n_a+1)(n_c+1)(n_d+1) {\rm tr}\left( 
\bm{\Sigma}_a\bm{\Sigma}_b \bm{\Sigma}_a \bm{\Sigma}_d 
\bm{\Sigma}_c \bm{\Sigma}_b \bm{\Sigma}_c \bm{\Sigma}_d \right) 
+(n_b+1)(n_c+1)(n_d+1) {\rm tr}\left( 
\bm{\Sigma}_a \bm{\Sigma}_b \bm{\Sigma}_c \bm{\Sigma}_d 
\bm{\Sigma}_a \bm{\Sigma}_d \bm{\Sigma}_c \bm{\Sigma}_b   \right) 
\\
&& 
+(n_a+1)(n_b+1){\rm tr}(\bm{\Sigma}_c \bm{\Sigma}_d)
{\rm tr}\left( \bm{\Sigma}_a \bm{\Sigma}_b\bm{\Sigma}_c \bm{\Sigma}_b\bm{\Sigma}_a \bm{\Sigma}_d \right) 
+(n_a+1)(n_c+1) {\rm tr}(\bm{\Sigma}_a\bm{\Sigma}_b \bm{\Sigma}_a\bm{\Sigma}_d )
{\rm tr}( \bm{\Sigma}_b \bm{\Sigma}_c \bm{\Sigma}_d \bm{\Sigma}_c )  \\
&&
+(n_a+1)(n_d+1){\rm tr}(\bm{\Sigma}_b\bm{\Sigma}_c)
{\rm tr}\left( \bm{\Sigma}_a\bm{\Sigma}_b \bm{\Sigma}_a \bm{\Sigma}_d \bm{\Sigma}_c \bm{\Sigma}_d \right) 
+(n_b+1)(n_c+1){\rm tr}(\bm{\Sigma}_a \bm{\Sigma}_d) {\rm tr}\left( 
 \bm{\Sigma}_a \bm{\Sigma}_b \bm{\Sigma}_c \bm{\Sigma}_d \bm{\Sigma}_c \bm{\Sigma}_b \right) 
\\
&&
+(n_b+1)(n_d+1){\rm tr}(\bm{\Sigma}_a \bm{\Sigma}_b\bm{\Sigma}_c\bm{\Sigma}_b)
{\rm tr}(\bm{\Sigma}_a \bm{\Sigma}_d \bm{\Sigma}_c \bm{\Sigma}_d)
+(n_c+1)(n_d+1){\rm tr}(\bm{\Sigma}_a\bm{\Sigma}_b)
{\rm tr}\left( \bm{\Sigma}_a \bm{\Sigma}_d \bm{\Sigma}_c \bm{\Sigma}_b \bm{\Sigma}_c \bm{\Sigma}_d \right) 
\\
&& 
+(n_a+1){\rm tr}(\bm{\Sigma}_b\bm{\Sigma}_c) {\rm tr}(\bm{\Sigma}_c \bm{\Sigma}_d) 
{\rm tr}( \bm{\Sigma}_a\bm{\Sigma}_b \bm{\Sigma}_a  \bm{\Sigma}_d ) 
+(n_b+1){\rm tr}(\bm{\Sigma}_a \bm{\Sigma}_d) {\rm tr}(\bm{\Sigma}_c\bm{\Sigma}_d) 
{\rm tr}(\bm{\Sigma}_a \bm{\Sigma}_b\bm{\Sigma}_c\bm{\Sigma}_b)
 \\
&&
+(n_c+1){\rm tr}(\bm{\Sigma}_a\bm{\Sigma}_b) {\rm tr}(\bm{\Sigma}_a \bm{\Sigma}_d) 
{\rm tr}( \bm{\Sigma}_b \bm{\Sigma}_c \bm{\Sigma}_d \bm{\Sigma}_c ) 
+(n_d+1){\rm tr}(\bm{\Sigma}_a\bm{\Sigma}_b){\rm tr}(\bm{\Sigma}_b\bm{\Sigma}_c)
{\rm tr}( \bm{\Sigma}_a \bm{\Sigma}_d \bm{\Sigma}_c \bm{\Sigma}_d ) 
 \\
&& +{\rm tr}(\bm{\Sigma}_a\bm{\Sigma}_b){\rm tr}(\bm{\Sigma}_a \bm{\Sigma}_d)
{\rm tr}(\bm{\Sigma}_b\bm{\Sigma}_c){\rm tr}( \bm{\Sigma}_c \bm{\Sigma}_d)
 \biggr]
\end{eqnarray*}
Hence, the variance of ${\rm tr}\left( 
\bm{T}_{a} \bm{T}_{b} \bm{T}_{c} \bm{T}_{d} \right)$ is
\begin{eqnarray*}
\lefteqn{
{\rm V} \left[ {\rm tr}\left( 
\bm{T}_{a} \bm{T}_{b} \bm{T}_{c} \bm{T}_{d} \right) \right] 
} \\
&=& n_a n_b n_c n_d \biggl[ 
(n_a n_b+n_a n_c+n_a n_d+n_b n_c+n_b n_d+n_c n_d+1)
\left\{ {\rm tr}(\bm{\Sigma}_a \bm{\Sigma}_b \bm{\Sigma}_c \bm{\Sigma}_d) \right\}^2
\\
&&
+(n_a n_b n_c+n_a n_b n_d+n_a n_c n_d+n_b n_c n_d+n_a+n_b+n_c+n_d){\rm tr}\left( 
\bm{\Sigma}_a \bm{\Sigma}_b \bm{\Sigma}_c \bm{\Sigma}_d 
\bm{\Sigma}_a \bm{\Sigma}_b \bm{\Sigma}_c \bm{\Sigma}_d \right) \\
&&
+(n_a+1)(n_b+1)(n_c+1){\rm tr}\left( 
\bm{\Sigma}_a \bm{\Sigma}_b \bm{\Sigma}_c \bm{\Sigma}_d 
\bm{\Sigma}_c \bm{\Sigma}_b \bm{\Sigma}_a \bm{\Sigma}_d \right) 
+(n_a+1)(n_b+1)(n_d+1){\rm tr}\left( 
\bm{\Sigma}_a \bm{\Sigma}_b\bm{\Sigma}_c \bm{\Sigma}_b
\bm{\Sigma}_a\bm{\Sigma}_d \bm{\Sigma}_c \bm{\Sigma}_d \right) \\
&&
+(n_a+1)(n_c+1)(n_d+1) {\rm tr}\left( 
\bm{\Sigma}_a\bm{\Sigma}_b \bm{\Sigma}_a \bm{\Sigma}_d 
\bm{\Sigma}_c \bm{\Sigma}_b \bm{\Sigma}_c \bm{\Sigma}_d \right) 
+(n_b+1)(n_c+1)(n_d+1) {\rm tr}\left( 
\bm{\Sigma}_a \bm{\Sigma}_b \bm{\Sigma}_c \bm{\Sigma}_d 
\bm{\Sigma}_a \bm{\Sigma}_d \bm{\Sigma}_c \bm{\Sigma}_b   \right) 
\\
&& 
+(n_a+1)(n_b+1){\rm tr}(\bm{\Sigma}_c \bm{\Sigma}_d)
{\rm tr}\left( \bm{\Sigma}_a \bm{\Sigma}_b\bm{\Sigma}_c \bm{\Sigma}_b\bm{\Sigma}_a \bm{\Sigma}_d \right) 
+(n_a+1)(n_c+1) {\rm tr}(\bm{\Sigma}_a\bm{\Sigma}_b \bm{\Sigma}_a\bm{\Sigma}_d )
{\rm tr}( \bm{\Sigma}_b \bm{\Sigma}_c \bm{\Sigma}_d \bm{\Sigma}_c )  \\
&&
+(n_a+1)(n_d+1){\rm tr}(\bm{\Sigma}_b\bm{\Sigma}_c)
{\rm tr}\left( \bm{\Sigma}_a\bm{\Sigma}_b \bm{\Sigma}_a \bm{\Sigma}_d \bm{\Sigma}_c \bm{\Sigma}_d \right) 
+(n_b+1)(n_c+1){\rm tr}(\bm{\Sigma}_a \bm{\Sigma}_d) {\rm tr}\left( 
 \bm{\Sigma}_a \bm{\Sigma}_b \bm{\Sigma}_c \bm{\Sigma}_d \bm{\Sigma}_c \bm{\Sigma}_b \right) 
\\
&&
+(n_b+1)(n_d+1){\rm tr}(\bm{\Sigma}_a \bm{\Sigma}_b\bm{\Sigma}_c\bm{\Sigma}_b)
{\rm tr}(\bm{\Sigma}_a \bm{\Sigma}_d \bm{\Sigma}_c \bm{\Sigma}_d)
+(n_c+1)(n_d+1){\rm tr}(\bm{\Sigma}_a\bm{\Sigma}_b)
{\rm tr}\left( \bm{\Sigma}_a \bm{\Sigma}_d \bm{\Sigma}_c \bm{\Sigma}_b \bm{\Sigma}_c \bm{\Sigma}_d \right) 
\\
&& 
+(n_a+1){\rm tr}(\bm{\Sigma}_b\bm{\Sigma}_c) {\rm tr}(\bm{\Sigma}_c \bm{\Sigma}_d) 
{\rm tr}( \bm{\Sigma}_a\bm{\Sigma}_b \bm{\Sigma}_a  \bm{\Sigma}_d ) 
+(n_b+1){\rm tr}(\bm{\Sigma}_a \bm{\Sigma}_d) {\rm tr}(\bm{\Sigma}_c\bm{\Sigma}_d) 
{\rm tr}(\bm{\Sigma}_a \bm{\Sigma}_b\bm{\Sigma}_c\bm{\Sigma}_b)
 \\
&&
+(n_c+1){\rm tr}(\bm{\Sigma}_a\bm{\Sigma}_b) {\rm tr}(\bm{\Sigma}_a \bm{\Sigma}_d) 
{\rm tr}( \bm{\Sigma}_b \bm{\Sigma}_c \bm{\Sigma}_d \bm{\Sigma}_c ) 
+(n_d+1){\rm tr}(\bm{\Sigma}_a\bm{\Sigma}_b){\rm tr}(\bm{\Sigma}_b\bm{\Sigma}_c)
{\rm tr}( \bm{\Sigma}_a \bm{\Sigma}_d \bm{\Sigma}_c \bm{\Sigma}_d ) 
 \\
&& +{\rm tr}(\bm{\Sigma}_a\bm{\Sigma}_b){\rm tr}(\bm{\Sigma}_a \bm{\Sigma}_d)
{\rm tr}(\bm{\Sigma}_b\bm{\Sigma}_c){\rm tr}( \bm{\Sigma}_c \bm{\Sigma}_d)
 \biggr].
\end{eqnarray*}

Therefore, the conclusion follows from Assumption~\ref{ass1}.
This completes the proof of Proposition~\ref{prop21}.
\qed

\subsection{Proof of Lemma~\ref{lem23}}

From the definitions of $\bm{T}_{a}(n_a)$, $\bm{T}_{b}(n_b)$, $\bm{T}_{c}(n_c)$, $\bm{T}_{d}(n_d)$, 
\[ 
{\rm E}_h[M] 
= {\rm E}_h\left[\frac{1}{r_p}{\rm tr}\left( 
\bm{T}_{a}(n_a) \bm{T}_{b}(n_b) \bm{T}_{c}(n_c) \bm{T}_{d}(n_d) \right)\right] \]
is given as follows:
\begin{itemize}
\item for $h=1,\dots,n_a$, it holds that
\[ {\rm E}_h[M]=\frac{{n_b n_c n_d}}{r_p}{\rm tr}\left( 
\left\{ \bm{T}_{a}(h) + (n_a-h)\bm{\Sigma}_a\right\} \bm{\Sigma}_{b} \bm{\Sigma}_{c} \bm{\Sigma}_{d} \right); \]
\item for $h=n_a+1,\dots,n_a+n_b$, it holds that
\[ {\rm E}_h[M]=\frac{{n_c n_d}}{r_p}{\rm tr}\left( 
\bm{T}_{a}(n_a) \left\{\bm{T}_b(h-n_a)+(n_a+n_b-h)\bm{\Sigma}_{b}\right\} \bm{\Sigma}_{c} \bm{\Sigma}_{d} \right); \]
\item for $h=n_a+n_b+1,\dots,n_a+n_b+n_c$, it holds that
\[
 {\rm E}_h[M] 
= \frac{{n_d}}{r_p}{\rm tr}\left( 
\bm{T}_{a}(n_a) \bm{T}_{b}(n_b) 
\left\{ \bm{T}_c(h-n_a-n_b) 
+(n_a+n_b+n_c-h)\bm{\Sigma}_{c}\right\} 
\bm{\Sigma}_{d} \right); 
\]
\item for $h=n_a+n_b+n_c+1,\dots,n_a+n_b+n_c+n_d$, it holds that
\begin{eqnarray*}
{\rm E}_h[M]
&=& \frac{1}{r_p}{\rm tr}\left( 
\bm{T}_{a}(n_a) \bm{T}_{b}(n_b) \bm{T}_{c}(n_c)  
\right. \\ && \left. \times
\left\{ \bm{T}_d(h-n_a-n_b-n_c)+(n_a+n_b+n_c+n_d-h)\bm{\Sigma}_{d}\right\} \right). 
\end{eqnarray*}
\end{itemize}
Hence, 
$D_h = {\rm E}_h[M] - {\rm E}_{h-1}[M]$ 
is given as follows:
\begin{itemize}
\item for $h=1,\dots,n_a$, it holds that
\[ D_h = 
\frac{{n_b n_c n_d}}{r_p}{\rm tr}\left( 
\left( \bm{x}_h\bm{x}_h' - \bm{\Sigma}_a\right) \bm{\Sigma}_{b} \bm{\Sigma}_{c} \bm{\Sigma}_{d} \right); 
\]
\item for $h=n_a+1,\dots,n_a+n_b$, it holds that
\[ D_h = 
\frac{{n_c n_d}}{r_p}{\rm tr}\left( 
\bm{T}_{a}(n_a) \left(\bm{y}_{h-n_a}\bm{y}_{h-n_a}'-\bm{\Sigma}_{b}\right) \bm{\Sigma}_{c} \bm{\Sigma}_{d} \right);
\]
\item for $h=n_a+n_b+1,\dots,n_a+n_b+n_c$, it holds that
\[ D_h = 
\frac{{n_d}}{r_p}{\rm tr}\left( 
\bm{T}_{a}(n_a) \bm{T}_{b}(n_b) \left(\bm{z}_{h-n_a-n_b}\bm{z}_{h-n_a-n_b}'-\bm{\Sigma}_{c}\right) 
\bm{\Sigma}_{d} \right);
\]
\item for $h=n_a+n_b+n_c+1,\dots,n_a+n_b+n_c+n_d$, it holds that
\[ D_h = 
\frac{1}{r_p}{\rm tr}\left( 
\bm{T}_{a}(n_a) \bm{T}_{b}(n_b) \bm{T}_{c}(n_c) 
\left(\bm{w}_{h-n_a-n_b-n_c}\bm{w}_{h-n_a-n_b-n_c}'-\bm{\Sigma}_{d}\right) \right). \]
\end{itemize}

Recalling that $\sigma_h^2 = {\rm E}_{h-1}[D_h^2]$, by using Lemma~\ref{lem42}, $\sigma_h^2$ is given as follow:
\begin{itemize}
\item for $h=1,\dots,n_a$, it holds that
\begin{eqnarray*}
\sigma_{h}^{2} &=& 
{\rm E}_{h-1}\left[ \left\{ \frac{{n_b n_c n_d}}{r_p}{\rm tr}\left( 
\left( \bm{x}_h\bm{x}_h' - \bm{\Sigma}_a\right) \bm{\Sigma}_{b} \bm{\Sigma}_{c} \bm{\Sigma}_{d} \right) 
\right\}^2 \right] \\
&=& 
\frac{n_b^2 n_c^2 n_d^2}{r_p^2} 
{\rm E}_{h-1}\left[ \left\{ 
\bm{x}_h' \bm{\Sigma}_{a}^{-\frac{1}{2}} \bm{\Sigma}_{a}^{\frac{1}{2}} 
\bm{\Sigma}_{b} \bm{\Sigma}_{c} \bm{\Sigma}_{d} 
\bm{\Sigma}_{a}^{\frac{1}{2}} \bm{\Sigma}_{a}^{-\frac{1}{2}} \bm{x}_h 
- {\rm tr}\left( 
\bm{\Sigma}_{a}^{\frac{1}{2}} \bm{\Sigma}_{b} \bm{\Sigma}_{c} \bm{\Sigma}_{d} \bm{\Sigma}_{a}^{\frac{1}{2}} \right) 
\right\}^2 \right] \\
&=& 
\frac{n_b^2 n_c^2 n_d^2}{r_p^2} 
\left\{ {\rm tr}\left( \left( 
\bm{\Sigma}_{a}^{\frac{1}{2}} \bm{\Sigma}_{b} \bm{\Sigma}_{c} \bm{\Sigma}_{d} \bm{\Sigma}_{a}^{\frac{1}{2}} \right) 
\left( 
\bm{\Sigma}_{a}^{\frac{1}{2}} \bm{\Sigma}_{b} \bm{\Sigma}_{c} \bm{\Sigma}_{d} \bm{\Sigma}_{a}^{\frac{1}{2}} \right) 
\right) \right. \\
&& \left. + {\rm tr}\left( \left( 
\bm{\Sigma}_{a}^{\frac{1}{2}} \bm{\Sigma}_{b} \bm{\Sigma}_{c} \bm{\Sigma}_{d} \bm{\Sigma}_{a}^{\frac{1}{2}} \right)
\left( 
\bm{\Sigma}_{a}^{\frac{1}{2}} \bm{\Sigma}_{d} \bm{\Sigma}_{c} \bm{\Sigma}_{b} \bm{\Sigma}_{a}^{\frac{1}{2}} \right) 
\right) \right\} \\
&=& 
\frac{n_b^2 n_c^2 n_d^2}{r_p^2} 
\left\{ {\rm tr}\left( 
\bm{\Sigma}_{a} \bm{\Sigma}_{b} \bm{\Sigma}_{c} \bm{\Sigma}_{d} \bm{\Sigma}_{a} 
\bm{\Sigma}_{b} \bm{\Sigma}_{c} \bm{\Sigma}_{d} 
\right) 
+ {\rm tr}\left( 
\bm{\Sigma}_{a} \bm{\Sigma}_{b} \bm{\Sigma}_{c} \bm{\Sigma}_{d} 
\bm{\Sigma}_{a} \bm{\Sigma}_{d} \bm{\Sigma}_{c} \bm{\Sigma}_{b} 
\right) \right\};
\end{eqnarray*}
\item for $h=n_a+1,\dots,n_a+n_b$,  it holds that
\begin{eqnarray*}
\sigma_{h}^{2}
&=& 
{\rm E}_{h-1}\left[ \left\{ 
\frac{{n_c n_d}}{r_p} {\rm tr}\left( 
\bm{T}_{a}(n_a) \left(\bm{y}_{h-n_a}\bm{y}_{h-n_a}'-\bm{\Sigma}_{b}\right) \bm{\Sigma}_{c} \bm{\Sigma}_{d} \right)
\right\}^2 \right] \\
&=& 
\frac{n_c^2 n_d^2}{r_p^2} 
{\rm E}_{h-1}\left[ \left\{ 
{\rm tr}\left( 
\bm{y}_{h-n_a}' \bm{\Sigma}_{b}^{-\frac{1}{2}} \bm{\Sigma}_{b}^{\frac{1}{2}} 
\bm{\Sigma}_{c} \bm{\Sigma}_{d} \bm{T}_{a}(n_a) 
\bm{\Sigma}_{b}^{\frac{1}{2}} \bm{\Sigma}_{b}^{-\frac{1}{2}} \bm{y}_{h-n_a} \right) 
\right. \right. \\ && \left. \left.
- {\rm tr}\left( 
\bm{\Sigma}_{b}^{\frac{1}{2}} \bm{\Sigma}_{c} \bm{\Sigma}_{d} \bm{T}_{a}(n_a) \bm{\Sigma}_{b}^{\frac{1}{2}} \right)
\right\}^2 \right] \\
&=& 
\frac{n_c^2 n_d^2}{r_p^2} 
\left\{ {\rm tr}\left( \left( 
\bm{\Sigma}_{b}^{\frac{1}{2}} \bm{\Sigma}_{c} \bm{\Sigma}_{d} \bm{T}_{a}(n_a) \bm{\Sigma}_{b}^{\frac{1}{2}} \right) 
\left( 
\bm{\Sigma}_{b}^{\frac{1}{2}} \bm{\Sigma}_{c} \bm{\Sigma}_{d} \bm{T}_{a}(n_a) \bm{\Sigma}_{b}^{\frac{1}{2}} \right) 
\right) \right. \\
&& \left. + {\rm tr}\left( \left( 
\bm{\Sigma}_{b}^{\frac{1}{2}} \bm{\Sigma}_{c} \bm{\Sigma}_{d} \bm{T}_{a}(n_a) \bm{\Sigma}_{b}^{\frac{1}{2}} \right) 
\left( 
\bm{\Sigma}_{b}^{\frac{1}{2}} \bm{T}_{a}(n_a) \bm{\Sigma}_{d} \bm{\Sigma}_{c} \bm{\Sigma}_{b}^{\frac{1}{2}} \right) 
\right) \right\} \\
&=& 
\frac{n_c^2 n_d^2}{r_p^2} 
\left\{ {\rm tr}\left( 
\bm{T}_{a}(n_a) \bm{\Sigma}_{b} \bm{\Sigma}_{c} \bm{\Sigma}_{d} \bm{T}_{a}(n_a) 
\bm{\Sigma}_{b} \bm{\Sigma}_{c} \bm{\Sigma}_{d} 
\right) \right. \\
&&  \left.
+ {\rm tr}\left( 
\bm{T}_{a}(n_a) \bm{\Sigma}_{b} \bm{T}_{a}(n_a) 
\bm{\Sigma}_{d} \bm{\Sigma}_{c} \bm{\Sigma}_{b} \bm{\Sigma}_{c} \bm{\Sigma}_{d} 
\right) \right\};
\end{eqnarray*}
\item for $h=n_a+n_b+1,\dots,n_a+n_b+n_c$, it holds that
\begin{eqnarray*}
\sigma_{h}^{2} &=& 
{\rm E}_{h-1}\left[ 
\left\{ 
\frac{{n_d}}{r_p}{\rm tr}\left( 
\bm{T}_{a}(n_a) \bm{T}_{b}(n_b) \left(\bm{z}_{h-n_a-n_b}\bm{z}_{h-n_a-n_b}'-\bm{\Sigma}_{c}\right) 
\bm{\Sigma}_{d} \right) 
\right\}^2 \right] \\
&=& \frac{n_d^2 }{r_p^2} 
{\rm E}_{h-1}\left[ \left\{ 
{\rm tr}\left( \bm{z}_{h-n_a-n_b}' \bm{\Sigma}_{c}^{-\frac{1}{2}} \bm{\Sigma}_{c}^{\frac{1}{2}}
\bm{\Sigma}_{d}\bm{T}_{a}(n_a) \bm{T}_{b}(n_b) 
\bm{\Sigma}_{c}^{\frac{1}{2}} \bm{\Sigma}_{c}^{-\frac{1}{2}} \bm{z}_{h-n_a-n_b} \right) 
\right. \right. \\
&& \left. \left. 
- {\rm tr}\left( \bm{\Sigma}_{c}^{\frac{1}{2}}
\bm{\Sigma}_{d}\bm{T}_{a}(n_a) \bm{T}_{b}(n_b) 
\bm{\Sigma}_{c}^{\frac{1}{2}}\right) 
\right\}^2 \right] \\
&=& \frac{n_d^2}{r_p^2} 
\left\{ 
{\rm tr}\left( \bm{\Sigma}_{c}^{\frac{1}{2}}
\bm{\Sigma}_{d} \bm{T}_{a}(n_a) \bm{T}_{b}(n_b) 
\bm{\Sigma}_{c}^{\frac{1}{2}}\bm{\Sigma}_{c}^{\frac{1}{2}}
\bm{\Sigma}_{d} \bm{T}_{a}(n_a) \bm{T}_{b}(n_b) 
\bm{\Sigma}_{c}^{\frac{1}{2}} \right) \right. \\
&& \left. + 
{\rm tr}\left( \bm{\Sigma}_{c}^{\frac{1}{2}}
\bm{\Sigma}_{d} \bm{T}_{a}(n_a) \bm{T}_{b}(n_b) 
\bm{\Sigma}_{c}^{\frac{1}{2}}
\bm{\Sigma}_{c}^{\frac{1}{2}}
\bm{T}_{b}(n_b) \bm{T}_{a}(n_a) \bm{\Sigma}_{d} 
\bm{\Sigma}_{c}^{\frac{1}{2}} \right) 
\right\} \\
&=& \frac{n_d^2}{r_p^2} 
\left\{ 
{\rm tr}\left( 
\bm{T}_{a}(n_a) \bm{T}_{b}(n_b) \bm{\Sigma}_{c} \bm{\Sigma}_{d} 
\bm{T}_{a}(n_a) \bm{T}_{b}(n_b) \bm{\Sigma}_{c} \bm{\Sigma}_{d} 
\right) \right. \\
&& \left. + 
{\rm tr}\left( 
\bm{T}_{a}(n_a) \bm{T}_{b}(n_b) \bm{\Sigma}_{c} \bm{T}_{b}(n_b) 
\bm{T}_{a}(n_a) \bm{\Sigma}_{d} \bm{\Sigma}_{c} \bm{\Sigma}_{d} 
\right) 
\right\};
\end{eqnarray*}
\item for $h=n_a+n_b+n_c+1,\dots,n_a+n_b+n_c+n_d$, it holds that
\begin{eqnarray*}
\sigma_{h}^{2} 
&=& 
{\rm E}_{h-1}\left[ \left\{ 
\frac{1}{r_p}{\rm tr}\left( 
\bm{T}_{a}(n_a) \bm{T}_{b}(n_b) \bm{T}_{c}(n_c) 
\left(\bm{w}_{h-n_a-n_b-n_c}\bm{w}_{h-n_a-n_b-n_c}'-\bm{\Sigma}_{d}\right) \right) 
\right\}^2 \right] \\
&=& \frac{1}{r_p^2} 
{\rm E}_{h-1}\left[ \left\{ 
{\rm tr}\left( \bm{w}_{h-n_a-n_b-n_c}' \bm{\Sigma}_{d}^{-\frac{1}{2}} \bm{\Sigma}_{d}^{\frac{1}{2}}
\bm{T}_{a}(n_a) \bm{T}_{b}(n_b) \bm{T}_{c}(n_c) 
\bm{\Sigma}_{d}^{\frac{1}{2}} \bm{\Sigma}_{d}^{-\frac{1}{2}} \bm{w}_{h-n_a-n_b-n_c} \right) 
\right. \right. \\
&& \left. \left.
 - {\rm tr}\left( \bm{\Sigma}_{d}^{\frac{1}{2}}
\bm{T}_{a}(n_a) \bm{T}_{b}(n_b) \bm{T}_{c}(n_c) 
\bm{\Sigma}_{d}^{\frac{1}{2}}\right) 
\right\}^2 \right] \\
&=& \frac{1}{r_p^2} 
\left\{ 
{\rm tr}\left( \bm{\Sigma}_{d}^{\frac{1}{2}}
\bm{T}_{a}(n_a) \bm{T}_{b}(n_b) \bm{T}_{c}(n_c) 
\bm{\Sigma}_{d}^{\frac{1}{2}}\bm{\Sigma}_{d}^{\frac{1}{2}}
\bm{T}_{a}(n_a) \bm{T}_{b}(n_b) \bm{T}_{c}(n_c) 
\bm{\Sigma}_{d}^{\frac{1}{2}} \right) \right. \\
&& \left. + 
{\rm tr}\left( \bm{\Sigma}_{d}^{\frac{1}{2}}
\bm{T}_{a}(n_a) \bm{T}_{b}(n_b) \bm{T}_{c}(n_c) 
\bm{\Sigma}_{d}^{\frac{1}{2}}
\bm{\Sigma}_{d}^{\frac{1}{2}}
\bm{T}_{c}(n_c) \bm{T}_{b}(n_b) \bm{T}_{a}(n_a) 
\bm{\Sigma}_{d}^{\frac{1}{2}} \right) 
\right\} \\
&=& \frac{1}{r_p^2} 
\left\{ 
{\rm tr}\left( 
\bm{T}_{a}(n_a) \bm{T}_{b}(n_b) \bm{T}_{c}(n_c) \bm{\Sigma}_{d} 
\bm{T}_{a}(n_a) \bm{T}_{b}(n_b) \bm{T}_{c}(n_c) \bm{\Sigma}_{d} 
\right) \right. \\
&& \left. + 
{\rm tr}\left( 
\bm{T}_{a}(n_a) \bm{T}_{b}(n_b) \bm{T}_{c}(n_c) \bm{\Sigma}_{d}
\bm{T}_{c}(n_c) \bm{T}_{b}(n_b) \bm{T}_{a}(n_a) \bm{\Sigma}_{d} 
\right) 
\right\}.
\end{eqnarray*}
\end{itemize}

In general, as
\begin{eqnarray*}
{\rm V}[X+Y] &=& {\rm V}[X] + {\rm V} [Y] + 2 {\rm Cov} [X,Y]
\le {\rm V}[X] + {\rm V} [Y] + 2\left( {\rm V}[X] {\rm V}[Y] \right)^{\frac{1}{2}} \\
&\le & {\rm V}[X] + {\rm V}[Y]+2\max ({\rm V}[X] , {\rm V}[Y]) \le 4\max ({\rm V}[X], {\rm V}[Y])
\end{eqnarray*}
for random variables $X$ and $Y$ which have finite second moments, in order to prove
\[ {\rm V} \left[ \sum_{h=1}^{n_a+n_b+n_c+n_d}\sigma_{h}^{2} \right] \to 0, \]
it suffices to show
\begin{equation}\label{t521}
{\rm V} \left[ \sum_{h=1}^{n_a}\sigma_{h}^{2} \right] \to 0,
\end{equation}
\begin{equation}\label{t522}
{\rm V} \left[ \sum_{h=n_a+1}^{n_a+n_b}\sigma_{h}^{2} \right] \to 0,
\end{equation}
\begin{equation}\label{t523}
{\rm V} \left[ \sum_{h=n_a+n_b+1}^{n_a+n_b+n_c}\sigma_{h}^{2} \right] \to 0, 
\end{equation}
and
\begin{equation}\label{t524}
{\rm V}\left[ \sum_{h=n_a+n_b+n_c+1}^{n_a+n_b+n_c+n_d}\sigma_{h}^{2} \right] \to 0.
\end{equation}

\begin{itemize}
\item \textit{Proof of \eqref{t521}.}
As $r_p = p^2\sqrt{n_a n_b n_c n_d}$, it follows from
\[
\sum_{h=1}^{n_a}\sigma_{h}^{2}
= 
\frac{n_bn_cn_d}{p^4} 
\left\{ {\rm tr}\left( 
\bm{\Sigma}_{a} \bm{\Sigma}_{b} \bm{\Sigma}_{c} \bm{\Sigma}_{d} \bm{\Sigma}_{a} 
\bm{\Sigma}_{b} \bm{\Sigma}_{c} \bm{\Sigma}_{d} 
\right) 
+ {\rm tr}\left( 
\bm{\Sigma}_{a} \bm{\Sigma}_{b} \bm{\Sigma}_{c} \bm{\Sigma}_{d} 
\bm{\Sigma}_{a} \bm{\Sigma}_{d} \bm{\Sigma}_{c} \bm{\Sigma}_{b} 
\right) \right\}
\]
that
\[ {\rm E} \left[ \left( \sum_{h=1}^{n_a}\sigma_{h}^{2} \right)^2 \right] 
= O\left( \frac{n_{b}^{2}n_{c}^{2}n_{d}^{2}}{p^6} \right) \to 0. \]

\item \textit{Proof of \eqref{t522}.}
It follows from
\begin{eqnarray*}
\sum_{h=n_a+1}^{n_a+n_b}\sigma_{h}^{2} 
&=& 
\frac{n_cn_d}{p^4n_a} 
\left\{ {\rm tr}\left( 
\bm{T}_{a}(n_a) \bm{\Sigma}_{b} \bm{\Sigma}_{c} \bm{\Sigma}_{d} \bm{T}_{a}(n_a) 
\bm{\Sigma}_{b} \bm{\Sigma}_{c} \bm{\Sigma}_{d} 
\right) \right. \\
&&\left.
 + {\rm tr}\left( 
\bm{T}_{a}(n_a) \bm{\Sigma}_{b} \bm{T}_{a}(n_a) 
\bm{\Sigma}_{d} \bm{\Sigma}_{c} \bm{\Sigma}_{b} \bm{\Sigma}_{c} \bm{\Sigma}_{d} 
\right) \right\} 
\end{eqnarray*}
that
\[ {\rm E} \left[ \left( \sum_{h=n_a+1}^{n_a+n_b}\sigma_{h}^{2} \right)^2 \right] 
= \frac{n_{c}^{2}n_{d}^{2}}{p^8n_{a}^{2}}
O( n_{a}^{2}p^4 ) = O\left( \frac{n_{c}^{2}n_{d}^{2}}{p^4} \right) \to 0 .\]

\item \textit{Proof of \eqref{t523}.}
By using Proposition~\ref{prop48}, it follows from
\begin{eqnarray*}
\sum_{h=n_a+n_b+1}^{n_a+n_b+n_c}\sigma_{h}^{2} 
&=& 
\frac{n_d}{p^4n_an_b} 
\left\{ 
{\rm tr}\left( 
\bm{T}_{a}(n_a) \bm{T}_{b}(n_b) \bm{\Sigma}_{c} \bm{\Sigma}_{d} 
\bm{T}_{a}(n_a) \bm{T}_{b}(n_b) \bm{\Sigma}_{c} \bm{\Sigma}_{d} 
\right) \right. \\
&& \left. + {\rm tr}\left( 
\bm{T}_{a}(n_a) \bm{T}_{b}(n_b) \bm{\Sigma}_{c} \bm{T}_{b}(n_b) 
\bm{T}_{a}(n_a) \bm{\Sigma}_{d} \bm{\Sigma}_{c} \bm{\Sigma}_{d} 
\right) 
\right\}
\end{eqnarray*}
that
\[ {\rm E} \left[ \left( \sum_{h=n_a+n_b+1}^{n_a+n_b+n_c}\sigma_{h}^{2} \right)^2 \right] 
= \frac{n_{d}^{2}}{p^8n_{a}^{2}n_{b}^{2}}
O( n_{a}^{2}n_{b}^{2}p^6 ) = O\left( \frac{n_{d}^{2}}{p^2} \right) \to 0. \]

\item \textit{Proof of \eqref{t524}.}
It follows from
\begin{eqnarray*}
\lefteqn{
\sum_{h=n_a+n_b+n_c+1}^{n_a+n_b+n_c+n_d}\sigma_{h}^{2} }\\
&=& 
\frac{1}{p^4n_an_bn_c} 
\left\{ 
{\rm tr}\left( 
\bm{T}_{a}(n_a) \bm{T}_{b}(n_b) \bm{T}_{c}(n_c) \bm{\Sigma}_{d} 
\bm{T}_{a}(n_a) \bm{T}_{b}(n_b) \bm{T}_{c}(n_c) \bm{\Sigma}_{d} 
\right) \right. \\
&& \left. + 
{\rm tr}\left( 
\bm{T}_{a}(n_a) \bm{T}_{b}(n_b) \bm{T}_{c}(n_c) \bm{\Sigma}_{d}
\bm{T}_{c}(n_c) \bm{T}_{b}(n_b) \bm{T}_{a}(n_a) \bm{\Sigma}_{d} 
\right) 
\right\}
\end{eqnarray*}
that
\begin{eqnarray*}
\lefteqn{ {\rm E} \left[ \left( \sum_{h=n_a+n_b+n_c+1}^{n_a+n_b+n_c+n_d}\sigma_{h}^{2} \right)^2 \right] } \\
&=& 
\frac{1}{p^8n_{a}^{2}n_{b}^{2}n_{c}^{2}} 
{\rm E} \left[ \left\{ 
{\rm tr}\left( 
\bm{T}_{a}(n_a) \bm{T}_{b}(n_b) \bm{T}_{c}(n_c) \bm{\Sigma}_{d} 
\bm{T}_{a}(n_a) \bm{T}_{b}(n_b) \bm{T}_{c}(n_c) \bm{\Sigma}_{d} 
\right) \right. \right. \\
&& \left. \left. + 
{\rm tr}\left( 
\bm{T}_{a}(n_a) \bm{T}_{b}(n_b) \bm{T}_{c}(n_c) \bm{\Sigma}_{d}
\bm{T}_{c}(n_c) \bm{T}_{b}(n_b) \bm{T}_{a}(n_a) \bm{\Sigma}_{d} 
\right) 
\right\}^2 \right] \\
&=& 
\frac{1}{p^8n_{a}^{2}n_{b}^{2}n_{c}^{2}} \times \\
&& {\rm E} \left[ 
{\rm tr}\left( 
\bm{T}_{a}(n_a) \bm{T}_{b}(n_b) \bm{T}_{c}(n_c) \bm{\Sigma}_{d} 
\bm{T}_{a}(n_a) \bm{T}_{b}(n_b) \bm{T}_{c}(n_c) \bm{\Sigma}_{d} 
\right) 
\right.\\ &&\left. 
\times {\rm tr}\left( 
\bm{T}_{a}(n_a) \bm{T}_{b}(n_b) \bm{T}_{c}(n_c) \bm{\Sigma}_{d} 
\bm{T}_{a}(n_a) \bm{T}_{b}(n_b) \bm{T}_{c}(n_c) \bm{\Sigma}_{d} 
\right) \right. \\
&& +
{\rm tr}\left( 
\bm{T}_{a}(n_a) \bm{T}_{b}(n_b) \bm{T}_{c}(n_c) \bm{\Sigma}_{d} 
\bm{T}_{a}(n_a) \bm{T}_{b}(n_b) \bm{T}_{c}(n_c) \bm{\Sigma}_{d} 
\right)
\\ &&
{\times\rm tr}\left( 
\bm{T}_{a}(n_a) \bm{T}_{b}(n_b) \bm{T}_{c}(n_c) \bm{\Sigma}_{d}
\bm{T}_{c}(n_c) \bm{T}_{b}(n_b) \bm{T}_{a}(n_a) \bm{\Sigma}_{d} 
\right) \\
&& +
{\rm tr}\left( 
\bm{T}_{a}(n_a) \bm{T}_{b}(n_b) \bm{T}_{c}(n_c) \bm{\Sigma}_{d}
\bm{T}_{c}(n_c) \bm{T}_{b}(n_b) \bm{T}_{a}(n_a) \bm{\Sigma}_{d} 
\right)
\\ &&
\times {\rm tr}\left( 
\bm{T}_{a}(n_a) \bm{T}_{b}(n_b) \bm{T}_{c}(n_c) \bm{\Sigma}_{d} 
\bm{T}_{a}(n_a) \bm{T}_{b}(n_b) \bm{T}_{c}(n_c) \bm{\Sigma}_{d} 
\right) \\
&& \left. +
{\rm tr}\left( 
\bm{T}_{a}(n_a) \bm{T}_{b}(n_b) \bm{T}_{c}(n_c) \bm{\Sigma}_{d}
\bm{T}_{c}(n_c) \bm{T}_{b}(n_b) \bm{T}_{a}(n_a) \bm{\Sigma}_{d} 
\right) 
\right. \\ && \left.
\times {\rm tr}\left( 
\bm{T}_{a}(n_a) \bm{T}_{b}(n_b) \bm{T}_{c}(n_c) \bm{\Sigma}_{d}
\bm{T}_{c}(n_c) \bm{T}_{b}(n_b) \bm{T}_{a}(n_a) \bm{\Sigma}_{d} 
\right) 
 \right] \\
& = & 
\frac{1}{p^8n_{b}^{2}n_{c}^{2}} 
{\rm E} \left[ 
{\rm tr}\left( 
\bm{\Sigma}_{a} \bm{T}_{b}(n_b) \bm{T}_{c}(n_c) \bm{\Sigma}_{d} \right) 
{\rm tr}\left( 
\bm{\Sigma}_{a} \bm{T}_{b}(n_b) \bm{T}_{c}(n_c) \bm{\Sigma}_{d} 
\right) 
\right. \\ && \left. \times
{\rm tr}\left( 
\bm{\Sigma}_{a} \bm{T}_{b}(n_b) \bm{T}_{c}(n_c) \bm{\Sigma}_{d} \right) 
{\rm tr}\left( 
\bm{\Sigma}_{a} \bm{T}_{b}(n_b) \bm{T}_{c}(n_c) \bm{\Sigma}_{d} 
\right) \right. \\
&& +
{\rm tr}\left( 
\bm{\Sigma}_{a} \bm{T}_{b}(n_b) \bm{T}_{c}(n_c) \bm{\Sigma}_{d} \right) 
{\rm tr}\left( 
\bm{\Sigma}_{a} \bm{T}_{b}(n_b) \bm{T}_{c}(n_c) \bm{\Sigma}_{d} 
\right)
\\ && \times
{\rm tr}\left( 
\bm{\Sigma}_{a} \bm{T}_{b}(n_b) \bm{T}_{c}(n_c) \bm{\Sigma}_{d}
\bm{T}_{c}(n_c) \bm{T}_{b}(n_b) \right) 
{\rm tr}\left( \bm{\Sigma}_{a} \bm{\Sigma}_{d} 
\right) \\
&& +
{\rm tr}\left( 
\bm{\Sigma}_{a} \bm{T}_{b}(n_b) \bm{T}_{c}(n_c) \bm{\Sigma}_{d}
\bm{T}_{c}(n_c) \bm{T}_{b}(n_b) \right) 
{\rm tr}\left( \bm{\Sigma}_{a} \bm{\Sigma}_{d} 
\right)
\\ && \times
{\rm tr}\left( 
\bm{\Sigma}_{a} \bm{T}_{b}(n_b) \bm{T}_{c}(n_c) \bm{\Sigma}_{d} \right) 
{\rm tr}\left( 
\bm{\Sigma}_{a} \bm{T}_{b}(n_b) \bm{T}_{c}(n_c) \bm{\Sigma}_{d} 
\right) \\
&& \left. +
{\rm tr}\left( 
\bm{\Sigma}_{a} \bm{T}_{b}(n_b) \bm{T}_{c}(n_c) \bm{\Sigma}_{d}
\bm{T}_{c}(n_c) \bm{T}_{b}(n_b) \right) 
{\rm tr}\left( \bm{\Sigma}_{a} \bm{\Sigma}_{d} 
\right) 
\right. \\ && \left. \times
{\rm tr}\left( 
\bm{\Sigma}_{a} \bm{T}_{b}(n_b) \bm{T}_{c}(n_c) \bm{\Sigma}_{d}
\bm{T}_{c}(n_c) \bm{T}_{b}(n_b) \right) 
{\rm tr}\left( \bm{\Sigma}_{a} \bm{\Sigma}_{d} 
\right) 
 \right]  + o(1) \\
& = & 
\frac{1}{p^8n_{c}^{2}} 
{\rm E} \left[ 
{\rm tr}\left( 
\bm{\Sigma}_{b} \bm{T}_{c}(n_c) \bm{\Sigma}_{d}
\bm{T}_{c}(n_c) \right) 
{\rm tr}\left( \bm{\Sigma}_{b} \bm{\Sigma}_{a} \right) 
{\rm tr}\left( \bm{\Sigma}_{a} \bm{\Sigma}_{d} 
\right) 
\right. \\ && \left. \times
{\rm tr}\left( 
\bm{\Sigma}_{b} \bm{T}_{c}(n_c) \bm{\Sigma}_{d}
\bm{T}_{c}(n_c) \right) 
{\rm tr}\left( \bm{\Sigma}_{b} \bm{\Sigma}_{a} \right) 
{\rm tr}\left( \bm{\Sigma}_{a} \bm{\Sigma}_{d} 
\right) 
 \right]  + o(1) \\
& = & 
\frac{1}{p^8} 
{\rm tr}\left( 
\bm{\Sigma}_{c} \bm{\Sigma}_{d} \right) 
{\rm tr}\left( \bm{\Sigma}_{c} \bm{\Sigma}_{b} \right) 
{\rm tr}\left( \bm{\Sigma}_{b} \bm{\Sigma}_{a} \right) 
{\rm tr}\left( \bm{\Sigma}_{a} \bm{\Sigma}_{d} 
\right) 
\\ && \times
{\rm tr}\left( 
\bm{\Sigma}_{c} \bm{\Sigma}_{d} \right) 
{\rm tr}\left( \bm{\Sigma}_{c} \bm{\Sigma}_{b} \right) 
{\rm tr}\left( \bm{\Sigma}_{b} \bm{\Sigma}_{a} \right) 
{\rm tr}\left( \bm{\Sigma}_{a} \bm{\Sigma}_{d} 
\right) 
+ o(1) \\
& \to & 
\sigma_{ab}^{2}\sigma_{ad}^{2}\sigma_{bc}^{2}\sigma_{cd}^{2}.
\end{eqnarray*}
and that
\begin{eqnarray*}
\lefteqn{ {\rm E} \left[ \sum_{h=n_a+n_b+n_c+1}^{n_a+n_b+n_c+n_d}\sigma_{h}^{2} \right] } \\
&=& 
\frac{1}{p^4n_{a}n_{b}n_{c}} 
{\rm E} \left[ 
{\rm tr}\left( 
\bm{T}_{a}(n_a) \bm{T}_{b}(n_b) \bm{T}_{c}(n_c) \bm{\Sigma}_{d} 
\bm{T}_{a}(n_a) \bm{T}_{b}(n_b) \bm{T}_{c}(n_c) \bm{\Sigma}_{d} 
\right) \right. \\
&& \left. + 
{\rm tr}\left( 
\bm{T}_{a}(n_a) \bm{T}_{b}(n_b) \bm{T}_{c}(n_c) \bm{\Sigma}_{d}
\bm{T}_{c}(n_c) \bm{T}_{b}(n_b) \bm{T}_{a}(n_a) \bm{\Sigma}_{d} 
\right) 
\right] \\
& = & 
\frac{1}{p^4n_{b}n_{c}} 
{\rm E} \left[ 
{\rm tr}\left( 
\bm{\Sigma}_{a} \bm{T}_{b}(n_b) \bm{T}_{c}(n_c) \bm{\Sigma}_{d} \right)
{\rm tr}\left( \bm{\Sigma}_{a} \bm{T}_{b}(n_b) \bm{T}_{c}(n_c) \bm{\Sigma}_{d} 
\right) \right. \\
&& \left. + 
{\rm tr}\left( 
\bm{\Sigma}_{a} \bm{T}_{b}(n_b) \bm{T}_{c}(n_c) \bm{\Sigma}_{d}
\bm{T}_{c}(n_c) \bm{T}_{b}(n_b) \right)
{\rm tr}\left( \bm{\Sigma}_{a} \bm{\Sigma}_{d} 
\right) 
\right]  + o(1) \\
& = & 
\frac{1}{p^4n_{c}} 
{\rm E} \left[ 
{\rm tr}\left( 
\bm{\Sigma}_{b} \bm{T}_{c}(n_c) \bm{\Sigma}_{d}
\bm{T}_{c}(n_c) \right)
{\rm tr}\left(\bm{\Sigma}_{b} \bm{\Sigma}_{a} \right)
{\rm tr}\left(\bm{\Sigma}_{a} \bm{\Sigma}_{d} 
\right) 
\right]  + o(1) \\
& = & 
\frac{1}{p^4} 
{\rm E} \left[ 
{\rm tr}\left( 
\bm{\Sigma}_{b} \bm{\Sigma}_{c} \right)
{\rm tr}\left( \bm{\Sigma}_{d}
\bm{\Sigma}_{c} \right)
{\rm tr}\left( \bm{\Sigma}_{b} \bm{\Sigma}_{a} \right)
{\rm tr}\left( \bm{\Sigma}_{a} \bm{\Sigma}_{d} 
\right) 
\right]  + o(1) \\
&\to & \sigma_{ab}\sigma_{ad}\sigma_{bc}\sigma_{cd}.
\end{eqnarray*}
Hence, 
\begin{eqnarray*}
\lefteqn{ {\rm V} \left[ \sum_{h=n_a+n_b+n_c+1}^{n_a+n_b+n_c+n_d}\sigma_{h}^{2} \right] } \\
&=& {\rm E} \left[ \left( \sum_{h=n_a+n_b+n_c+1}^{n_a+n_b+n_c+n_d}\sigma_{h}^{2} \right)^2 \right] 
- \left( {\rm E} \left[ \sum_{h=n_a+n_b+n_c+1}^{n_a+n_b+n_c+n_d}\sigma_{h}^{2} \right] \right)^2 \\
&\to& 0. 
\end{eqnarray*}
\end{itemize}

This completes the proof of Lemma~\ref{lem23}.
\qed

\subsection{Proof of Lemma~\ref{lem24}}

To prove
\[ \sum_{h=1}^{n_a+n_b+n_c+n_d} {\rm E} [D_{h}^{4}] \to 0, \]
it suffices to show
\begin{equation}\label{t531}
\sum_{h=1}^{n_a} {\rm E} [D_{h}^{4}] \to 0,
\end{equation}
\begin{equation}\label{t532}
\sum_{h=n_a+1}^{n_a+n_b} {\rm E} [D_{h}^{4}]\to 0,
\end{equation}
\begin{equation}\label{t533}
\sum_{h=n_a+n_b+1}^{n_a+n_b+n_c} {\rm E} [D_{h}^{4}] \to 0, 
\end{equation}
and
\begin{equation}\label{t534}
\sum_{h=n_a+n_b+n_c+1}^{n_a+n_b+n_c+n_d} {\rm E} [D_{h}^{4}] \to 0.
\end{equation}

\begin{itemize}
\item \textit{Proof of \eqref{t531}.}
For $h=1,\dots,n_a$, it follows from Lemma~\ref{lem43} that
\begin{eqnarray*}
{\rm E} [D_{h}^{4}] &=& 
{\rm E} \left[ \left\{ \frac{{n_b n_c n_d}}{r_p}{\rm tr}\left( 
\left( \bm{x}_h\bm{x}_h' - \bm{\Sigma}_a\right) \bm{\Sigma}_{b} \bm{\Sigma}_{c} \bm{\Sigma}_{d} \right) 
\right\}^4 \right] \\
&=& 
\frac{n_{b}^{4} n_{c}^{4} n_{d}^{4}}{r_p^4} {\rm E} \left[ \left\{  
\bm{x}_h'\bm{\Sigma}_{b} \bm{\Sigma}_{c} \bm{\Sigma}_{d}\bm{x}_h 
- {\rm tr}\left( \bm{\Sigma}_a \bm{\Sigma}_{b} \bm{\Sigma}_{c} \bm{\Sigma}_{d} \right) 
\right\}^4 \right] \\
&=& 
\frac{n_{b}^{4} n_{c}^{4} n_{d}^{4}}{r_p^4} {\rm E} \left[ \left\{  
\bm{x}_h'\bm{\Sigma}_{a}^{-\frac{1}{2}}\bm{\Sigma}_{a}^{\frac{1}{2}}
\bm{\Sigma}_{b} \bm{\Sigma}_{c} \bm{\Sigma}_{d}\bm{\Sigma}_{a}^{\frac{1}{2}}\bm{\Sigma}_{a}^{-\frac{1}{2}}\bm{x}_h 
- {\rm tr}\left( \bm{\Sigma}_{a}^{\frac{1}{2}} \bm{\Sigma}_{b} \bm{\Sigma}_{c} \bm{\Sigma}_{d} 
\bm{\Sigma}_{a}^{\frac{1}{2}}\right) 
\right\}^4 \right] \\
&=& \frac{n_{b}^{4} n_{c}^{4} n_{d}^{4}}{r_p^4} 
\left\{ 3{\rm tr}\left( \left( \bm{A}+\bm{A}' \right)^4 \right) 
+ \frac{3}{4}\left\{ {\rm tr}\left( \left( \bm{A}+\bm{A}' \right)^2 \right) \right\}^2 \right\},
\end{eqnarray*}
where $\bm{A}=\bm{\Sigma}_{a}^{\frac{1}{2}} \bm{\Sigma}_{b} \bm{\Sigma}_{c} \bm{\Sigma}_{d} 
\bm{\Sigma}_{a}^{\frac{1}{2}}$.
In general, for a nonnegative definite matrix $\bm{B} \neq \bm{0}$ whose Spectral decomposition is given by $\bm{B}=\bm{U}\bm{\Lambda}\bm{U}'$, it holds that
\[ \left\{ {\rm tr}(\bm{B}) \right\}^2 = \left( {\rm tr}(\bm{\Lambda}) \right)^2 
= \left( \sum_{i=1}^{p}\lambda_i \right)^2 > \sum_{i=1}^{p}\lambda_{i}^{2} 
= {\rm tr}(\bm{\Lambda}^2) = {\rm tr}(\bm{B}^2). \]
This implies that
\begin{eqnarray*}
\sum_{h=1}^{n_a} {\rm E} [D_{h}^{4}] &=& 
\frac{n_a n_{b}^{4} n_{c}^{4} n_{d}^{4}}{r_p^4} 
\left[ 3{\rm tr}\left( \left( \bm{A}+\bm{A}' \right)^4 \right) 
+ \frac{3}{4}\left\{ {\rm tr}\left( \left( \bm{A}+\bm{A}' \right)^2 \right) \right\}^2 \right] \\
&<& \frac{n_a n_{b}^{4} n_{c}^{4} n_{d}^{4}}{r_p^4} 
\left[ 3\left\{ {\rm tr}\left( \left( \bm{A}+\bm{A}' \right)^2 \right) \right\}^2 
+ \frac{3}{4}\left\{ {\rm tr}\left( \left( \bm{A}+\bm{A}' \right)^2 \right) \right\}^2 \right] \\
&=& \frac{15}{4}\frac{n_a n_{b}^{4} n_{c}^{4} n_{d}^{4}}{r_p^4} 
\left\{ {\rm tr}\left( \left( \bm{A}+\bm{A}' \right)^2 \right) \right\}^2 \\
&=& \frac{15}{4}\frac{n_a n_{b}^{4} n_{c}^{4} n_{d}^{4}}{r_p^4} 
\left\{ {\rm tr}\left( \left( \bm{\Sigma}_{a}^{\frac{1}{2}} \bm{\Sigma}_{b} \bm{\Sigma}_{c} \bm{\Sigma}_{d} 
\bm{\Sigma}_{a}^{\frac{1}{2}}
+\bm{\Sigma}_{a}^{\frac{1}{2}} \bm{\Sigma}_{d} \bm{\Sigma}_{c} \bm{\Sigma}_{b} 
\bm{\Sigma}_{a}^{\frac{1}{2}} \right)^2 \right) \right\}^2 \\
&=& \frac{15}{4}\frac{n_{b}^{2}n_{c}^{2}n_{d}^{2}}{p^6n_{a}} 
\left\{ \frac{{\rm tr}\left( \left( \bm{\Sigma}_{a}^{\frac{1}{2}} \bm{\Sigma}_{b} \bm{\Sigma}_{c} \bm{\Sigma}_{d} 
\bm{\Sigma}_{a}^{\frac{1}{2}}
+\bm{\Sigma}_{a}^{\frac{1}{2}} \bm{\Sigma}_{d} \bm{\Sigma}_{c} \bm{\Sigma}_{b} 
\bm{\Sigma}_{a}^{\frac{1}{2}} \right)^2 \right)}{p} \right\}^2 \\
& \to & 0.
\end{eqnarray*}

\item \textit{Proof of \eqref{t532}.}
For $h=n_a+1,\dots,n_a+n_b$, it follows from Lemma~\ref{lem43} that
\begin{eqnarray*}
{\rm E} [D_{h}^{4}] &=& 
{\rm E} \left[ \left\{ \frac{{n_cn_d}}{r_p}{\rm tr}\left( 
\bm{T}_{a}(n_a) \left(\bm{y}_{h-n_a}\bm{y}_{h-n_a}'-\bm{\Sigma}_{b}\right) \bm{\Sigma}_{c} \bm{\Sigma}_{d} \right) 
\right\}^4 \right] \\
&=& \frac{n_{c}^{4} n_{d}^{4}}{r_p^4} 
{\rm E} \left[ \left\{ {\rm tr}\left( 
\bm{y}_{h-n_a}'\bm{\Sigma}_{b}^{-\frac{1}{2}}\bm{\Sigma}_{b}^{\frac{1}{2}}
\bm{\Sigma}_{c} \bm{\Sigma}_{d} 
\bm{T}_{a}(n_a) \bm{\Sigma}_{b}^{\frac{1}{2}}\bm{\Sigma}_{b}^{-\frac{1}{2}}\bm{y}_{h-n_a} \right) 
\right.\right. \\ && \left.\left.
- {\rm tr}\left( 
\bm{\Sigma}_{b}^{\frac{1}{2}} \bm{\Sigma}_{c} \bm{\Sigma}_{d} \bm{T}_{a}(n_a) \bm{\Sigma}_{b}^{\frac{1}{2}} \right) 
\right\}^4 \right] \\
&=& \frac{n_{c}^{4} n_{d}^{4}}{r_p^4} 
{\rm E} \left[ 3{\rm tr}\left( \left( \bm{B}+\bm{B}' \right)^4 \right) 
+ \frac{3}{4}\left\{ {\rm tr}\left( \left( \bm{B}+\bm{B}' \right)^2 \right) \right\}^2 \right],
\end{eqnarray*}
where $\bm{B}=\bm{\Sigma}_{b}^{\frac{1}{2}} \bm{\Sigma}_{c} \bm{\Sigma}_{d} \bm{T}_{a}(n_a) 
\bm{\Sigma}_{b}^{\frac{1}{2}}$, so that
\begin{eqnarray*}
{\rm E} [D_{h}^{4}] 
&=& \frac{n_{c}^{4} n_{d}^{4}}{r_p^4} 
{\rm E} \left[ 3{\rm tr}\left( \left( \bm{B}+\bm{B}' \right)^4 \right) 
+ \frac{3}{4}\left\{ {\rm tr}\left( \left( \bm{B}+\bm{B}' \right)^2 \right) \right\}^2 \right] \\
&<& \frac{n_{c}^{4} n_{d}^{4}}{r_p^4} 
{\rm E} \left[ 3\left\{ {\rm tr}\left( \left( \bm{B}+\bm{B}' \right)^2 \right) \right\}^2
+ \frac{3}{4}\left\{ {\rm tr}\left( \left( \bm{B}+\bm{B}' \right)^2 \right) \right\}^2 \right] \\
&=& 
\frac{15}{4}\frac{n_{c}^{4} n_{d}^{4}}{r_p^4} 
{\rm E} \left[ \left\{ {\rm tr}\left( \left( \bm{B}+\bm{B}' \right)^2 \right) \right\}^2 \right] \\
&=& 
\frac{15}{4}\frac{n_{c}^{4} n_{d}^{4}}{r_p^4} 
{\rm E} \left[ \left\{ {\rm tr}\left( \left( \bm{\Sigma}_{b}^{\frac{1}{2}} \bm{\Sigma}_{c} \bm{\Sigma}_{d} \bm{T}_{a}(n_a) 
\bm{\Sigma}_{b}^{\frac{1}{2}}+\bm{\Sigma}_{b}^{\frac{1}{2}} \bm{T}_{a}(n_a) \bm{\Sigma}_{d} \bm{\Sigma}_{c} 
\bm{\Sigma}_{b}^{\frac{1}{2}} \right)^2 \right) \right\}^2 \right].
\end{eqnarray*}
Hence,
\begin{eqnarray*}
\lefteqn{ \sum_{h=n_a+1}^{n_a+n_b} {\rm E} [D_{h}^{4}] } \\
&<& 
\frac{15}{4}\frac{n_b n_{c}^{4} n_{d}^{4}}{r_p^4} 
{\rm E} \left[ \left\{ {\rm tr}\left( \left( \bm{\Sigma}_{b}^{\frac{1}{2}} \bm{\Sigma}_{c} \bm{\Sigma}_{d} \bm{T}_{a}(n_a) 
\bm{\Sigma}_{b}^{\frac{1}{2}}+\bm{\Sigma}_{b}^{\frac{1}{2}} \bm{T}_{a}(n_a) \bm{\Sigma}_{d} \bm{\Sigma}_{c} 
\bm{\Sigma}_{b}^{\frac{1}{2}} \right)^2 \right) \right\}^2 \right] \\
&=& 
\frac{15}{4}\frac{n_{c}^{2}n_{d}^{2}}{p^8n_{a}^{2}n_{b}} 
\\ &&
O\left( {\rm E} \left[ {\rm tr}\left( \bm{T}_{a}(n_a) \bm{\Sigma}_{b} \bm{\Sigma}_{c} \bm{\Sigma}_{d} 
\bm{T}_{a}(n_a) \bm{\Sigma}_{b} \bm{\Sigma}_{c} \bm{\Sigma}_{d} \right) 
{\rm tr}\left( \bm{T}_{a}(n_a) \bm{\Sigma}_{b} \bm{\Sigma}_{c} \bm{\Sigma}_{d} 
\bm{T}_{a}(n_a) \bm{\Sigma}_{b} \bm{\Sigma}_{c} \bm{\Sigma}_{d} \right) \right] \right) \\
&=& 
\frac{15}{4}\frac{n_{c}^{2}n_{d}^{2}}{p^8n_{a}^{2}n_{b}} 
O(n_{a}^{2}p^4) \\
&\to & 0.
\end{eqnarray*}

\item \textit{Proof of \eqref{t533}.}
For $h=n_a+n_b+1,\dots,n_a+n_b+n_c$, it follows from Lemma~\ref{lem43} that
\begin{eqnarray*}
{\rm E} [D_{h}^{4}] 
&=& 
{\rm E} \left[ \left\{ \frac{{n_d}}{r_p} {\rm tr}\left( 
\bm{T}_{a}(n_a) \bm{T}_{b}(n_b) \left(\bm{z}_{h-n_a-n_b}\bm{z}_{h-n_a-n_b}'-\bm{\Sigma}_{c}\right) 
\bm{\Sigma}_{d} \right) \right\}^4 \right] \\
&=& \frac{n_{d}^{4}}{r_p^4} 
{\rm E} \left[ \left\{ \bm{z}_{h-n_a-n_b}' \bm{\Sigma}_{c}^{-\frac{1}{2}} \bm{\Sigma}_{c}^{\frac{1}{2}} \bm{\Sigma}_{d} 
\bm{T}_{a}(n_a) \bm{T}_{b}(n_b) \bm{\Sigma}_{c}^{\frac{1}{2}} \bm{\Sigma}_{c}^{-\frac{1}{2}} \bm{z}_{h-n_a-n_b} 
\right.\right. \\ && \left.\left.
-{\rm tr}\left( \bm{\Sigma}_{c}^{\frac{1}{2}} \bm{\Sigma}_{d} 
\bm{T}_{a}(n_a) \bm{T}_{b}(n_b) \bm{\Sigma}_{c}^{\frac{1}{2}} \right) \right\}^4 \right] \\
&=& \frac{n_{d}^{4}}{r_p^4} 
{\rm E} \left[ 3{\rm tr}\left( \left( \bm{C}+\bm{C}' \right)^4 \right) 
+ \frac{3}{4}\left\{ {\rm tr}\left( \left( \bm{C}+\bm{C}' \right)^2 \right) \right\}^2 \right],
\end{eqnarray*}
where $\bm{C}=\bm{\Sigma}_{c}^{\frac{1}{2}} \bm{\Sigma}_{d} 
\bm{T}_{a}(n_a) \bm{T}_{b}(n_b) \bm{\Sigma}_{c}^{\frac{1}{2}}$, so that
\begin{eqnarray*}
{\rm E} [D_{h}^{4}] 
&=& \frac{n_{d}^{4}}{r_p^4} 
{\rm E} \left[ 3{\rm tr}\left( \left( \bm{C}+\bm{C}' \right)^4 \right) 
+ \frac{3}{4}\left\{ {\rm tr}\left( \left( \bm{C}+\bm{C}' \right)^2 \right) \right\}^2 \right] \\
&<& \frac{n_{d}^{4}}{r_p^4} 
{\rm E} \left[ 3\left\{ {\rm tr}\left( \left( \bm{C}+\bm{C}' \right)^2 \right) \right\}^2
+ \frac{3}{4}\left\{ {\rm tr}\left( \left( \bm{C}+\bm{C}' \right)^2 \right) \right\}^2 \right] \\
&=& 
\frac{15}{4} \frac{n_{d}^{4}}{r_p^4} 
{\rm E} \left[ \left\{ {\rm tr}\left( \left( \bm{C}+\bm{C}' \right)^2 \right) \right\}^2 \right] \\
&=& 
\frac{15}{4} \frac{n_{d}^{4}}{r_p^4} 
{\rm E} \left[ \left\{ {\rm tr}\left( \left( \bm{\Sigma}_{c}^{\frac{1}{2}} \bm{\Sigma}_{d} 
\bm{T}_{a}(n_a) \bm{T}_{b}(n_b) \bm{\Sigma}_{c}^{\frac{1}{2}}
+\bm{\Sigma}_{c}^{\frac{1}{2}} \bm{T}_{b}(n_b) 
\bm{T}_{a}(n_a) \bm{\Sigma}_{d} \bm{\Sigma}_{c}^{\frac{1}{2}} \right)^2 \right) \right\}^2 \right].
\end{eqnarray*}
Hence, by using Proposition~\ref{prop48}, 
\begin{eqnarray*}
\lefteqn{ \sum_{h=n_a+n_b+1}^{n_a+n_b+n_c} {\rm E} [D_{h}^{4}] } \\
&<& \frac{15}{4} \frac{n_c n_{d}^{4}}{r_p^4} 
{\rm E} \left[ \left\{ {\rm tr}\left( \left( \bm{\Sigma}_{c}^{\frac{1}{2}} \bm{\Sigma}_{d} 
\bm{T}_{a}(n_a) \bm{T}_{b}(n_b) \bm{\Sigma}_{c}^{\frac{1}{2}}
+\bm{\Sigma}_{c}^{\frac{1}{2}} \bm{T}_{b}(n_b) 
\bm{T}_{a}(n_a) \bm{\Sigma}_{d} \bm{\Sigma}_{c}^{\frac{1}{2}} \right)^2 \right) \right\}^2 \right] \\
&\sim & \frac{15}{4} \frac{n_c n_{d}^{4}}{r_p^4} 
{\rm E} \left[ \left\{ {\rm tr}\left( 
\bm{T}_{a}(n_a) \bm{\Sigma}_{d} \bm{\Sigma}_{c} \bm{\Sigma}_{d} 
\bm{T}_{a}(n_a) \bm{T}_{b}(n_b) \bm{\Sigma}_{c} \bm{T}_{b}(n_b) 
 \right) \right\}^2 \right] \\
&=& \frac{15}{4} \frac{n_c n_{d}^{4}}{r_p^4} 
{\rm E} \left[ {\rm tr}\left( 
\bm{T}_{a}(n_a) \bm{\Sigma}_{d} \bm{\Sigma}_{c} \bm{\Sigma}_{d} 
\bm{T}_{a}(n_a) \bm{T}_{b}(n_b) \bm{\Sigma}_{c} \bm{T}_{b}(n_b) 
 \right) 
\right. \\ && \left. \times
 {\rm tr}\left( 
\bm{T}_{a}(n_a) \bm{\Sigma}_{d} \bm{\Sigma}_{c} \bm{\Sigma}_{d} 
\bm{T}_{a}(n_a) \bm{T}_{b}(n_b) \bm{\Sigma}_{c} \bm{T}_{b}(n_b) 
 \right) \right] \\
&=& 
\frac{15}{4}\frac{n_{d}^{2}}{p^8n_{a}^{2}n_{b}^{2}n_{c}} O(n_{a}^{2}n_{b}^{2}p^6) \\
&\to & 0.
\end{eqnarray*}

\item \textit{Proof of \eqref{t534}.}
For $h=n_a+n_b+n_c+1,\dots,n_a+n_b+n_c+n_d$, it follows from Lemma~\ref{lem43} that
\begin{eqnarray*}
{\rm E} [D_{h}^{4}] 
&=& 
{\rm E} \left[ \left\{ \frac{1}{r_p} {\rm tr}\left( 
\bm{T}_{a}(n_a) \bm{T}_{b}(n_b) \bm{T}_{c}(n_c) 
\left(\bm{w}_{h-n_a-n_b-n_c}\bm{w}_{h-n_a-n_b-n_c}'-\bm{\Sigma}_{d}\right) \right) \right\}^4 \right] \\
&=& \frac{1}{r_p^4} 
{\rm E} \left[ \left\{ 
\bm{w}_{h-n_a-n_b-n_c}'\bm{\Sigma}_{d}^{-\frac{1}{2}}\bm{\Sigma}_{d}^{\frac{1}{2}}\bm{T}_{a}(n_a) \bm{T}_{b}(n_b) 
\bm{T}_{c}(n_c) \bm{\Sigma}_{d}^{\frac{1}{2}}\bm{\Sigma}_{d}^{-\frac{1}{2}}\bm{w}_{h-n_a-n_b-n_c} 
\right. \right. \\
&& \left. \left. 
- {\rm tr}\left( 
\bm{\Sigma}_{d}^{\frac{1}{2}} \bm{T}_{a}(n_a) \bm{T}_{b}(n_b) \bm{T}_{c}(n_c) \bm{\Sigma}_{d}^{\frac{1}{2}} \right) 
\right\}^4 \right] \\
&=& \frac{1}{r_p^4} 
{\rm E} \left[ 3{\rm tr}\left( \left( \bm{D}+\bm{D}' \right)^4 \right) 
+ \frac{3}{4}\left\{ {\rm tr}\left( \left( \bm{D}+\bm{D}' \right)^2 \right) \right\}^2 \right],
\end{eqnarray*}
where $\bm{D}=\bm{\Sigma}_{d}^{\frac{1}{2}} \bm{T}_{a}(n_a) \bm{T}_{b}(n_b) \bm{T}_{c}(n_c) \bm{\Sigma}_{d}^{\frac{1}{2}}$, so that
\begin{eqnarray*}
{\rm E} [D_{h}^{4}] 
&=& \frac{1}{r_p^4} 
{\rm E} \left[ 3{\rm tr}\left( \left( \bm{D}+\bm{D}' \right)^4 \right) 
+ \frac{3}{4}\left\{ {\rm tr}\left( \left( \bm{D}+\bm{D}' \right)^2 \right) \right\}^2 \right] \\
&<& 
\frac{1}{r_p^4} 
{\rm E} \left[ 3 \left( {\rm tr}\left( \left( \bm{D}+\bm{D}' \right)^2 \right) \right)^2 
+ \frac{3}{4}\left\{ {\rm tr}\left( \left( \bm{D}+\bm{D}' \right)^2 \right) \right\}^2 \right] \\
&=& 
\frac{15}{4}\frac{1}{r_p^4} 
{\rm E} \left[ \left\{ {\rm tr}\left( \left( \bm{D}+\bm{D}' \right)^2 \right) \right\}^2 \right] \\
&=& 
\frac{15}{4}\frac{1}{r_p^4} 
{\rm E} \left[ \left\{ {\rm tr}\left( \left( \bm{\Sigma}_{d}^{\frac{1}{2}} \bm{T}_{a}(n_a) \bm{T}_{b}(n_b) \bm{T}_{c}(n_c) 
\bm{\Sigma}_{d}^{\frac{1}{2}}
+\bm{\Sigma}_{d}^{\frac{1}{2}} \bm{T}_{c}(n_c) \bm{T}_{b}(n_b) \bm{T}_{a}(n_a) 
\bm{\Sigma}_{d}^{\frac{1}{2}} \right) \right) \right\}^2 \right].
\end{eqnarray*}
Hence, 
\begin{eqnarray*}
\lefteqn{
\sum_{h=n_a+n_b+n_c+1}^{n_a+n_b+n_c+n_d} {\rm E} [D_{h}^{4}] } \\
&<& 
\frac{15}{4} \frac{n_d}{r_p^4} 
{\rm E} \left[ \left\{ {\rm tr}\left( \left( \bm{\Sigma}_{d}^{\frac{1}{2}} \bm{T}_{a}(n_a) \bm{T}_{b}(n_b) \bm{T}_{c}(n_c) 
\bm{\Sigma}_{d}^{\frac{1}{2}}
+\bm{\Sigma}_{d}^{\frac{1}{2}} \bm{T}_{c}(n_c) \bm{T}_{b}(n_b) \bm{T}_{a}(n_a) 
\bm{\Sigma}_{d}^{\frac{1}{2}} \right)^2 \right) \right\}^2 \right] \\
&\sim & 
\frac{15}{4} \frac{n_d}{r_p^4} 
{\rm E} \left[ \left\{ {\rm tr}\left( \bm{T}_{a}(n_a) 
\bm{\Sigma}_{d} \bm{T}_{a}(n_a) \bm{T}_{b}(n_b) \bm{T}_{c}(n_c) 
\bm{\Sigma}_{d} \bm{T}_{c}(n_c) \bm{T}_{b}(n_b) \right) \right\}^2 \right] \\
&=& 
\frac{15}{4}\frac{n_d}{r_p^4} 
{\rm E} \left[ {\rm tr}\left( \bm{T}_{a}(n_a) 
\bm{\Sigma}_{d} \bm{T}_{a}(n_a) \bm{T}_{b}(n_b) \bm{T}_{c}(n_c) 
\bm{\Sigma}_{d} \bm{T}_{c}(n_c) \bm{T}_{b}(n_b) \right) \right. \\
&& \left. \times {\rm tr}\left( \bm{T}_{a}(n_a) 
\bm{\Sigma}_{d} \bm{T}_{a}(n_a) \bm{T}_{b}(n_b) \bm{T}_{c}(n_c) 
\bm{\Sigma}_{d} \bm{T}_{c}(n_c) \bm{T}_{b}(n_b) \right) \right] \\
&=& 
\frac{15}{4}\frac{1}{p^8n_{a}^{2}n_{b}^{2}n_{c}^{2}n_{d}} 
\left\{ O(n_{a}^{4}n_{b}^{4}n_{c}^{4}p^2) + O(n_{a}^{2}n_{b}^{4}n_{c}^{4}p^4) + 
O(n_{a}^{4}n_{b}^{2}n_{c}^{4}p^4) + O(n_{a}^{4}n_{b}^{4}n_{c}^{2}p^4) 
\right. \\ && \left. 
+ O(n_{a}^{2}n_{b}^{2}n_{c}^{4}p^6) + 
O(n_{a}^{2}n_{b}^{4}n_{c}^{2}p^6) + O(n_{a}^{4}n_{b}^{2}n_{c}^{2}p^6) + 
O(n_{a}^{2}n_{b}^{2}n_{c}^{2}p^8) \right\} \\
&=& 
\frac{15}{4}\frac{1}{p^8n_{a}^{2}n_{b}^{2}n_{c}^{2}n_{d}} O(n_{a}^{2}n_{b}^{2}n_{c}^{2}p^8)
\\
& \to & 0.
\end{eqnarray*}
\end{itemize}

This completes the proof of Lemma~\ref{lem24}.
\qed

\section*{Acknowledgments}
The first author was supported in part by Japan Society for the Promotion of Science KAKENHI Grant Number 18K13454.
This study was partly carried out when the first author (KT) was a member of Graduate School of Arts and Sciences, the University of Tokyo.

\end{document}